\documentclass[a4paper, 11pt]{article}

\input xy 
\xyoption{all}

\usepackage{amsmath,amsthm,amssymb}
\usepackage[page,header]{appendix}
\usepackage{morefloats}
\usepackage{color}
\usepackage{makeidx}
\usepackage{hyperref}	 	
\usepackage{fancybox}
\usepackage{mathabx}
\usepackage{tikz}
\usetikzlibrary{matrix}

\makeindex 

\usepackage{latexsym, amsmath, amssymb, amsfonts, amscd}
\usepackage{amsthm}
\usepackage{t1enc}
\usepackage[mathscr]{eucal}
\usepackage{indentfirst}
\usepackage{graphicx, pb-diagram}
\usepackage{enumerate}
\usepackage[all,poly,web,knot]{xy}

\newcommand{\Title}{Title}

\numberwithin{equation}{section}

\theoremstyle{definition}\newtheorem{definition}{Definition}[section]

\newtheorem{defititle}[definition]{\Title}
\newtheorem{notation}[definition]{Notation}

\newtheorem{remark}[definition]{Remark}

\newtheorem{ex}[definition]{Example}
\newtheorem{exs}[definition]{Examples}
\newtheorem{prop}[definition]{Proposition}
\newtheorem{proposition-definition}[definition]{Proposition-Definition}
\newtheorem{lemma}[definition]{Lemma}
\newtheorem{thm}[definition]{Theorem}
\newtheorem{cor}[definition]{Corollary}

\newtheorem*{prop*}{Proposition}
\newtheorem*{theorem*}{Theorem}
\newtheorem*{ex*} {Example}
\newtheorem*{remark*} {Remark}
\newtheorem{thmx}{Theorem}

\newcommand{\cB}{\mathcal{B}}
\newcommand{\cG}{{G}} 

\newcommand{\cF}{\mathcal{F}}

\newcommand{\cW}{\mathcal{W}}
\newcommand{\cV}{\mathcal{V}}
\newcommand{\cK}{K} 

\newcommand{\cI}{\mathcal{I}}

\newcommand{\cU}{\mathcal{U}}

\newcommand{\id}{{\hbox{id}}}
\newcommand{\ie}{{\it i.e.}\;}

\newcommand{\cf}{{\it cf.}\/ }

\newcommand{\vX}{\mathfrak{X}}

\def\gpd{\,\lower1pt\hbox{$\longrightarrow$}\hskip-.24in\raise2pt
             \hbox{$\longrightarrow$}\,}

\renewcommand{\latticebody}{\drop@{ }}

\newcommand{\R}{\ensuremath{\mathbb R}}

\newcommand{\g}{\ensuremath{\mathfrak{g}}}

\newcommand{\cA}{\mathcal{A}}
\newcommand{\cS}{\mathcal{S}}

\newcommand{\RR}{\ensuremath{\mathbb R}}

\newcommand{\bb}{\mathbf{b}}
\newcommand{\bc}{\mathbf{c}}
\newcommand{\bt}{\mathbf{t}}                  
\newcommand{\bs}{\mathbf{s}}                  

\newcommand{\hd}[1]{\Omega^{1/2}({#1})}
\newcommand{\shd}[1]{\Gamma_c  ( \Omega^{1/2} ({#1}) )}

\def\act{\mathbin{\hbox{$<\kern-.4em\mapstochar\kern.4em$}}}
\def\ract{\mathbin{\hbox{$\mapstochar\kern-.3em>$}}}
\def\exp{\mathrm{exp}}
\def\balpha{\boldsymbol{\alpha}}

\def\bgamma{\boldsymbol{\gamma}}
\def\rar{\overrightarrow}
\def\lar{\overleftarrow}

\def\PB(#1,#2,#3,#4){\left\{\begin{matrix}#1&\!\!\!\stackrel{?}{\longrightarrow}&\!\!\!#2\\
\downarrow&&\!\!\!\downarrow\\
#3&\!\!\!\stackrel{?}{\longrightarrow}&\!\!\!#4\end{matrix}\right\}}

\def\pb(#1,#2,#3,#4){ \hom(#1 \to #3, #2 \to #4)}

\begin{document}

\begin{center}
{\Large{\bf Singular subalgebroids}}


{\large{\sc by  Marco Zambon}}\\
{\sc with an appendix by Iakovos Androulidakis}

\end{center}

{\footnotesize
  
\vskip 2pt KU Leuven
\vskip-4pt  Department of Mathematics
\vskip-4pt Celestijnenlaan 200B box 2400
\vskip-4pt BE-3001 Leuven, Belgium.
\vskip-4pt e-mail: \texttt{marco.zambon@kuleuven.be}

National and Kapodistrian University of Athens
\vskip -4pt Department of Mathematics
\vskip -4pt Panepistimiopolis
\vskip -4pt GR-15784 Athens, Greece
\vskip -4pt e-mail: \texttt{iandroul@math.uoa.gr}
}
\bigskip
\everymath={\displaystyle}

\date{today}

\begin{abstract}\noindent 
We introduce singular subalgebroids of an integrable Lie algebroid,
extending the notion of Lie subalgebroid by dropping the constant rank requirement. 
We lay the bases of a Lie theory for singular subalgebroids: we  construct   the associated holonomy groupoids, adapting the procedure of Androulidakis-Skandalis for singular foliations, in a way that keeps track of the choice of Lie groupoid integrating the ambient Lie algebroid. In the regular case, this recovers the integration of Lie subalgebroids by Moerdijk-Mr{\v{c}}un. The holonomy groupoids are topological groupoids,
and are  
 suitable for noncommutative geometry as they allow  for the construction of the associated convolution algebras.
Further we carry out the construction for morphisms in a functorial way.
\end{abstract}
 
\setcounter{tocdepth}{2} 
\tableofcontents

\section*{Introduction}
\addcontentsline{toc}{section}{Introduction}

 {Lie algebroids arise in differential geometry, mathematical physics and control theory. The standard viewpoint is to declare their sub-objects to be (wide) Lie subalgebroids, i.e. involutive constant-rank subbundles. However  there is a multitude of interesting
``singular'' examples that violate the constant-rank requirement.  
This leads us to introduce here a new class of subalgebroids, quite more singular than 
the usual Lie subalgebroids: we call them \emph{singular subalgebroids}. 
}
 
 {Our aim is to build a Lie theory for singular subalgebroids. In this paper we}
 construct a topological groupoid  canonically associated to them, called \emph{holonomy groupoid}, which depends on a choice of integration $G$ of the ambient Lie algebroid. 
The construction 
parallels the one of \cite{AndrSk}, and here too the holonomy groupoid is a topological groupoid.
This construction  encompasses the integration of wide Lie subalgebroids by Moerdijk-Mr{\v{c}}un \cite{MMRC} and the holonomy groupoids of singular foliations of Androulidakis-Skandalis \cite{AndrSk}. A novel feature is the presence of many interesting morphisms.  We prove a version Lie's second theorem in this context (integration of morphisms),
showing that our holonomy groupoid construction is functorial.

 {Building on the present work,} in a follow-up paper with Androulidakis  \cite{AZ4} we 
 provide a version of Lie's third theorem, making precise how one can view the holonomy groupoid as an ``integration'' of the singular subalgebroid. {This 
 requires us to work in the realm of diffeological groupoids.}
In that paper we 
also show that although the holonomy groupoid is not smooth,  it is still possible to do differential geometry on it.

{Finally here, using the holonomy groupoid we attach a $C^{\ast}$-algebra to a singular subalgebroid, paving the way to the development of pseudodifferential calculus, index theory, and other noncommutative geometry constructions for such structures.}
 
{Recently Laurent Gengoux-Lavau-Strobl  \cite{LavauThesis}\cite{LLG} showed that
singular foliations are tightly connected with higher algebraic structures:
under reasonable assumptions, a singular foliation admits a canonical $L_{\infty}$-algebroid which ``resolves'' it and which provides fine invariants. We expect their construction to extend to singular subalgebroids.} 
  
\subsection*{Singular subalgebroids}

Fix a Lie algebroid $A$ over a manifold $M$.
A {\bf singular subalgebroid}
is a $C^{\infty}(M)$-submodule $\cB$ of   $\Gamma_c (A)$ (the module of compactly supported sections of $A$), which is locally finitely generated and closed w.r.t. the Lie bracket. 

Let us display two obvious classes of singular  subalgebroids, whose intersection consists exactly of the regular foliations.
\begin{ex*}[Wide Lie subalgebroids] Let $B$ be a wide Lie subalgebroid of $A$, {\ie a Lie subalgebroid supported on the whole of $M$.} Then
$\Gamma_c(B)$ is a singular subalgebroid.
\end{ex*}
\begin{ex*}[Singular foliations] 
The singular subalgebroids of $A=TM$ are exactly the singular foliations on $M$.  {Here singular foliation is meant in the sense of \cite{AndrSk}, a notion inspired by the work of 
 Stefan and Sussman in the 1970's.}
\end{ex*}

There is an interesting class that strictly contains  the first one: the singular subalgebroids $\cB$ which are \emph{projective}, i.e., so that there exists a vector bundle over $M$ whose module of  compactly supported  sections is $\cB$. Notice that such a vector bundle is then a Lie algebroid (but not necessarily a Lie subalgebroid of $A$). 

In turn, projective subalgebroids are contained in a larger class, that of singular subalgebroids which are images of Lie groupoid morphisms covering the identity. 
Other examples of singular subalgebroids will be given in \S \ref{subsec:ex}.

{Singular subalgebroids can also be viewed as a nice class of Lie-Rinehart algebras \cite{Rinehart}, more general than Lie algebroids.}

\subsection*{Main results}

 For singular foliations $\cF$ on $M$, which as we saw are exactly the singular subalgebroids of $TM$, the holonomy groupoid was constructed by Androulidakis-Skandalis \cite{AndrSk}.
There the crucial idea was that of a bisubmersion.
Bisubmersions are   manifolds $U$ endowed with two submersive maps to $M$, and are defined locally from the data provided by the singular foliation. Their dimension is variable, and the holonomy groupoid $H(\cF)$ is a quotient of a disjoint union of bisubmersions.

For singular {subalgebroids $\cB$} of an integrable Lie algebroid $A$, after choosing an integrating Lie groupoid $G$, {taking a new point of view} we extend the notion of  \cite{AndrSk} by defining 
bisubmersions to be smooth maps $U\to G$ satisfying certain conditions.
With this notion we can construct the {\bf holonomy groupoid} $H^G(\cB)$ in a way analogous to  \cite{AndrSk}.
 
A feature of the construction we give here is that it keeps track of the choice of Lie groupoid $\cG$ integrating  $A$. More precisely, the  holonomy groupoid $H^G(\cB)$ comes together   with a canonical morphism to $G$. 
For instance the groupoid $H(\cF)$ given in \cite{AndrSk}, in the current context,  is   the holonomy groupoid associated to $\cF$ (viewed as a singular subalgebroid) when we choose $G$ to be the pair groupoid $M \times M$. The canonical morphism to  
  $G$ is just the target-source map.
  
The main result of the paper is Thm. \ref{thm:holgroidconstr}, which can be paraphrased in a simplified way as follows:

\begin{thmx}\label{thmx:a}
{Let $\cB$ be a singular subalgebroid of an integrable Lie algebroid $A$, and $G$ a Lie groupoid  integrating $A$.}
There exists a canonical map $$\Phi\colon H^{\cG}(\cB) \to \cG$$ where
\begin{itemize}
\item[1)] $H^{\cG}(\cB)$ is a  topological groupoid which is ``nice'' 
  and ``integrates $\cB$'',
\item[2)] $\Phi$  is a topological groupoid morphism  ``integrating'' the inclusion $\iota\colon \cB \hookrightarrow \Gamma_c(A)$.
\end{itemize}
\end{thmx}

In  joint work with Androulidakis   \cite{AZ4} 
\begin{itemize}
\item we show that $H^{\cG}(\cB)$ has some smoothness properties: first, {it is leafwise smooth in the sense that} its restriction to the leaves of $\cB$ are Lie groupoids {\cite[\S 2]{AZ4}}, and second, it has a structure of diffeological groupoid that allows to recover $\cB$ {\cite[\S 5.3, \S6.4]{AZ4}}. This is what we mean by 
``nice''  
  and ``integrates $\cB$'' in 1) above.
  \item we consider certain diffeological groupoids endowed with maps to Lie groupoids, and
  show that a morphism of such objects always induces a morphism of singular subalgebroids {\cite[\S 5.4]{AZ4}}. Together with Ex. \ref{ex:Phi}, this explains ``integrating'' in 2) above.
{Further, for every leaf $L$ of $\mathcal{B}$, the  map $H^G(\mathcal{B})|_L\to G$ obtained restricting $\Phi$
is a Lie groupoid morphism integrating the Lie algebroid morphism $\mathcal{B}_L\to A$ induced by the inclusion $\iota$, see \cite[\S 2.5]{AZ4}. Here 
$\mathcal{B}_L$ is a transitive Lie algebroid over $L$; when $L$ is an embedded leaf, its sections are $\mathcal{B}/I_L\mathcal{B}$ for $I_L$ the ideal of functions on $M$ vanishing on $L$.}
   \end{itemize}

Further, for singular subalgebroids which are images of Lie groupoid morphisms (this includes singular foliations), we give a  minimality property for $H^{\cG}(\cB)$ in Prop. \ref{prop:lift}. 
This minimality property was postulated first by Moerdijk-Mr{\v{c}}un  for wide Lie subalgebroids and is satisfied by the holonomy groupoid of a singular foliation.

{
The construction of the holonomy groupoid $H^{\cG}(\cB)$ allows to transfer almost verbatim the   construction of the convolution $*$-algebra of \cite{AndrSk}, see Appendix \ref{section:convsing}  by   Iakovos Androulidakis. Recall \cite{AndrSk,IakAnal,PseudodiffCalcSingFol}  that 
in the case of singular foliations
the $K$-theory of the corresponding $C^{\ast}$-algebra is the recipient of the analytic index for longitudinal elliptic pseudodifferential operators.}

{A feature of singular subalgebroids compared to singular foliations is that morphisms abound. 
In Thm. \ref{thm:morph} we show that the  holonomy groupoid construction extends to morphisms covering the identity. (For more general morphisms we refer to Appendix \ref{sec:appmorsub}.) This provides new statements even for singular foliations.
\begin{thmx}\label{thmx:B}
Let $F\colon \cG_1\to \cG_2$ be a morphism of Lie groupoids covering $Id_M$. Let $\cB_i$ be a singular subalgebroid of $Lie(\cG_i)$ for $i=1,2$,
 such that
$F_*(\cB_1)\subset \cB_2$.
Then there is a canonical morphism of topological groupoids $$\Xi \colon H^{\cG_1}(\cB_1)\to H^{\cG_2}(\cB_2)$$ covering $Id_M$ and 
making the following diagram commute:
\begin{equation*} 
 \xymatrix{
H^{\cG_1}( {\cB}_1)  \ar[d]^{\Phi_1}   \ar@{-->}[r]^{\Xi} &H^{\cG_2}(\cB_2)  \ar[d]^{\Phi_2}     \\
\cG_1 \ar[r]^{F} &   \cG_2 }
\end{equation*}.
\end{thmx}
}

 \subsection*{The   Moerdijk-Mr{\v{c}}un integration of wide Lie subalgebroids}
Theorem \ref{thmx:a} extends and unifies previous results by Moerdijk-Mr{\v{c}}un  \cite{MMRC} and Androulidakis-Skandalis \cite{AndrSk}
for the two   obvious classes of singular  subalgebroids displayed above -- wide Lie subalgebroids and singular foliations --. We now elaborate on the first class, as
it is instructive to compare Theorem \ref{thmx:a} with the results of    Moerdijk-Mr{\v{c}}un in \cite{MMRC}. 

Let $A\to M$ be a Lie algebroid, and  fix   a Lie groupoid $\cG$ integrating $A$. Let $B\to M$ be a wide Lie subalgebroid of $A$. Moerdijk-Mr{\v{c}}un   show:
\begin{theorem*}{\bf (\cite[Thm. 2.3]{MMRC})}
There exists a unique map $$\Phi\colon H_{min}\to \cG$$ where
\begin{itemize}
\item[1)] $H_{min}$ is a  Lie groupoid  integrating $B$,
\item[2)] $\Phi$  is a Lie groupoid morphism  integrating the inclusion $\iota\colon B\hookrightarrow A$,
\item[3)] minimality property: for any Lie groupoid morphism 
$\tilde{H}\to G$ integrating\footnote{So 
 $\tilde{H}$ is necessarily a Lie groupoid integrating $B$, and 
 the morphism is an immersion.}  $\iota$, there exists a surjective Lie groupoid morphism $\tilde{H}\to H_{\min}$ integrating $Id_B$ and making this diagram commute:
\begin{equation*}
\xymatrix{
\tilde{H}  \ar[rd]_{}    \ar@{-->}[rr] & & H_{min}   \ar[ld]^{\Phi}  \\
&\cG & }
\end{equation*}
\end{itemize}
\end{theorem*}

Moerdijk-Mr{\v{c}}un refer to $H_{min}$ as the 
 \emph{minimal integral of $B$ over $\cG$}. By 3) above, $H_{min}$ is unique up to isomorphism. The construction of Thm. \ref{thmx:a}, applied to $\cB:=\Gamma_c(B)$, yields exactly the minimal integral $H_{min}$ together with the above Lie groupoid morphism $\Phi$, see Prop. \ref{prop:HGBisHmin}.
    
 To put this result into context, recall that the wide Lie algebroid $B$ is integrable (because $A$ is),  and that the inclusion $\iota$ integrates to a  morphism $H_{max}\to \cG$, where $H_{max}$ is the source simply connected Lie groupoid   integrating $B$. All other Lie groupoids integrating $B$ are quotients of $H_{max}$. In general they do not admit a morphism to $\cG$ integrating $\iota$, {and $H_{min}$  is the ``smallest'' integration admitting such a morphism.} \emph{We want to stress   that the result of Moerdijk-Mr{\v{c}}un, just as our Theorem \ref{thmx:a}, does  not  contain as a special case the integration of Lie algebroids (indeed, an integration $G$ of the Lie algebroid $A$ is part of the hypotheses).}
 
 As an example of the above theorem, take the case $A=TM$. Then a wide Lie subalgebroid is just an involutive distribution, which by the Frobenius theorem corresponds of a (regular) foliation on $M$. The Lie groupoids $H_{max}$ and $H_{min}$ are nothing else than the monodromy and holonomy groupoids of this foliation.
 
Notice that the above theorem of Moerdijk-Mr{\v{c}}un  starts with a Lie subalgebroid $B$ of $A$ (rather than with an abstract Lie algebroid $B$), and that it produces a Lie groupoid morphism to $\cG$ integrating the inclusion  $\iota\colon B\hookrightarrow A$ (rather than only a Lie groupoid integrating $B$). In other words, in the above theorem     $H_{min}$ is {naturally} endowed with a \emph{morphism} {$\Phi \colon H_{min} \to  \cG$}. {Notice that it would not be wise to disregard this morphism and consider only its image $\Phi(H_{min})$. First, the latter is a set-theoretic subgroupoid of $\cG$, which usually fails to be smooth. (When $B$ is an involutive distribution on $M$, $\Phi(H_{min})$ is the graph of the equivalence relation given by the associated regular foliation, and its failure to be smooth was one of the reasons to introduce the holonomy groupoid in the first place, see the remarks in  \cite{Phillipsholonomicimperative}). Second, the morphism of Lie groupoids $\Phi$
is usually not injective and hence contains more information than its image.}
 
\subsection*{{Androulidakis-Skandalis' holonomy groupoids of singular foliations}}

 We now highlight the aspects of this work that represent the main novelties in  comparison   with the work of  Androulidakis-Skandalis \cite{AndrSk}. (For singular foliations, the construction of Thm. \ref{thmx:a}  yields the holonomy groupoid of Androulidakis-Skandalis.) Let $G$ be a Lie groupoid and  $\cB$ be a singular subalgebroid of $A:=Lie(G)$.
 \begin{itemize}
 \item  Our definition of bisubmersion for $\cB$ (Def. \ref{subsec:defbisub}) 
 is not a straight-forward generalization of the one of \cite{AndrSk}.
Ours is given by a smooth map to $G$, which typically fails to be a  submersion. In the case that $\cB$ is a singular foliation, our definition does not recover on the nose the notion of bisubmersion from \cite{AndrSk}, but it corresponds  bijectively to it if one assembles two submersions into $M$ to one map to the pair groupoid $M\times M$ 
 (Prop. \ref{prop:equivbi}). 
Just as in \cite{AndrSk}, bisubmersions for $\cB$ have the following features:
their bisections induce automorphisms of the  singular subalgebroid $\cB$ (see Remark \ref{rem:rephrasea}), and they
  allow for the construction of the holonomy groupoid by providing its ``building blocks'' (Thm. \ref{thm:holgroidconstr}).
  
\item 
The holonomy groupoid depends of the choice of Lie groupoid $G$ integrating $A$.
 In \S \ref{section:vary} we display how the holonomy groupoid changes if we replace $G$ by another Lie groupoid  of which $G$ is a quotient.
For singular foliations, \ie when $A=TM$, there is a canonical choice for $G$, namely the pair groupoid $M\times M$. With this choice we recover the holonomy groupoid of a singular foliation of \cite{AndrSk}.

\item Morphisms between  Lie algebroids abound, even in the special case of morphisms covering the identity on the base (take for instance the anchor map). We show
that when such a morphism maps a singular subalgebroid into another, there is a canonically induced morphisms between the corresponding holonomy groupoids (Theorem 
\ref{thmx:B}), and that this assignment is functorial. When one restricts to singular foliations, 
there are not as many morphisms, and the natural ones are given by smooth maps between manifolds with singular foliations. Their effect at the level of holonomy groupoids  is 
not considered in \cite{AndrSk} and plays an important role  in \cite{GZQuotiens}.\end{itemize}

\noindent{\bf Conventions:} All Lie groupoids are {assumed to be source connected}, not necessarily Hausdorff, but with Hausdorff source-fibers. Given a Lie groupoid $\cG\rightrightarrows M$, we denote by $\bt$ and $\bs$ its target and source maps, and by $i\colon \cG\to \cG$ the inversion map. We denote by $1_x\in \cG$ the identity
element corresponding to a point $x\in M$, and by $1_M\subset \cG$ the submanifold of identity elements.
Two elements $g,h\in \cG$ are composable if $\bs(g)=\bt(h)$. We identify the Lie algebroid of $\cG$, which we sometimes denote by $Lie(\cG)$, with $\ker (d\bs)|_M$. 

{We use the term ``generators'' only in the context of modules over $C^{\infty}(U)$, for $U$ a manifold, and to mean that they generate as a $C^{\infty}(U)$-module.}

\noindent{\bf Acknowledgements:} 
{M.Z. thanks Iakovos Androulidakis -- who collaborated in this project in its early stages and is the author of Appendix \ref{section:convsing} -- for fruitful discussions and constructive advice, and Ivan Struchiner for inspiring comments on this work. We thank the referee for providing constructive remarks.}
This work was partially supported
by grants  
MTM2011-22612 and ICMAT Severo Ochoa  SEV-2011-0087 (Spain), Pesquisador Visitante Especial grant  88881.030367/2013-01 (CAPES/Brazil),  IAP Dygest,
the long term structural funding -- Methusalem grant of the Flemish Government,
the FNRS-FWO under EOS project G0H4518N, the FWO research project G083118N (Belgium).

\section{Singular subalgebroids}\label{section:SingSubalgd}

In this section we introduce the notion of singular subalgebroid, give several examples,
{and in \S \ref{sec:rivf} make some observations for later use.} Throughout this section we are going to consider a  Lie algebroid $A\to M$ with anchor $\rho\colon A \to TM$ (see for instance \cite{CW}\cite{MK2}\cite{LecturesIntegrabilty}).  
  
\subsection{Definition of singular subalgebroid}
\label{section:singdef}

We define the main object of interest of this paper:

\begin{definition}\label{dfn:singsubalg}
A {\bf singular subalgebroid} of $A$ is an involutive, locally finitely generated $C^{\infty}(M)$-submodule $\cB$ of ${\Gamma_c(A)}$.
\end{definition}

The notion of singular subalgebroid is obtained from the notion of wide Lie subalgebroid, by dropping the requirement of being a (constant rank) subbundle of $A$. This is achieved by focusing on the $C^{\infty}(M)$-module ${\Gamma_c(A)}$ of compactly supported sections of $A$, rather than on the Lie algebroid $A$ itself.

{For any   $C^{\infty}(M)$-submodule $\cB$ of $\Gamma_c(A)$, we define
its {\bf global hull} \cite[\S 1.1]{AndrSk}\cite{AZ4}  to be $$\widehat{\cB}:=\{Z\in \Gamma(A): fZ\in \cB \text{ for all } f\in C_c^{\infty}(M)\}.$$ It is a $C^{\infty}(M)$-submodule of $\Gamma(A)$ containing $\cB$.}

{
A subset $\mathcal{G}$ of $\widehat{\cB}$ is said to be a {\bf generating set for $\cB$} 
if
$\cB=Span_{C^{\infty}_{c}(M)}(\mathcal{G})$, where the latter is the set of finite $C^{\infty}_{c}(M)$-linear combinations of elements of $\mathcal{G}$.}

{Now we can explain the meaning of $\cB$ being ``locally finitely generated'' in Def. \ref{dfn:singsubalg}: it means that for every point of $M$ there is an open neighbourhood $i\colon U \hookrightarrow M$ such that the submodule
$$i^*\cB:=\{Z|_U:Z\in \cB \text{ has support in }U\}$$ 
of $\Gamma_c(A|_U)$ admits a finite generating set. In other words,  there are finitely many
 $Y_1,\dots,Y_n\in\widehat{i^*\cB}$ such that\footnote{Explicitly, $\widehat{i^*\cB}=\{{Y\in \Gamma(A|_U): fY\in i^*\cB \text{ for all } f\in  C^{\infty}_c(U)}\}$.} 
every element of $i^*\cB$ is a $C^{\infty}_c(U)$-linear combination of the $Y_j$'s.}

\subsection{{Motivating examples}}\label{subsec:motivex}

{This notion of singular subalgebroid is motivated by the following two special cases (whose intersection are exactly the regular foliations).

\begin{ex}[{\bf Singular foliations}]\label{ex:singfol} Recall from \cite{AndrSk}  that a {singular foliation} on a manifold $M$ is  an involutive, locally finitely generated submodule of the $C^{\infty}(M)$-module of vector fields {with compact support} $\vX_c(M)$. The singular subalgebroids of $A=TM$ are exactly the singular foliations on $M$.
\end{ex}

\begin{ex}[{\bf Wide Lie subalgebroids}]\label{ex:wideliesub} Recall from \cite[Def. 3.3.21]{MK2} that a wide Lie subalgebroid of $A$ is a vector subbundle $B\to M$, whose sections are closed with respect to the Lie bracket. In this case, $\Gamma_c(B)$ is a singular subalgebroid.
\end{ex}

{{Ex. \ref{ex:wideliesub} belongs} to the larger class of singular subalgebroids  arising from Lie groupoid morphism, which we introduce in \S \ref{subsec:morph}. In this paper we will focus mainly on these 3 classes, which at present are the only ones for which we are able to describe in an explicit way the holonomy groupoid.}
\begin{center}
\hspace{2cm}
\includegraphics[scale=.28]{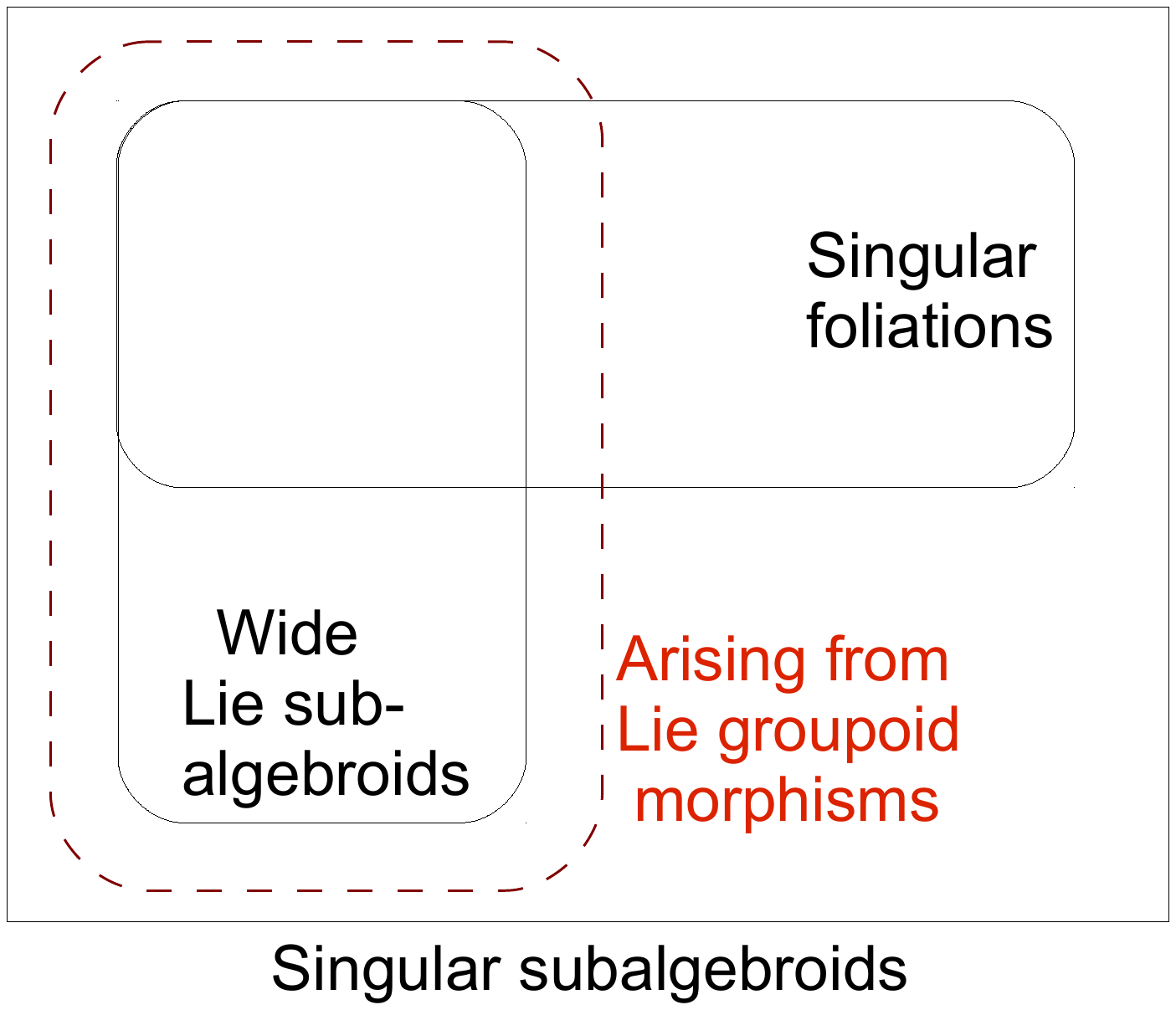}
\end{center} 

\subsection{Further examples}\label{subsec:ex}

Let us now display four geometric contexts in which singular subalgebroids arise.
{Both \S \ref{subsec:morph} and \S \ref{subsec:Liesubalgoids} contain as a special case wide Lie subalgebroids (Example \ref{ex:wideliesub} above).}

\subsubsection{Arising from Lie algebroid morphisms} \label{subsec:morph}

Let {$\psi \colon E\to A$} be a morphism of Lie algebroids covering  the identity on the base manifolds. Then the image of the induced map of compactly supported sections, 
$$\cB:=\psi(\Gamma_c(E)),$$
is a singular subalgebroid  of $A$. We say that $\cB$
{\bf arises from the Lie {algebroid} morphism} $\psi$. }

(The above can be vastly generalized, replacing  $\Gamma_c(E)$ by any singular subalgebroid of $E$, and by allowing $\psi$ to cover a diffeomorphism of the base or even a surjective submersion (see Lemma \ref{lem:FB2}).)
 
 \begin{remark} 
 Given two Lie algebroids $A_1\to M_1$ and $A_2\to M_2$, there is a notion of \emph{comorphism}\footnote{That is,  a pair $(\Phi,
f)$ where $f\colon M_1\to M_2$ is any differentiable map and 
$\Phi\colon f^{!}A_2\to A_1$ is a vector bundle map over $Id_{M_1}$, where $f^{!}A_2$ denotes the pullback of the vector bundle $A_2$ via $f$, such that the induced map of sections $\Gamma(A_2)\to \Gamma(A_1)$ preserves the Lie bracket and the anchor maps satisfy $f_*\circ \rho_{A_1}\circ \Phi=\rho_{A_2}$.}  from $A_1$ to $A_2$
(see \cite[Def. 4.3.16]{MK2}).
{It induces a linear map} $\Gamma(A_2)\to \Gamma(A_1)$. 
{The $C^{\infty}_c(M_1)$-module generated by its image is a  singular subalgebroid of $A_1$}. Example \ref{ex:Poissonmap} below is of this kind, since a Poisson map between Poisson manifolds $M_1\to M_2$ induces a comorphism from $T^*M_1$ to $T^*M_2$. \end{remark}

\begin{exs}\label{ex:proj}
\begin{enumerate}
\item {A singular subalgebroid $\cB$ of $A$   is called {\bf projective} if there exists a vector bundle $B$ over $M$ such that $\Gamma_c(B)\cong \cB$ as $C^{\infty}(M)$-modules. In that case, there is \cite{AZ4} a  Lie algebroid structure on $B$ and almost injective Lie algebroid morphism $\tau \colon B\to A$ inducing the isomorphism $\Gamma_c(B)\cong \cB$, and these data are unique. In particular, $\cB$ arises from the Lie algebroid morphism $\tau$.
}

{A special case occurs when $\cB$ is the space of compactly supported sections of a wide Lie subalgebroid $B$ of $A$.}  
{In that case $\tau \colon B\to A$ is the inclusion.}

\item Given any Lie algebroid $A$, the anchor map $\rho\colon A\to TM$ is a Lie algebroid morphism. In this case $\cB:=\rho(\Gamma_c(A))$ is the singular foliation underlying $A$. {Further, any Lie algebroid morphism (covering the identity) giving rise to $\cB$ must be the anchor map of a Lie algebroid. }
\item  Let $A$ be a Lie algebroid. A  \emph{Nijenhuis operator} {\cite{MagriYvettePN}} is an endomorphism of vector bundles $N \colon A\to A$ over $Id_M$, whose Nijenhuis torsion $T_N(X, Y ) := [NX, NY ] - N[X, Y ]_N$ vanishes. Here  $[X,Y]_N:= [NX, Y ] + [X, NY ] - N[X, Y ]$.
In this case $N$ is a Lie algebroid morphism from $(A,[\cdot,\cdot]_N)$ to $(A,[\cdot,\cdot])$, so $\cB:=N(\Gamma_c(A))$ is a singular subalgebroid  of the original Lie algebroid $(A,[\cdot,\cdot])$.
\end{enumerate}
\end{exs}

{For later use we make the following definition.
\begin{definition}\label{def:arises}
Let $\cB$ be a singular subalgebroid of an integrable Lie algebroid $A$ over $M$.
We say that $\cB$ {\bf arises from a Lie groupoid morphism} (covering the identity)  if there is a Lie groupoid morphism  $\Psi\colon \cK\to\cG$ over $Id_M$,  where $\cG$ is any Lie groupoid integrating $A$, such that   $$\cB=\Psi_*(\Gamma_c(Lie(\cK))).$$
\end{definition}
Examples include compactly supported sections of wide Lie subalgebroids of $A$, for the latter are integrable. 
Clearly {Def. \ref{def:arises}}  implies that $\cB$ arises from a Lie algebroid morphism, namely $\Psi_*\colon Lie(\cK)\to A$. Conversely, if singular subalgebroid arises from a Lie algebroid morphism $\psi \colon E\to A$ with $E$ an integrable Lie algebroid, then this singular subalgebroid arises from a Lie groupoid morphism (namely, 
the Lie groupoid morphism $\Psi\colon \cK\to\cG$  integrating $\psi$, where $K$ is the source simply connected Lie groupoid integrating $E$).
}

\subsubsection{Globally finitely generated singular subalgebroids}\label{subsec:generated1}
Let $\balpha_1,\ldots,\balpha_n \in \Gamma(A)$ satisfying the following involutivity condition: For every $1\leq i,j \leq n$ there exist smooth functions $ f_{ij}^1,\ldots,f_{ij}^n \in C^{\infty}(M)$ such that $[\balpha_i,\balpha_j]=\sum_{k=1}^n f_{ij}^k \balpha_k$. Then the $C^{\infty}(M)$-submodule of $\Gamma(A)$ $$\cB:=C^{\infty}_c(M)\balpha_1 + \ldots + C^{\infty}_c(M)\balpha_n$$ is a singular subalgebroid. 

\begin{exs}\label{ex:momapsingsub}
\begin{enumerate}
\item Given a single section $\balpha \in \Gamma(A)$, the $C^{\infty}(M)$-module $\cB = C^{\infty}_c(M)\balpha$ is a singular subalgebroid.

\item {Let $\g$ be a Lie algebra and $\varphi \colon \g \to \Gamma(A)$ a Lie algebra morphism. Defining
 $\cB$ as the $C^{\infty}_c(M)$-{span} of $\{\varphi(v):v\in \g\}$ we obtain a singular subalgebroid of the kind above.
As $\balpha_1,\ldots,\balpha_n$ we can take  the image of a basis of $\g$; notice that in this case the functions $f_{ij}^k$ are constant. This example also  falls\footnote{Indeed
$\cB$ is the image of the Lie algebroid morphism $\g\times M\to A, (v,x)\mapsto (\varphi(v))|_x$ over $Id_M$, where 
$\g\times M$ is the transformation Lie algebroid of the infinitesimal action $\g\to \vX(M), v\mapsto \rho(\varphi(v))$ induced by $\varphi$ and the anchor of $A$.} into the class  considered in \S\ref{subsec:morph}.}

{A concrete example is the following.}  Let $(M,\omega)$ be a symplectic manifold, $\g$ a Lie algebra, and $J \colon M\to \g^*$ the moment map for some hamiltonian  action on $M$. Then the comoment map (pullback of functions)
$J^*\colon \g \to C^{\infty}(M)$ delivers a Lie algebra morphism into the central extension of $TM$ by the trivial vector bundle $M \times \RR$ twisted by $\omega$\footnote{Recall that $TM\oplus_{\omega} (M \times \RR)$ is a Lie algebroid with the bracket $[X\oplus V, Y\oplus W]=[X,Y]\oplus\{X(W)-Y(V)-\omega(X,Y)\}$.}:
$$\g \to \Gamma(TM\oplus_{\omega} (M \times \RR)), v \mapsto (X_{J^*v}, J^*v)$$
where $X_{J^*v}$ is the Hamiltonian vector field of  $J^{*}v \in C^{\infty}(M)$. When $\omega$ is an integral 2-form, $TM\oplus_{\omega} (M \times \RR)$  is the  Atiyah  algebroid of a circle bundle prequantizing $(M,\omega)$.
\end{enumerate}
\end{exs}

\subsubsection{From Lie subalgebroids supported on submanifolds}\label{subsec:Liesubalgoids}

Recall that a \emph{Lie subalgebroid} of $A$ over a closed embedded submanifold\footnote{
A Lie subalgebroid of $A$ over an immersion $\iota \colon N \to M$ can be defined as well: It is a vector bundle $B \to N$ together with a vector bundle morphism $j \colon B \to A$ over $\iota \colon N \to M$ such that: i) $\rho(j(B)) \subset \iota_*(TN)$, ii) $\tilde{\Gamma}(B) := \{\balpha \in \Gamma(A) \colon \balpha|_{\iota(N)} \subset j(B)\}$ is involutive, iii)  $[\tilde{\Gamma}(B),\tilde{\Gamma}(0_N)]\subset \tilde{\Gamma}(0_N)$.
}
 $N$ of $M$ (see \cite[Def. 4.3.14]{MK2}) is a subbundle $B\to N$, such that: 
 \begin{enumerate}
\item [i)] $\rho(B)\subset TN$,
\item  [ii)] $\tilde{\Gamma}(B):=\{\balpha\in \Gamma(A):\balpha|_N\subset B\}$ is involutive,
\item [iii)]
$[\tilde{\Gamma}(B),\tilde{\Gamma}(0_N)]\subset \tilde{\Gamma}(0_N)$, where $\tilde{\Gamma}(0_N):=\{\balpha\in \Gamma(A):\balpha|_N=0\}$.
\end{enumerate}

If   $B$ is a Lie subalgebroid of $A$  over a closed {embedded} submanifold $N$, 
then $$\cB:=\{\balpha\in \Gamma_c(A):\balpha|_N\subset B\}$$ is a singular subalgebroid  of $A$.

Let us describe $\cB$ near a point $p$ of $N$. Choose coordinates $\{x_i\}$ around $p$ adapted to $N$, \ie $\{x_i\}_{i>n}$ vanish on $N$ and $\{x_i\}_{i\le n}$, once restricted to $N$, provide coordinates there (here $n=dim(N)$). Let $\{\balpha_j\}$ be a frame of  {compactly supported sections} of $A$ adapted to $B$, \ie $\{\balpha_j|_N\}_{j\le b}\subset B$ where $b=rank(B)$.
Then $\cB$, locally near $p$, is generated by $$\{\balpha_j\}_{j\le b}\cup \{x_i\cdot \balpha_j\}_{i>n,j> b},$$
while on open sets disjoint from $N$,  $\cB$ is just given by restrictions of  {compactly supported} sections of $A$.
When $N$  has codimension one in $M$, $\cB$ is   projective {(see Def. \ref{ex:proj}).}
If $codim(N)\ge 2$ and $B\neq A|_N$, then $\cB$ is not   projective, because the number of generators above is strictly larger than   $rank(A)$.

 \begin{exs}\label{ep:GuLi}
\begin{enumerate}
\item Let $A=TM$. Let $B$ be the zero vector subbundle over $N$. Then $\cB$ consists of the vector fields on $M$ which vanish at points of $N$.

\item When $N$  has codimension one in $M$, as mentioned above, $\cB$ {is a projective singular subalgebroid,  i.e.} it  consists of {compactly supported sections} of an honest Lie algebroid over $M$, which Gualtieri and Li  \cite[Def. 2.11]{GuLi}  call \emph{elementary modification of $A$ along $B$} and  denote by $[A:B]$. We remark that they construct an integration of $[A:B]$ applying a blow-up procedure to a Lie groupoid integrating $A$ (assuming that $A$ is integrable) \cite[Thm. 2.9, Cor. 2.10]{GuLi}.

{ In particular, when $A=TM$ and $B=TN$, the Lie algebroid $[TM:TN]$ is called the \emph{log tangent bundle} associated to $N$, 
 and 
 $\cB$ is the projective singular foliation consisting of vector fields on $M$ tangent to $N$.}  \end{enumerate}
\end{exs}

\subsubsection{From Poisson geometry}\label{subsec:Pois}

{For the cotangent Lie algebroid of a Poisson manifold, certain singular subalgebroids can be constructed out of functions.}
Let $(M,\pi)$ be a Poisson manifold, and consider the 
Poisson algebra $(C^{\infty}(M),\cdot,\{\cdot,\cdot\})$.
Let $\cS$ be
a  \emph{Poisson} subalgebra of  $C^{\infty}(M)$, {which is locally finitely generated as a multiplicative algebra}.
  Then 
$${\cB:=Span_{C_c^{\infty}(M)}\{df:f\in \cS\}}$$  
is a singular subalgebroid   of the cotangent Lie algebroid $T^*M$.   {To see that $\cB$ is locally finitely generated as a $C^{\infty}(M)$-module one just uses the product rule, and to see that $\cB$ 
 is involutive, use the Leibniz rule for the Poisson bracket and the fact that $[df,dg]=d\{f,g\}$.} 
 
\begin{exs}\label{ex:Poissonmap}
\begin{enumerate}
\item  {Let $\phi \colon M\to P$ a Poisson map between Poisson manifolds, and $\cS_P$ a  Poisson subalgebra of $C^{\infty}(P)$ which is locally finitely generated as a multiplicative algebra.
Then the same holds for $\cS:=\phi^*(\cS_P)\subset C^{\infty}(M)$}.
 
\item {A special case of the above is given by Poisson maps $M \to \g^*$ to the dual of a Lie algebra, and choosing $\cS_{\g^*}$ to consist of polynomial functions on $\g^*$. Notice that in this case $\cB$ is of the kind\footnote{Indeed, $\g \to \Gamma(T^*M), v\mapsto d(\phi^*(v))$ is a Lie algebra morphism.} described in Ex. \ref{ex:momapsingsub} b).}

\item Let $f\in C^{\infty}(M)$. Then {$Span_{C^{\infty}_c(M)}\{df\}$} is a singular subalgebroid  of $T^*M$.
(This can be seen as a special case of \S \ref{subsec:generated1}, or {as a special case of the above taking  $\cS$ to be the   Poisson subalgebra of  $C^{\infty}(M)$ generated by $f$}.)
 For instance, take  $M=\RR^2$ with the standard ``symplectic'' structure $\pi=\partial_x\wedge \partial_y$, and let $f(x,y)=xy$. Then ${\cB}$ is given by all $C^{\infty}_c(\RR^2)$-multiples of $d(xy)$. 
{The anchor map     $\Pi\colon T^{\ast}\RR^2 \to T\RR^2$ of the cotangent Lie algebroid is just contraction with $\pi$. The singular foliation $\Pi(\cB)$ of $\RR^2$ {induced by $\cB$} is interesting: its leaves  agree with the connected components of the $f$-fibers, except on the preimage of $0$: $f^{-1}(0)$ is the union of the axes, and 
 it consists of 5 leaves, namely the 4 open half-axes and the origin.}
\end{enumerate}
\end{exs}

\subsection{Singular subalgebroids and right-invariant vector fields}\label{sec:rivf}
\label{sec:twofols}

Let $\cB$ be a singular subalgebroid of a Lie algebroid $A$ over $M$ with anchor $\rho$. $\cB$ induces a singular foliation on $M$, whose leaves are contained in the orbits of the Lie algebroid $A$, namely
\begin{equation}\label{eq:fb}
\cF_{\cB}:=\{\rho(\alpha):\alpha\in \cB\}.
\end{equation}
 
In this subsection we are concerned with another  singular foliation  associated to $\cB$, which will be very important to carry out our constructions. Assume $A$ is integrable and fix a Lie groupoid $\cG \gpd M$ integrating $A$. We denote the source and target maps of $\cG$ by $\bs, \bt \colon \cG \to M$ and identify the Lie algebroids $\ker(d\bs)\mid_M$ and $A$. 

Given a section $\balpha \in \cB \subset \Gamma_c(M,\ker(d\bs)|_M)$, put $\overset{\rightarrow}{\balpha}$ the right-invariant vector field on $\cG$ which extends $\balpha$.  Recall that
$\overset{\rightarrow}{\balpha}$ is an element of
$\Gamma(\cG,\ker(d\bs))\subset \vX(\cG)$, and
it is given by the formula $\overset{\rightarrow}{\balpha}_g = (R_g)_*\balpha_{\bt(g)}$, for all $g \in \cG$. We will consider  {the singular foliation}
\begin{equation}
\rar{\cB}:={Span_{C^{\infty}_{c}(\cG)} \{ \overset{\rightarrow}{\balpha} \mid \balpha \in \cB\}}.
\end{equation}
 Likewise, we denote by  $\overset{\leftarrow}{\cB}$   the ${C^{\infty}_c(\cG)}$-module 
 generated by the left-invariant vector fields $\overset{\leftarrow}{\balpha}$ for all $\balpha \in \cB$. {All the statements made  in this subsection for $\rar{\cB}$ hold in a similar way for  $\lar{\cB}$ too.}
  
 Notice that  $\text{support}(\rar{\balpha})=\bt^{-1}(\text{support}(\balpha))$, hence $\rar{\balpha}$ is not necessarily compactly supported.
 However, we have:
\begin{lemma}\label{lem:complete}
For every  section $\balpha \in \cB$ the vector field $\rar{\balpha} \in \rar{\cB}$ is complete.
\end{lemma}
\begin{proof}
Identifying the anchor map $\rho\colon A \to TM$ with $d\bt|_M\colon\ker(d\bs)\mid_M \to TM$, we get that $\rar{\balpha}$ is $\bt$-related with the vector field $\rho(\balpha)$, which has the same support as $\balpha$, whence it is complete. It follows from \cite[Thm. 3.6.4]{MK2} that $\rar{\balpha}$ is complete as well.
\end{proof}

{We now relate local generators of $\cB$ with local generators of $\rar{\cB}$.}
  Given $x\in M$,
let $I_x^M$ denote the ideal of functions on $M$ vanishing at $x$, and
$I_x^{\cG}$ the ideal of functions on $\cG$ vanishing at $x$.

{
\begin{remark}\label{rem:basis}
Let $\balpha_1,\cdots,\balpha_n\in \cB$. These elements are 
generators of $\cB$ in a neighborhood of $x$ if{f}  their images $[\balpha_1],\cdots,[\balpha_n]$
in $\cB/I_x^M\cB$ are a spanning set of this vector space. This is proved exactly as in the case of singular foliations \cite[Prop. 1.5 a)]{AndrSk}. If the latter form a basis of $\cB/I_x^M\cB$,  we say that $\balpha_1,\cdots,\balpha_n$ is a {\bf minimal} set of local generators.  
\end{remark}
}

\begin{lemma}\label{lem:basis}
Let $\balpha_1,\dots,\balpha_n$ be a finite subset of $\cB$.
Then $[\balpha_1],\dots,[\balpha_n]$ is a basis of $\cB/I_x^M\cB$ if{f}
 $[\overset{\rightarrow}{\balpha_1}],\dots,[\overset{\rightarrow}{\balpha_n}]$ is a basis\footnote{Actually the $\overset{\rightarrow}{\balpha_i}$ do not lie in $\rar{\cB}$ but rather in the 
global hull $\widehat{\rar{\cB}}$, see \S \ref{section:singdef}. This does not pose any problems since  the inclusion of  $\rar{\cB}$ in $\widehat{\rar{\cB}}$ induces an isomorphism 
 $\rar{\cB}/I_x^{\cG}\rar{\cB}\cong \widehat{\rar{\cB}}/I_x^{\cG}\widehat{\rar{\cB}}$.}
  of $\rar{\cB}/I_x^{\cG}\rar{\cB}$. 
\end{lemma} 
\begin{proof} 
``$\Rightarrow$''  
We first show that the $[\overset{\rightarrow}{\balpha_i}]$ are linearly independent. Let $c_1,\dots,c_n\in \RR$ with $\sum c_i \overset{\rightarrow}{\balpha_i} \in I_x^{\cG}\rar{\cB}$. Restricting from $\cG$ to $M$ we obtain $\sum c_i \balpha_i\in I_x^M\cB$, therefore all coefficients $c_i$ are zero. We now show that the $[\overset{\rightarrow}{\balpha_i}]$ are a spanning set of $\rar{\cB}/I_x^{\cG}\rar{\cB}$. The $\balpha_i$ generate the $C^{\infty}(M)$-module $\cB$ in a neighborhood of $x$ (see Remark \ref{rem:basis}), and hence
 the $\overset{\rightarrow}{\balpha_i}$ generate the $C^{\infty}(\cG)$-module  $\rar{\cB}$ near  $x$. Given any $X\in \rar{\cB}$, there are $f_i\in {C_c^{\infty}(\cG)}$ such that $X=\sum f_i \rar{\balpha_i}=\sum f_i(x) \rar{\balpha_i}+\sum (f_i-f_i(x))\rar{\balpha_i}$, and since $f_i-f_i(x)\in I_x^{\cG}$ we obtain $[X]=\sum f_i(x) [\rar{\balpha_i}]$.

``$\Leftarrow$'' The $[\balpha_i]$ are linearly independent:
 if $\sum c_i \balpha_i \in  I_x^M\cB$ then $\sum c_i \overset{\rightarrow}{\balpha_i} \in \bt^*(I_x^{M}) \rar{\cB}  
\subset  I_x^{\cG}\rar{\cB}$, showing that the $c_i$ all vanish. To show that the $[\balpha_i]$ are a spanning set of $\cB/I_x^M\cB$, notice that by assumption  any element of $\rar{\cB}$ can be written as $\sum c_i \rar{\balpha_i}$ (for suitable $c_i\in \RR)$ plus an element of $ I_x^{\cG}\rar{\cB}$. We pick $\balpha\in \cB$ and write $\rar{\balpha}$ in the above form.
 Restricting to $M$ we see that $\balpha$ equals $\sum c_i \balpha_i$ plus an element of $ I_x^{M}{\cB}$, \ie $[\balpha]=\sum c_i [\balpha_i]$.
\end{proof}

\section{Bisubmersions for singular subalgebroids}\label{section:relbisub}

In this whole section we
fix an integrable Lie algebroid $A\to M$ and a singular subalgebroid $\cB$. Further, we fix a Lie groupoid $\cG$ integrating $A$.

Recall from \cite{AndrSk} that the key ingredient for the construction of the holonomy groupoid of a singular foliation is the notion of bisubmersion. Here, in order to carry out the construction in the case of singular subalgebroids, we reformulate the notion of bisubmersion {in \S \ref{subsec:defbisub}. We then present examples, including
  path holonomy bisubmersions. The latter, upon applying the operations we outline in \S \ref{subsec:operations}, will be used to construct the   holonomy groupoid in the next Section. {Our \S \ref{subsubsec:phrel} and \S \ref{subsec:operations} follow closely \cite{AndrSk}.}
}

\subsection{Pullbacks of singular foliations}

{We collect   background material on pullbacks and generating sets for
 singular foliations (see \S \ref{subsec:motivex}).} 
\begin{definition}\label{def:pullback}
Let $\varphi \colon U\to V$ a smooth map between smooth manifolds.
\begin{enumerate}
\item Let $X\in \vX(U)$ and $Y\in \vX(V)$. We say that $X$ is {\bf $\varphi$-related} to $Y$ if{f}  $\varphi_*(X(p))=Y(\varphi(p))$ for all $p\in U$.
\item
Let $\cF$ be a $C^{\infty}(V)$-submodule of ${\vX_c(V)}$.
Define, as in \cite{AndrSk} (see also \cite[\S 1.1]{AZ2})
$$\varphi^{-1}(\cF):=\{X\in \vX_c(U):d\varphi(X)=\sum f_i(Y_i\circ \varphi) \text{ for finitely many $f_i\in C^{\infty}_{c}(U)$ and $Y_i\in \cF$}\}.$$ 
Here $d\varphi\colon TU\to \varphi^*(TV)$ is a vector bundle map covering $Id_U$, where $\varphi^*(TV)$ denotes the pullback vector bundle. Notice that 
$\varphi^{-1}(\cF)$ is a $C^{\infty}(U)$-submodule of   $\vX_c(U)$.   It is a foliation, {called \textbf{pullback foliation}}, whenever $\cF$ is a foliation and  $\varphi$ is transverse to $\cF$ \cite{AndrSk}.  
\end{enumerate}
\end{definition}
 
 Now let $\varphi \colon U\to V$ a smooth map and  $\cF$ be a $C^{\infty}(V)$-submodule of ${\vX_c(V)}$.
Fix a  generating set $\mathcal{G}$ of $\cF$, {as defined in \S \ref{section:singdef}}. We display two technical lemmas, which are not completely obvious due to the fact that {neither  $\mathcal{G}\subset \cF$ nor  $\mathcal{G}\supset \cF$ in general}.

\begin{lemma}\label{lem:liftequiv}
{The following conditions are equivalent:
\begin{itemize}
\item for every $Y\in \cF$ there is
a $Z\in \vX(U)$ which is $\varphi$-related to $Y$.
\item for every $Y\in \mathcal{G}$ there is a $Z\in \vX(U)$ which is $\varphi$-related to $Y$.
\end{itemize}
}
\end{lemma} 
\begin{proof}
We only show that the first condition implies the second (the converse is similar). Take a partition of unity $\{\psi_a\}$ on $V$ by functions with compact support. Given $Y\in \mathcal{G}$, by assumption there is $Z_a\in \vX(U)$ that is $\varphi$-related to $\psi_a Y\in \cF$. Since the partition of unity is locally finite, {one can arrange that} the sum $\sum_a Z_a$ is locally finite, and the resulting vector field is $\varphi$-related to $\sum_a\psi_a Y=Y$.
\end{proof}

Under certain conditions, $\varphi^{-1}(\cF)$ has a distinguished generating set.
\begin{lemma}\label{lem:fix}
 Assume that any of the equivalent  conditions in Lemma \ref{lem:liftequiv} is satisfied. {(This happens for instance when $\varphi$ is a submersion).} Then
 \begin{align*}
 \varphi^{-1}(\cF)=&Span_{C^{\infty}_{c}(U)}\{\text{$Z\in \vX(U)$: $Z$ is
 $\varphi$-related to an element of $\cF$}\}\\
 =&Span_{C^{\infty}_{c}(U)}\{\text{$Z\in \vX(U)$: $Z$ is
 $\varphi$-related to an element of $\mathcal{G}$ {or to $0$}}\}.
 \end{align*}
\end{lemma}
\begin{proof} In the first equality, the inclusion ``$\supset$'' is easily checked to  hold even when the assumption is not satisfied. For ``$\subset$'', take $X\in \vX_c(U)$ such that $d\varphi(X)=\sum f_i(Y_i\circ \varphi)$ where the sum is finite,  $f_i\in C^{\infty}_{c}(U)$  and $Y_i\in \cF$. By the assumption, there exist $Z_i\in 
\vX(U)$ that is $\varphi$-related  to $Y_i$, \ie $d\varphi(Z_i)=Y_i\circ \varphi$. Hence we can write $$
d\varphi(X)=\sum f_i d\varphi(Z_i)=d\varphi(\sum f_i Z_i).$$ This means that
$X=\sum f_i Z_i+Z$ where $Z\in \vX_c(U)$ is $\varphi$-related to the zero vector field on $V$, which is an element of $\cF$.

For  the second equality, to prove ``$\supset$'', take $Z\in \vX(U)$ which is $\varphi$-related to $Y\in \mathcal{G}$ and  $f\in C^{\infty}_{c}(U)$.  We can choose a function  $\psi\in C^{\infty}_{c}(V)$ with is one on $\varphi(Supp(f))$. We then have $fZ=f\varphi^{*}(\psi)Z$, and clearly $\varphi^{*}(\psi)Z$ is $\varphi$-related to $\psi Y\in \cF$. To show ``$\subset$'', take $X\in {\vX(U)}$ which is $\varphi$-related to an element of $\cF$, \ie to some $\sum g_i Y_i$ where $g_i\in C^{\infty}_{c}(U)$ and $Y_i\in \mathcal{G}$. By assumption there is $X_i\in \vX(U)$ that is $\varphi$-related  to $Y_i$, hence  $\sum \varphi^{*}(g_i) X_i$ is also related to $\sum g_i Y_i$. {Fix $f\in C^{\infty}_{c}(U)$. Then $fX=f\sum \varphi^{*}(g_i) X_i+Z$  where $Z\in \vX_c(U)$ is $\varphi$-related to the zero vector field on $V$.}
\end{proof}

\subsection{Definition of bisubmersion}\label{subsec:defbisub}

Let $A$ be an integrable Lie algebroid and $\cG$ a Lie groupoid integrating $A$. Let $\cB$ be a singular subalgebroid of $A$. 
\begin{definition}\label{dfn:bisubm2} A {\bf bisubmersion} for $\cB$ is a smooth map $\varphi\colon U \to \cG$, where $U$ is a manifold, such that 
\begin{enumerate}[i)]
\item $\bs_U:=\bs\circ\varphi$ and $\bt_U:=\bt\circ\varphi \colon U \to M$ are submersions,
\item for every $\balpha\in \cB$, there is $Z\in\vX(U)$
 which is $\varphi$-related to $\rar{\balpha}$ and $W\in\vX(U)$ which is $\varphi$-related to $\lar{\balpha}$,
\item 
$\varphi^{-1}(\rar{\cB})=  \Gamma_c(U,\ker d\bs_U)$ and $\varphi^{-1}(\lar{\cB})=  \Gamma_c(U,\ker d\bt_U)$.
\end{enumerate}
\end{definition}

\begin{notation}\label{not:bisubm2}
We denote a bisubmersion of $\cB$ by $(U,\varphi,\cG)$. One can bear in mind the following diagram:
\begin{equation*}
\xymatrix{
&U\ar[d]^{\varphi} & \\
&\ar[dr]^{\bs}\cG\ar[dl]_{\bt}&\\
M&&M 
}
\end{equation*}
\end{notation}
 
 \begin{remark}
In Def. \ref{dfn:bisubm2}  the map $\varphi$ is not required to be transverse to $\rar{\cB}$, and 
the conditions in Def. \ref{dfn:bisubm2} do not imply transversality in general (see the examples in \S \ref{subsub:LieGrbisub} with $\cK$ there being the trivial groupoid).
\end{remark}

{The rest of this subsection is devoted to explanations about conditions ii) and iii).}

\begin{remark}\label{rem:condii}
The first part condition ii) in Def. \ref{dfn:bisubm2} is expressed in terms of the generating set $\{ \overset{\rightarrow}{\balpha} \mid \balpha \in \cB\}$ of the singular foliation $\rar{\cB}$. Using Lemma \ref{lem:liftequiv} it can be  rephrased saying that any element of $\rar{\cB}$ can be lifted to $U$. Notice that any lift will lie in $\ker d\bs_U$, since right-invariant vector fields on $G$ lie in $\ker d\bs$.
 An analogue statement holds for the second part of condition ii).
\end{remark}

We now phrase the first part of condition iii) in Definition \ref{dfn:bisubm2} more explicitly.
{
 We have
\begin{align}\label{eq:varrar}
\varphi^{-1}(\rar{\cB})&=Span_{C^{\infty}_{c}(U)}\{\text{$Z\in\vX(U): Z$ is $\varphi$-related to $\overset{\rightarrow}{\balpha}$ for some $\balpha \in \cB$}\}\\
&=Span_{C^{\infty}_{c}(U)}\Gamma(U,\ker d\bs_U)^{proj,{\cB}},\nonumber
\end{align}
where the first equation holds by  Lemma \ref{lem:fix} (which can be applied due to condition ii)), {and in the second equation we used that   $\rar{\cB}$ lies in the kernel of $d\bs$}.
Here, $$\Gamma(U,\ker d\bs_U)^{proj,{\cB}}:=\{Z\in \Gamma(U,\ker d\bs_U): \text{$Z$ is  $\varphi$-related to $\overset{\rightarrow}{\balpha}$ for some $\balpha \in \cB$}\}.$$
}

\begin{lemma}\label{lem:bisubm2}
Given any map $\varphi\colon U \to \cG$, the following statements are equivalent:
\begin{itemize}
\item[a)] $\varphi^{-1}(\rar{\cB})=  \Gamma_c(U,\ker d\bs_U)$
\item[b)] $\Gamma(U,\ker d\bs_U)^{proj,{\cB}}$
generates ${\Gamma_c(U,\ker d\bs_U)}$ as a ${C^{\infty}_c(U)}$-module
\item[c)] $\Gamma(U,\ker d\bs_U)^{proj,{\cB}}$ spans $\ker (d\bs_U)_u$ at every $u\in U$.
\end{itemize}
\end{lemma}
\begin{proof}
Using eq. \eqref{eq:varrar} it is clear that a) and b) are equivalent.

Clearly b) implies c). For the converse, {we first make an observation}. 
By c), for any $u\in U$, we can take a basis of  $\ker (d\bs_U)_u$  and extend it to  $\balpha^1,\dots,\balpha^n\in \Gamma(U,\ker d\bs_U)^{proj,{\cB}}$. {These vector fields are  linearly independent  on a small open 
 neighborhood $U'$ of $u$.   
Therefore can write any section of $\ker d\bs_U$ with support in $U'$  as a 
$C_c^{\infty}(U')$-linear combination of the $\balpha^i$. 
 }
   
 Now fix $X\in {\Gamma_c(U,\ker d\bs_U)}$. The support of $X$, being compact, can be covered by finitely many open neighborhoods as above. Extend this {cover} to a open cover $\{U'_i\}$ of $U$, and choose a partition of unity $\{g_i\}\subset 
C_c^{\infty}(U)$ subordinate to it. Then $X=\sum g_i(X|_{U'_i})$ is a \emph{finite} sum. {The summands are
 sections of $\ker d\bs_U$ with support in $U'_i$. 
Applying the above observation to each summand, we see that $X$ is written as a finite $C_c^{\infty}(U)$-linear combination of elements of $\Gamma(U,\ker d\bs_U)^{proj,{\cB}}$.} \end{proof}

{We mention another characterization of Def. \ref{dfn:bisubm2}:}
\begin{remark}\label{donto}
{The first parts of conditions ii) and iii) in Definition \ref{dfn:bisubm2} are equivalent to the following: the map of $C^{\infty}(U)$-modules
 $$d\varphi \colon \Gamma_c(U;\ker d\bs_U) \to \varphi^*(\rar{\cB})$$
is well-defined and surjective. Here  $\varphi^*(\rar{\cB})$ is the $C^{\infty}(U)$-submodule of $\vX_c(U)$ generated by $f(\xi\circ\varphi)$ with $f \in C^{\infty}_c(U)$ and $\xi \in \rar{\cB}$.  }

Indeed, the first part of condition iii)   is equivalent  to $\varphi^{-1}(\rar{\cB}) \supset   {\Gamma_c(U,\ker d\bs_U)}$ (the other inclusion is obvious since $\rar{\cB}$ lies in the kernel of $d\bs$), and therefore is equivalent to the fact that  the 
above map is well-defined. The first part of condition ii), by Remark \ref{rem:condii},  says that this map is onto.
\end{remark}

  \subsection{Examples}\label{subsec:exbisub}
  
{We exhibit examples of bisubmersions for singular foliations and wide Lie subalgebroids (the two motivating examples displayed in \S \ref{subsec:motivex}), and generalizing the latter, for the examples treated in   \S \ref{subsec:morph}.}

\subsubsection{Bisubmersions of singular foliations}\label{section:usualbisub}

In \cite[definition 2.1]{AndrSk} a {\bf bisubmersion for a singular foliation} $(M,\cF)$ is defined as a triple $(U, \bt_U, \bs_U)$  consisting of a manifold $U$ with two submersions $\bt_U$ and $\bs_U$ to $M$, such that 
\begin{equation}\label{eq:bisubfol}
\bt_U^{-1}(\cF)=\Gamma_c(U,\ker d\bt_U) + \Gamma_c(U,\ker d\bs_U)=\bs_U^{-1}(\cF).
\end{equation}

On the other hand, singular foliations are special cases of  singular subalgebroids; namely, they are the singular subalgebroids of $TM$. We show that the two notions of bisubmersion for singular foliations {essentially agree, since there is a canonical bijective correspondence between them}. 

\begin{prop}\label{prop:equivbi}
Let $(M,\cF)$ be a  singular foliation, $U$ a manifold, and $\bt_U \colon U\to M$ and  $\bs_U \colon U\to M$ submersions. The following are equivalent:
\begin{itemize}
\item[1)] $(U, \bt_U, \bs_U)$ is a bisubmersion for the singular foliation $\cF$ (in the sense of \cite[definition 2.1]{AndrSk});
\item[2)] the map $(\bt_U, \bs_U)\colon U\to M\times M$ is a bisubmersion (in the sense of definition \ref{dfn:bisubm2}), {where $M\times M$ is endowed with the pair groupoid structure}.
\end{itemize}
\end{prop}
We give the following lemma without proof.
 \begin{lemma}\label{lem:easyinclusion}
Let $\bt_U \colon U\to M$ and  $\bs_U \colon U\to M$ {be smooth maps}.
If $Z\in \vX(U)$ and $X,Y\in \vX(M)$, then $Z$ is $(\bt_U, \bs_U)$-related to  $\rar{X}+\lar{Y}$ if{f} it is $\bt_U$-related to  $X$ and $\bs_U$-related to  $Y$.
{Here $\rar{X}+\lar{Y}$ denotes the vector field $(X,Y)$ on $M\times M$.}
\end{lemma} 

\begin{proof}[Proof of proposition \ref{prop:equivbi}]
$1)\Rightarrow 2)$: Property i) in definition \ref{dfn:bisubm2} is obviously satisfied. For property ii)  we argue as follows. Every $X\in \cF$ can be $\bt_U$-lifted to $Z\in \Gamma(\ker d\bs_U)$, by the proof of \cite[Prop. 2.10 b)]{AndrSk}. Such a $Z$ is $(\bt_U, \bs_U)$-related to  $\rar{X}$ by Lemma \ref{lem:easyinclusion}.
Further, $X$ can be $\bs_U$-lifted to $W\in \Gamma(\ker d\bt_U)$, by the same argument in the proof of \cite[Prop. 2.10 b)]{AndrSk}(interchanging the roles of source and target), and such a $W$  is $(\bt_U, \bs_U)$-related to  $\lar{X}$.

We show {property iii)}, that is, $(\bt_U, \bs_U)^{-1}(\rar{\cF})=  \Gamma_c(U,\ker d\bs_U)$. 
We do so using  Lemma \ref{lem:bisubm2} b).
Let $Z\in \Gamma_c(U,\ker d\bs_U)$. By the {first equality in eq.} \eqref{eq:bisubfol},  we have $Z=\sum g_iZ_i $ where $g_i\in C_c^{\infty}(U)$ and $Z_i \in \vX(U)$ is $\bt_U$-projectable to an element of $\cF$.  
We can write each $Z_i$ as the sum of a vector field in $\Gamma(\ker d\bt_U)$ and one in $\Gamma(\ker d\bs_U)$ -- which we denote by the respective subscrips --.
{To see this, we have to apply with some care the first equality in eq. \eqref{eq:bisubfol}: choose a partition of unity $\{\psi_a\}$  with compact support on $U$. Then $\psi_aZ_i$ equals an element of $\Gamma_c(\ker d\bt_U)$ plus an element of $\Gamma_c(\ker d\bs_U)$, and we may assume\footnote{By multiplying them with a suitable function with value $1$ on $Supp(\psi_a)$.}
 that their support is contained in a small enough neighborhood of $Supp(\psi_a)$.
Hence $Z_i=\sum_a \psi_aZ_i$ equals a locally finite  sum of elements of $\Gamma_c(\ker d\bt_U)$ plus a locally finite  sum of elements of $\Gamma_c(\ker d\bt_U)$.}
Altogether we obtain
$$Z=  Z'+\sum g_i(Z_i)_{\ker d\bs_U},\;\; \text{  where }Z':=\sum g_i(Z_i)_{\ker d\bt_U}.$$
Notice that $Z'$ lies in $\Gamma_c(U,\ker d\bs_U)$ (being the difference of two vector fields with this property), 
hence it is $(\bt_U, \bs_U)$-related to the zero vector field.
Similarly, each $(Z_i)_{ker d\bs_U}$ is $\bt_U$-projectable to an element of $\cF$ (being the difference $Z_i-(Z_i)_{ker d\bt_U}$ of two vector fields with this property), hence {by Lemma \ref{lem:easyinclusion}} it is $(\bt_U, \bs_U)$-related to an element of {$\{\overset{\rightarrow}{X} : X \in \cF\}$}.
 In conclusion,
$Z$ is a $C_c^{\infty}(U)$-linear combination of vector fields which are $(\bt_U, \bs_U)$-related to  elements of {$\{\overset{\rightarrow}{X} : X \in \cF\}$},
proving that the condition in Lemma \ref{lem:bisubm2} b) holds.  {Lemma \ref{lem:bisubm2} b)}. {The second equality in item iii) of  definition \ref{dfn:bisubm2} follows in an analogous way.}

$2)\Rightarrow 1)$:  $\bt_U$ and $\bs_U$ are submersions by item i) of  definition \ref{dfn:bisubm2}. Hence we just need to show eq. \eqref{eq:bisubfol}.
By Lemma \ref{lem:easyinclusion} and item iii) of  definition \ref{dfn:bisubm2} we have
\begin{align*}
\bt_U^{-1}(\cF)\cap \bs_U^{-1}(\cF)\supset (\bt_U, \bs_U)^{-1}(\rar{\cF})+(\bt_U, \bs_U)^{-1}(\lar{\cF})=\Gamma_c(U,\ker d\bs_U) + \Gamma_c(U,\ker d\bt_U).
\end{align*}
We show $\bt_U^{-1}(\cF)\ \subset \Gamma_c(U,\ker d\bs_U) + \Gamma_c(U,\ker d\bt_U)$ (the argument for $\bs_U^{-1}(\cF)$ is analogous). Let $Z\in \vX(U)$ be $\bt_U$-related  to some $X\in \cF$. By condition ii), there is a vector field $W$ on $U$ which is $(\bt_U,\bs_U)$-related to $\rar{X}$, \ie $W$ lies in $\Gamma(\ker d\bs_U)$ and is 
$\bt_U$-related to ${X}$.  We conclude by writing
$Z=(Z-W)+W$, with $Z-W\in \Gamma(\ker d\bt_U)$.
\end{proof}
 
Further, a bisubmersion for any  singular subalgebroid $\cB$ gives rise to a bisubmersion for $\cF_{\cB}$, the induced singular foliation defined in eq. \eqref{eq:fb}. We have:

\begin{lemma}\label{lem:usualbi}
Let $(U,\varphi,\cG)$ a bisubmersion of $\cB$. Then $(U,\bt_U,\bs_U)$ is a bisubmersion of the singular foliation $\cF_{\cB}$ (in the sense of \cite{AndrSk}).
\end{lemma}
 
 \begin{proof} Since $\bt_U$ and $\bs_U$ are submersions,  we just have to show that 
$$\bt_U^{-1}(\cF_{\cB})=\Gamma_c(U,\ker d\bt_U) + \Gamma_c(U,\ker d\bs_U)$$
and similarly for $\bs_U^{-1}(\cF_{\cB})$. We start proving the above equality.

``$\supset$'' By item iii) of  definition \ref{dfn:bisubm2} and {by Lemma \ref{lem:bisubm2}
 b)}, $\Gamma_c(U,\ker d\bt_U) + \Gamma_c(U,\ker d\bs_U)$ is generated by elements   which are $\varphi$-related to  elements of {$\{ \overset{\rightarrow}{\balpha} : \balpha \in \cB\}+
\{ \overset{\leftarrow}{\balpha} : \balpha \in \cB\}$}.
The latter are $\bt$-related to elements of $\cF_{\cB}$. As $\bt\circ \varphi=\bt_U$,  the above generators are $\bt_U$-related to  elements of $\cF_{\cB}$.

``$\subset$''  Let $Z\in \vX(U)$ be $\bt_U$-related to some $X\in \cF_{\cB}$. There exists $\balpha\in \cB$ with $\rho(\balpha)=X$. Since under the identification $A\cong \ker(d\bs)|_M$ the anchor $\rho$ is identified with $d\bt|_M$, we see that $\rar{\balpha}$ is $\bt$-related to $X$. By item ii) of  definition \ref{dfn:bisubm2}, there is $W\in \vX(U)$ which is  $\varphi$-related to  $\rar{\balpha}$. Hence $W$ is also related to $X$ under the map  $\bt_{U}=\bt\circ \varphi$.  Notice that  
$W\in 
\Gamma(U,\ker d\bs_U)$, and further $Z-W\in \Gamma(U,\ker d\bt_U)$. Writing $Z=W+(Z-W)$ we conclude the proof of the inclusion.

To show the above equality for $\bs_U^{-1}$ in place of 
$\bt_U^{-1}$
we proceed as follows: consider the above equality for the inverse bisubmersion $({U},\bar{\varphi}, \cG)$ (see definition \ref{def:inv}), and use $\bar{\bt}_{U}=\bs_U$, $\bs_{{U}}=\bar{\bt}_U$.
\end{proof}

 \subsubsection{Lie groupoid morphisms as bisubmersions}\label{subsub:LieGrbisub}

{If a singular subalgebroid arises from a Lie groupoid morphism (see Def. \ref{def:arises}, this includes wide Lie subalgebroids), then that morphism is automatically a bisubmersion:}

\begin{prop}\label{prop:imagerelbi}  
Let $\varphi\colon \cK \to \cG$ be
a morphism of   Lie groupoids covering the identity on $M$. Denote by $\cB:=\varphi_*(\Gamma_c(Lie(\cK)))$ {the  singular subalgebroid  of $Lie(\cG)$ it gives rise to}. Then $\varphi\colon \cK \to \cG$  is a bisubmersion for   $\cB$.
\end{prop}

\begin{proof}
We check the properties of definition \ref{dfn:bisubm2}:
\begin{enumerate}[i)]
\item is satisfied since $\cK$ is a Lie groupoid over $M$.
\item is an immediate consequence of the Claim below, {since any element of $\cB$ is of the form $\varphi_*X$ for some $X\in \cB$.}

\item   {for all $X \in \Gamma_c(Lie(\cK))$, the claim below implies that $\rar{X}$  
lies in $\Gamma(K,\ker d\bs_K)^{proj,{\cB}}$. Such $\rar{X}$ span $\ker d\bs_K$ at every point of $K$, so we can conclude using Lemma \ref{lem:bisubm2} c).
The same argument applies to  $\lar{X}$.}
\end{enumerate}

\underline{Claim:} \emph{Denote $E:=Lie(\cK)$ and $A:=Lie(\cG)$.
 Let $X \in \Gamma_c(E)$, and denote $\balpha_X := \varphi_*X\in \Gamma(A)$. Then $\rar{X}$ is 
$\varphi$-related  to   $\rar{\balpha_X}$. Similarly,  $\lar{X}$ is 
$\varphi$-related  to   $\lar{\balpha_X}$. }

Consider $\rar{X}$, the right-invariant vector field on $\cK$ which restricts to $X$ along $M$. Let $k\in \cK$. We have
$$\varphi_*(\rar{X}_k )= 
\varphi_*
 (R^{\cK}_{k})_*\left( X_{\bt(k)}\right)
  = (R^{\cG}_{\varphi(k)})_*\varphi_*\left( X_{\bt(k)}\right)
  =  (R^{\cG}_{\varphi(k)})_* \left((\balpha_X)_{\bt(k)}\right)
  = \rar{(\balpha_X)}_{\varphi(k)},$$
showing that  $\rar{X}$ is 
$\varphi$-related  to   $\rar{\balpha_X}$.
  Here we denote $R^{\cG}$ the right-translation by in $\cG$, and likewise $R^{\cK}$ the right-translation in $\cK$. In the second equality we used that $\varphi$ is a groupoid morphism.
  
Now consider $\lar{X}$. We have $\lar{X}=-i^{\cK}_*\rar{X}$, where $i^{\cK}$ is the inversion map of the Lie groupoid $\cK$. Hence
$$\varphi_*(\rar{X})=-\varphi_*i^{\cK}_*(\rar{X})=
-i^{\cG}_*\varphi_*(\rar{X})=-i^{\cG}_* (\rar{\alpha_X})=
\lar{\alpha_X},$$
where in the second equality we used that $\varphi$ is a groupoid morphism, and in the third that  $\rar{X}$ is 
$\varphi$-related  to   $\rar{\balpha_X}$.
\end{proof}

We spell out Prop. \ref{prop:imagerelbi} in the case of a wide Lie subalgebroid:  
\begin{cor}\label{prop:Lierelbi} Let $A$ be an integrable Lie algebroid and $B$ a wide Lie subalgebroid of $A$. Let $\varphi\colon \cK \to \cG$ be
a morphism\footnote{The map $\varphi$ is always a (not necessarily injective) immersion, and covers the identity on $M$.} of   Lie groupoids which integrates the inclusion $\iota \colon B\hookrightarrow A$. 

Then $\varphi\colon \cK \to \cG$  is a bisubmersion for the 
 singular subalgebroid $\cB = \Gamma_c(B)$.
\end{cor}

\begin{ex}\label{ep:groid}
For $\cB = \Gamma_c(A)$ we obtain that $\id\colon \cG \to \cG$ is a bisubmersion.
\end{ex}

\subsection{Path-holonomy bisubmersions}\label{subsubsec:phrel}

We give an explicit construction of bisubmersions starting from local generators of $\cB$. Recall from Lemma \ref{lem:complete} that if $\balpha \in \cB$, the vector field $\rar{\balpha}$ is complete. We denote by $\exp_y \rar{\balpha}$ its time-1 flow applied to the point $y\in M$.

\begin{definition}
\label{dfn:pathhol}  
\begin{enumerate}
\item Let $x \in M$ and $\balpha_1,\dots,\balpha_n \in \cB$ such that $[\balpha_1],\dots,[\balpha_n]$ span   $\cB/I_x\cB$. The associated {\bf {path holonomy} bisubmersion} is the map  $$\varphi\colon U  \to \cG, \;\;(\lambda,y)\mapsto \exp_y \sum \lambda_i \overset{\rightarrow}{\balpha_i},$$ {where $U$ is a neighborhood of $(0,x)$ in $\RR^n \times M$  such that $\bt\circ \varphi$ is\footnote{Such a neighborhood exists since
$\varphi$ is a submersion at $(0,x)$.}
 a submersion}, and we write $\lambda=(\lambda_1,\dots,\lambda_n)\in \RR^n$.

\item We say that $(U,\varphi, \cG)$ is {\bf minimal} (at $x$) if $[\balpha_1],\dots,[\balpha_n]$ is a basis of {the vector space}  $\cB/I_x\cB$.
\end{enumerate}
\end{definition}

\begin{ex}\label{ex:sfolpair}
Consider the case when $A=TM$, so that $\cF:=\cB$ is a singular foliation on $M$. 
Fix $x\in M$ and let $X_1,\dots,X_n\in \cF$ so that they induce a basis of $\cF/I_x\cF$.
We make two choices of Lie groupoid $\cG$ integrating the Lie algebroid $TM$.
\begin{itemize}
\item[a)]  As $\cG$ let us choose the pair groupoid $M\times M$. Then  the {path holonomy}  bisubmersion we obtain is
$$\varphi\colon U \to M\times M,\;\; (y,\lambda)\mapsto (\exp_y (\sum \lambda_i  {X_i}), y).$$ 
{Notice that under the bijection given by Prop. \ref{prop:equivbi}, it corresponds to the path holonomy bisubmersion $(U, \bt, \bs)$ for the singular foliation $\cF$ (induced by   $X_1,\dots,X_n$) as in \cite{AndrSk}.}

\item[b)] Now as $\cG$ we take   the fundamental groupoid $\Pi(M)$. {Then  the {path holonomy}  bisubmersion we obtain is
}
$$\widetilde{\varphi}\colon U \to \Pi(M),\;\; (y,\lambda)\mapsto 
\text{homotopy class of $\gamma_{(y,\lambda)}$},$$ where 
$\gamma_{(y,\lambda)}$ is the path $[0,1] \ni t\mapsto \exp_y (t\sum \lambda_i  {X_i})$, {\ie the integral curve of $\sum \lambda_i  {X_i}$ that starts at $y$}. 

{To see this, use the canonical Lie groupoid isomorphism $\tilde{M}\times_{\pi_1(M)}\tilde{M}\cong \Pi(M)$, where $\tilde{M}$ is the universal covering space of $M$, and take the unique lift of $\sum \lambda_i  {X_i}$ to a vector field on $\tilde{M}$.}
\end{itemize}

{Notice that the above two path holonomy bisubmersions are related} by $\varphi=\pi\circ\widetilde{\varphi}$, where $\pi\colon \Pi(M) \to M \times M$ is the morphism of Lie groupoids sending the homotopy class of a path $\gamma$ in $M$ to $(\gamma(0),\gamma(1))$. We will elaborate on this example in \S\ref{section:vary}.  
\end{ex}

Let us show now that the object defined in definition \ref{dfn:pathhol} is really a bisubmersion as in definition \ref{dfn:bisubm2}. To this aim recall that $\bs_U:=\bs\circ\varphi$ and $\bt_U:=\bt\circ\varphi \colon U \to M$.
 
\begin{prop}\label{prop:pathhol}
Let $x \in M$ , let $\balpha_1,\dots,\balpha_n \in \cB$ such that $[\balpha_1],\dots,[\balpha_n]$ span   $\cB/I_x\cB$, and let 
$(U,\varphi, \cG)$
be as in definition \ref{dfn:pathhol}. Then $(U,\varphi,\cG)$ is a bisubmersion.
\end{prop}

\begin{remark}\label{rem:FB}
Recall from \S \ref{sec:twofols} that there are two foliations associated to $\cB$.  
In the proof of proposition \ref{prop:pathhol} we will use the following {relations} between  $(U,\varphi, \cG)$ 
(as in definition \ref{dfn:pathhol})
and bisubmersions for these two foliations: 
\begin{itemize}
\item[a)]  For the singular foliation 
$\cF_{\cB}$ on $M$: $$(U,\bt_U,\bs_U)$$ is a bisubmersion for   $\cF_{\cB}$ (in the sense of \cite[definition 2.1]{AndrSk}). Indeed, it is the ``path-holonomy bisubmersion'' of $\cF_{\cB}$
constructed using the generators $X_i=\rho(\balpha_i)$ of the foliation $\cF_{\cB}$ on $M$. 
This is proven\footnote{Once we establish proposition \ref{prop:pathhol}, an alternative proof is obtained applying lemma \ref{lem:usualbi}.} exactly as in  \cite[Prop. 2.10 a)]{AndrSk}, which holds since
the $X_i$ define a spanning set of $\cF_x:=\cF/I_x\cF$ (it does not matter that they do not define a basis in general).   

\item[b)]
For the singular foliation {$\rar{\cB}$} on $\cG$: a set of local generators of  {$\rar{\cB}$}  is   $\rar{\balpha_1},\dots,\rar{\balpha_n}$, by Lemma \ref{lem:basis}. Let $$(V, \bt_V,\bs_V)$$ be the  path-holonomy bisubmersion {(in the sense of \cite{AndrSk})} associated to these generators near $x$. Notice that, as $M$ embeds in $\cG$ as the identity section, we can view $U\subset \RR^n\times M$ as a subset of $V\subset \RR^n\times \cG$. Restricting the source and target map of $V$ we see that  $\bs_U=\bs_V|_U\colon U\to M$  (the vector bundle projection) and $\varphi=\bt_V|_U\colon U\to \cG$.
\end{itemize}
\end{remark}

\begin{proof}[Proof of proposition \ref{prop:pathhol}]

We show that the three conditions listed in Definition \ref{dfn:bisubm2} hold.

Condition i) holds since $(U,\bt_U,\bs_U)$ is a bisubmersion for the foliation $\cF_{\cB}$, {by Remark \ref{rem:FB} a)}.

In the rest of the proof we use the bisubmersion  $(V,\bt_V,\bs_V)$   for the foliation $\rar{\cB}$ described in Remark \ref{rem:FB} b).
We show the first part of condition ii), \ie, that for any  $\balpha\in \cB$, there is a vector field on $U$
which is $\varphi$-related to  $\rar{\balpha}$. This holds because there exists $\tilde{Z}\in \Gamma(V, \ker d\bs_V)$ which $\bt_V$-projects to $\rar{\balpha}$ (for example by \cite[Prop. 2.10(b)]{AndrSk}, and therefore $\tilde{Z}|_U$ is tangent to $U$ and  $\varphi$-related to $\rar{\balpha}$.

We show the first equation in condition iii). 
We do so using  Lemma \ref{lem:bisubm2} b). 
Let $Z\in \Gamma_c(U, \ker d\bs_U)$, and take an extension $\tilde{Z}\in \Gamma_c(V, \ker d\bs_V)$ to $V$.  
By   proposition \ref{prop:equivbi}
{$(\bt_V,\bs_V)\colon V\to G\times G$ is a bisubmersion for the singular subalgebroid $\rar{\cB}$. Applying  Lemma \ref{lem:bisubm2} b) to it} 
 we see that 
$\tilde{Z}$ is a $C_c^{\infty}(V)$-linear combination of vector fields which are $(\bt_V, \bs_V)$-related to  elements of $\rar{\cB}$. In other words, $\tilde{Z}=\sum \tilde{g_i}\tilde{Y_i}$ where $\tilde{g_i}\in C_c^{\infty}(V)$ and the vector fields $Y_i $ lie in $\Gamma(V,\ker d\bs_V)$ and are $\bt_V$-related to some element of $\rar{\cB}$. In particular,
$(\tilde{Y_i})|_{U}$ is tangent to $U$ and, since $\varphi=\bt_V|_U$, it is  $\varphi$-related to some element of $\rar{\cB}$. Hence $Z=\sum (\tilde{g_i})|_U(\tilde{Y_i})|_{U}$ lies in $\varphi^{-1}(\rar{\cB})$.

{The second parts of conditions ii) and iii) in Definition \ref{dfn:bisubm2} are proven analogously to the above.}
\end{proof}

  {The following Lemma allows to simplify some of the later constructions and proofs, see Rem. \ref{rem:invphbisub}.}
 
\begin{lemma}\label{lem:kappa} Let  $\balpha_1,\dots,\balpha_n \in \cB$ and let $(U,\varphi, \cG)$ be {a path holonomy bisubmersion} 
as in Definition \ref{dfn:pathhol}.
There exists a diffeomorphism $\kappa$ making the following diagram commute:
\begin{equation*}
\xymatrix{
U\ar[rd]^{\varphi}  \ar[rr]^{\kappa} & & U\ar[ld]^{i\circ \varphi} \\
&\cG& }
\end{equation*}
In particular,  $\bs_U\circ \kappa=\bt_U$ and $\bt_U\circ \kappa=\bs_U$.
\end{lemma}
 
\begin{proof} Consider the map
$$\kappa \colon U\to U, \kappa(\lambda,x)=(-\lambda, \bt_U(\lambda,x)).$$ 
One computes easily that $\kappa^2=Id_{U}$, hence $\kappa$ is a diffeomorphism\footnote{The idea of considering $\kappa$ comes from  \cite[Prop. 2.10 a)]{AndrSk}, where $\kappa$ is shown to satisfy
$\bs_U\circ \kappa=\bt_U$ and $\bt_U\circ \kappa=\bs_U$.}.
 
To check that the diagram commutes, we fix $(\lambda,x)\in U$ and compute
$$(i\circ \varphi\circ \kappa)(\lambda,x)=i(\exp_{\bt_U(\lambda,x)}(-\sum \lambda_i\rar{\balpha_i})).$$ By definition we have $\varphi(\lambda,x)=\exp_x(\sum \lambda_i\rar{\balpha_i})$. We need to show that these two points of $\cG$ agree, or equivalently that  
\begin{equation}\label{eq:oneminusone}
\exp_{\bt_U(\lambda,x)}(-\sum \lambda_i\rar{\balpha_i})\cdot \exp_x(\sum \lambda_i\rar{\balpha_i})=1_x.
\end{equation}
Use the short form notation $\balpha:=\sum \lambda_i{\balpha_i}$.
For all $\epsilon\in \RR$ define   section of $\bs$ by $\psi_{\epsilon}\colon M\to \cG, \psi_{\epsilon}(x)=\exp_x(\epsilon  \rar{\balpha})$. Its image defines a bisection, at least for $\epsilon$ is small enough. The right invariance of $\rar{\balpha}$ implies that
the above family of bisections satisfies\footnote{Recall from\cite[Prop. 1.4.2, p.22]{MK2} that the product of bisections is defined as
$(\psi_{\epsilon}\ast \psi_{\sigma})(y)=\psi_{\epsilon}(\bt(\psi_{\sigma}(y)))\cdot \psi_{\sigma}(y)$ for all $y\in M$.} $\psi_{\epsilon}\ast \psi_{\sigma}= \psi_{\epsilon+\sigma}$.
 In particular we obtain
$$(\psi_{-1}\ast\psi_1)(x)=\psi_0(x)=x,$$ which is exactly eq. \eqref{eq:oneminusone}.
\end{proof}

\subsection{Operations involving bisubmersions}\label{subsec:operations}

We explain how to handle bisubmersions algebraically. This will be used the construction of the holonomy groupoid in \S \ref{section:HB}. To this end,  we fix a Lie groupoid $\cG$ and a singular subalgebroid $\cB$ of its Lie algebroid.

{Recall that, given a singular foliation $\cF$, in Prop. \ref{prop:pathhol} we established a bijection between bisubmersions for $\cF$ in the sense of \cite{AndrSk} and 
bisubmersions for $\cF$ regarded as a singular subalgebroid (Def. \ref{dfn:bisubm2}). All the operations we define in this subsection, in the special case of singular foliations,  correspond under the above bijection with the operations introduced in \cite{AndrSk}.}

\subsubsection{Morphisms}

\begin{definition}\label{def:morph}
{Let $(U_i,\varphi_i,\cG)$, $i=1,2$ be bisubmersions for $\cB$.} A {\bf morphism} of bisubmersions is a map $f \colon U_1\to U_2$ such that
$\varphi_1 =\varphi_2\circ f$.
\begin{equation*}
\xymatrix{
U_1  \ar[rd]_{\varphi_1  }  \ar[rr]^{f} & & U_2 \ar[ld]^{\varphi_2} \\
&\cG&\\
  }
\end{equation*}

\end{definition}

There is a simple way to construct new bisubmersions out of old ones, which we will use in the sequel.
\begin{lemma}\label{lem:UV}
Let $(V,\varphi,\cG)$ be a bisubmersion for $\cB$. Let $U$ be a manifold and $p\colon U\to V$ be a  submersion. Then
$(U,\varphi\circ p,\cG)$ is a bisubmersion for $\cB$.
Further,  $p$ is a morphism of bisubmersions.
 \begin{equation*}
\xymatrix{
U  \ar[r]^{p} \ar@{-->}[dr]      &   V \ar[d]^{\varphi}  \\
 & \cG }
\end{equation*}
\end{lemma}
\begin{remark}\label{rem:pq}
 If $q\colon V\to U$ is a section of $p$ (\ie, $p\circ q=Id_V$), then $q$ is also a morphism of bisubmersions. Notice that such a global section might not exist, but local sections of $p$ exist around every point of the open subset $p(U)$ of $V$.
\end{remark}
\begin{proof}
We check that $(U,\varphi\circ p,\cG)$  satisfies the conditions of  definition \ref{dfn:bisubm2}. Since $p$ is a submersion, i) clearly holds, and ii) holds too as $p$  allows to lift vector fields on $V$ to vector fields on $U$. For iii),
 by Lemma \ref{lem:bisubm2} {c)} that it suffices to show that
\begin{equation*}
\cS_U:=\{W\in \Gamma(U,\ker d\bs_U): \text{$W$ is  $(\varphi\circ p)$-related to an element 
{$\rar{\alpha}$}
}\}
\end{equation*} 
{spans $d_u\bs_U$ at every $u\in U$.} 
By assumption,  
\begin{equation*}
\cS_V:=\{Z\in \Gamma(V,\ker d\bs_V): \text{$Z$ is  $\varphi$-related to an element 
{$\rar{\alpha}$}}\}
\end{equation*} 
{spans $d_v\bs_V$ at every $v\in V$.}
 Since $\ker d\bs_U=p_*^{-1}(\ker d\bs_V)$,
taking all lifts (via $p$) of all elements of  $\cS_V$ we obtain a subset of $\cS_U$ which contains $\Gamma_c(\ker(p_*))$ and {spans $d_u\bs_U$ at every $u\in U$}.
 Finally, $p$ is a morphism of bisubmersions by construction.
 \end{proof}

\subsubsection{Inverses}

\begin{definition}\label{def:inv}
Let $(U,\varphi,\cG)$ be a bisubmersion for $\cB$. Its {\bf inverse} is ${\bar{\varphi}:=}i\circ \varphi \colon U \to \cG$, where $i\colon \cG\to \cG, i(g)=g^{-1}$. We denote it $({U},\bar{\varphi}, \cG)$.
\end{definition}

\begin{prop}\label{prop:inv}
The inverse of a bisubmersion is a bisubmersion.
\end{prop}
\begin{proof}
Given a bisubmersion  $(U,\varphi,\cG)$, first notice that $\bar{\bs}_U := \bs\circ \bar{\varphi}  = \bt_U$ and likewise $\bar{\bt}_U := \bt\circ \bar{\varphi}=s_U$ are submersions. For  condition ii) of definition \ref{dfn:bisubm2}, let $\balpha \in \cB$. Since this condition holds for $(U,\varphi,\cG)$, there is $Z\in \vX(U)$ which is $\varphi$-related to $\rar{\balpha}$.
Hence $Z$ is $i\circ \varphi$-related to $i_*\rar{\balpha}=-\lar{\balpha}$, and therefore $-Z$ is $\bar{\varphi}$-related to $\lar{\balpha}$. Similarly, there is $W\in \vX(U)$ which is $\varphi$-related to $\lar{\balpha}$, and $-W$ is $\bar{\varphi}$-related to $\rar{\balpha}$.

For  condition iii) of definition \ref{dfn:bisubm2} notice that the inversion map $i \colon \cG \to \cG$ gives rise to an isomorphism $i_* \colon \rar{\cB} \to \lar{\cB}$, whence $\bar{\varphi}^{-1}(\rar{\cB})=\varphi^{-1}(\lar{\cB})$ and $\bar{\varphi}^{-1}(\lar{\cB})=\varphi^{-1}(\rar{\cB})$. We conclude because $(U,\varphi,\cG)$ satisfies condition iii) of definition \ref{dfn:bisubm2} and  $\bar{\bs}_U = \bt_U, \bar{\bt}_U = \bs_U$.
\end{proof}

\begin{remark}\label{rem:invphbisub}
{If $U$ is a path holonomy bisubmersion, then its inverse bisubmersion is isomorphic to $U$ itself. This follows from Lemma \ref{lem:kappa}.} 
\end{remark}

\subsubsection{Compositions}

\begin{definition}\label{def:comp}
Let $(U_j,\varphi_j,\cG)$ be bisubmersions for $\cB$, $j=1,2$. Their {\bf composition} is
$$m \circ (\varphi_1,\varphi_2) \colon U_1\times_{{\bs_{U_1}},\bt_{U_2}}U_2\to \cG\times_{\bs,\bt}\cG\to \cG$$
where $m$ is the groupoid multiplication on $\cG$ and  $\bs_{U_1} = \bs\circ\varphi_1$ and $\bt_{U_2} = \bt\circ\varphi_2$. We denote it $(U_1 \circ U_2,\varphi_1\cdot\varphi_2,\cG)$.
\end{definition}

\begin{equation*}
\xymatrix{
&U_1\times_{{\bs_{U_1}},\bt_{U_2}}U_2\ar[d]^{(\varphi_1,\varphi_2)} & \\
&\cG\times_{\bs,\bt}\cG \ar[d]^{m}&\\
&\ar[dr]^{\bs}\cG\ar[dl]_{\bt}&\\
M&&M 
}
\end{equation*}

Notice that if the open subsets $\bs_1(U_1)$ and $\bt_2(U_2)$ of $M$ are disjoint, then the composition $U_1 \circ U_2$ is the empty set.
 
\begin{prop}\label{prop:compo}
The  composition of two bisubmersions is a bisubmersion.
\end{prop}
\begin{proof}
Consider two bisubmersions $(U_i,\varphi_i,\cG)$, $i=1,2$. 
 For condition i) of definition \ref{dfn:bisubm2} consider the two natural maps
 $$\bt_{12}, \bs_{12} \colon U_1 \circ U_2 \to M$$ given by 
 $\bt_{12} :=\bt \circ (\varphi_1\cdot\varphi_2)= \bt_{U_1} \circ p_1$ and 
 $\bs_{12} := \bs \circ (\varphi_1\cdot\varphi_2) =
 \bs_{U_2} \circ p_2$, where $p_i\colon U_1 \circ U_2 \to U_i$ are the projections.  They are  both submersions because  the $p_i$ are submersions and because $\bt_{U_1}$ and $\bs_{U_2}$ are submersions. 

For condition ii), let $\balpha \in \cB$ and take a vector field $Z_1\in \vX(U_1)$ which is $\varphi_1$-related to  $\rar{\balpha}$ (it exists since condition
 ii) holds for the bisubmersions $(U_1,\varphi_1,\cG)$). Notice that  $(\rar{\balpha},0)$ restricts to a vector field on $\cG\times_{\bs,\bt}\cG$ which   is $m$-related to $\rar{\balpha}$.
Hence the vector field $(Z_1,0)$ {on $U_1 \circ U_2$} is related by $\varphi_1\cdot \varphi_2=m\circ (\varphi_1, \varphi_2)$ to $\rar{\balpha}$. Similarly, taking $W_2\in \vX(U_2)$ which is $\varphi_2$-related to  $\lar{\balpha}$, we see that $(0,W_2)$ is $\varphi_1\cdot \varphi_2$-related to $\lar{\balpha}$.

Now we prove that $(U_1 \circ U_2,\varphi_1\cdot\varphi_2,\cG)$ satisfies condition iii) of definition \ref{dfn:bisubm2}. For simplicity we'll show the first of the equalities appearing there, namely $$(\varphi_1\cdot\varphi_2)^{-1}(\rar{\cB})=\Gamma_c(U_1\circ U_2;\ker d\bs_{12})$$ (the second is proven similarly). 
First notice that there are distinguished elements of $\Gamma_c(U_1\circ U_2;\ker d\bs_{12})$:
\begin{enumerate}
\item[(*)] $(W_1,Z_2)\in \vX(U_1\circ U_2)$ such that $W_1$ is $\varphi_1$-related to $i_*\rar{\balpha}$, and  $Z_2$ is $\varphi_2$-related to $\rar{\balpha}$, for some $\balpha\in \cB$,
\item[(**)] $(Z_1,0)$ such that $Z_1$ is  $\varphi_1$-related to $\rar{\balpha}$ for some $\balpha\in \cB$.
\end{enumerate}
{Notice that 
both families  $(*)$ and $(**)$ consist of vector fields which  are $\varphi_1\cdot\varphi_2$-related to $\rar{\alpha}$ for some $\balpha\in \cB$.
Indeed, for any $\balpha\in \cB$, $(i_*\rar{\balpha},\rar{\balpha})$ is $m$-related to the zero vector field on $\cG$, and 
$(\rar{\balpha},0)$ is $m$-related to $\rar{\balpha}$, where $m$ is the multiplication.
Hence Lemma \ref{lem:bisubm2} c), together with the following claim, finishes the proof.}

\underline{Claim:}
\emph{The union of the   families  of vector fields $(*)$ and $(**)$, evaluated at any point  $(u_1,u_2)\in U_1\circ U_2$, span the kernel of $d\bs_{12}$ at $(u_1,u_2)$.}\\
Fix a vector in the  kernel of $d\bs_{12}$, that is, $(X_1,X_2)\in T_{u_1}U_1\times  T_{u_2}U_2$ such that
$$d\bs_{U_1}(X_1)=d\bt_{U_2}(X_2) \text{ and }d\bs_{U_2}(X_2)=0.$$
Since $(\varphi_2)^{-1}(\rar{\cB})=\Gamma_c(\ker d\bs_{U_2})$, {by  Lemma \ref{lem:bisubm2} c)} we can extend $X_2$ to a vector field $Z_2\in \vX(U_2)$ which is $\varphi_2$-related to
$\rar{\balpha}$ for some $\balpha\in \cB$. Since $U_1$ satisfies condition ii) in definition \ref{dfn:bisubm2}, there is $W_1\in \vX(U_1)$ which is $\varphi_1$-related to $i_*\rar{\balpha}=-\lar{\balpha}$. Notice that $(W_1,Z_2)$ belongs to the  family of vector fields $(*)$.

Further, the map $d_{u_1}\bs_{U_1}$ sends both $X_1$ and $(W_1)|_{u_1}$  to the same vector $d_{u_2}\bt_{U_2}(X_2)$. This means that  $X_1-(W_1)|_{u_1}$ lies in $\ker d_{u_1}\bs_{U_1}$, and therefore 
can be extended to a vector field $Z_1$ lying in $ \Gamma(U_1,\ker d\bs_{U_1})$. Since $(\varphi_1)^{-1}(\rar{\cB})=\Gamma_c(U_1,\ker d\bs_{U_1})$, {by  Lemma \ref{lem:bisubm2} c)} we can choose $Z_1$ so that it is $\varphi_1$-related to {some $\rar{\alpha'}$}, \ie so that $(Z_1,0)$ belongs to the  family of vector fields $(**)$. Now  $$(X_1,X_2)=((W_1)|_{u_1},X_2)+(X_1-(W_1)|_{u_1},0)=
\left((W_1,Z_2)+(Z_1,0)\right)|_{(u_1,u_2)},$$
proving the claim.
\end{proof}

\begin{ex}\label{G2bisub} Let $\cB=\Gamma_c(A)$, where $A=Lie(\cG)$. We saw in example \ref{ep:groid} that    $(\cG,Id_\cG,\cG)$ is a  bisubmersion for $\cB$. From proposition  \ref{prop:compo} we deduce that
$(\cG\times_{\bs,\bt}\cG, m, \cG)$ is also a  bisubmersion for $\cB$. The multiplication $m\colon \cG\times_{\bs,\bt}\cG \to \cG$ is a morphism of bisubmersions.
\end{ex}

\subsubsection{Bisections}

\begin{definition}\label{def:bisec}
Let $(U,\varphi,\cG)$ be a bisubmersion for $\cB$.
\begin{enumerate}
\item A \textbf{bisection} of $(U,\varphi,\cG)$ is a locally closed submanifold $\bb$ of $U$  such that the restrictions of both $\bs_U$ and $\bt_U$ to $\bb$ are diffeomorphisms from $\bb$ onto open subsets of $M$.

\item Let $u \in U$ and $\bc$ a bisection of the Lie groupoid $\cG$ with $\varphi(u)\in \bc$.
 We say that $\bc$ is carried by $(U,\varphi,\cG)$ at $u$ if there exists a bisection $\bb$ of $U$ such that $u \in \bb$ and $\varphi(\bb)$ is an open subset of $\bc$.
\end{enumerate}
\end{definition}

{Notice that if $\bb$ is a bisection of $(U,\varphi,\cG)$, then 
$\varphi(\bb)$ is a bisection of the Lie groupoid $\cG$, and further $\varphi|_{\bb}\colon \bb\to \varphi(\bb)$ is a diffeomorphism.
This is best seen viewing bisections of $(U,\varphi,\cG)$ as images of   maps $b$ which are sections of $\bs_U\colon U\to M$ so that $\bt_U \circ b$ is local diffeomorphism of $M$.
}

The existence of bisections at every $u \in U$ of a bisubmersion $(U,\varphi,\cG)$ is proven exactly like in \cite[Prop. 2.7]{AndrSk}. Notice that the proof uses only the fact that $\bt_U, \bs_U\colon U \to M$ are submersions. {The following proposition is given without proof.}
 
\begin{prop}
Let $(U,\varphi,\cG)$ and $(U_i,\varphi_i,\cG)$, $i=1,2$ be bisubmersions.
\begin{enumerate}
\item Let $u \in U$ and $\bc$  a local  bisection of $\cG$ carried by $(U,\varphi,\cG)$ at $u$. Then $\bc^{-1}$ is carried by the inverse bisubmersion $({U},\bar{\varphi},\cG)$ at $u$.
\item Let $u_i \in U_i$, $i=1,2$ be such that $\bs_{U_1}(u_1)=\bt_{U_2}(u_2)$ and let $\bc_i$ be local  bisections of $\cG$ carried by $(U_i,\varphi_i,\cG)$ at $u_i$ respectively, $i=1,2$. Then $\bc_1\cdot\bc_2$ is carried by the composition $(U_1\circ U_2,\varphi_1 \cdot \varphi_2,\cG)$ at $(u_1,u_2)$.
\end{enumerate}
\end{prop}

\subsubsection{{Modifying a bisubmersion by a bisection}}\label{modify}
 
Let $(U, \bt_U, \bs_U)$  be a bisubmersion for a  {singular foliation} $(M,\cF)$ (in the sense of \cite{AndrSk}, see \S \ref{section:usualbisub}), and $\bb$ a bisection  of $U$.
An easy but important fact is that the diffeomorphism carried by the bisection, namely $\phi\colon M\to M, \bs_U(u)\mapsto \bt_U(u)$  where $u\in \bb$,  is an \emph{automorphism}\footnote{To see this, 
notice that  $(\bt_U)|_{\bb}$ is an isomorphism of foliated manifolds between
 the submanifold $\bb$ -- endowed  with the restriction of the pullback foliation 
 $\bt_U^{-1}(\cF)$ -- and $(M,\cF)$. The same holds for $\bs_U$ in place of $\bt_U$, and the results follows since 
$\bt_U^{-1}(\cF)=\bs_U^{-1}(\cF)$.}
 of the singular foliation $\cF$, \ie $\phi_* \cF=\cF$. This fact was used in \cite[\S 2.3]{AndrSk} to obtain a new bisubmersion for $\cF$ out of $(U, \bt_U, \bs_U)$ and $\bb$.
 
We now derive an analog statement for bisubmersions of singular subalgebroids. This will be used in the proof of Corollary \ref{cor:crucial} and in Lemma \ref{lem:sm1}.

\begin{prop}\label{prop:preserveb}
Let   $(U,\varphi,\cG)$ be a bisubmersion for the singular subalgebroid $\cB$, and let $\bb$ be a bisection of $U$. Denote by $\bc:=\varphi(\bb)$ the induced bisection of the Lie groupoid $\cG$, and by $L_{\bc}\colon G\to G$ the
left multiplication by $\bc$. Denote by $\bc^{-1}$ the image of $\bc$ under the inversion map. Then

\begin{itemize}
\item [a)] $(L_{\bc})_*\rar{\cB}=\rar{\cB}$. 
\item [b)] $(U,L_{\bc^{-1}}\circ \varphi,\cG)$ is also a  bisubmersion for $\cB$. 
\end{itemize}

\end{prop}

\begin{remark}\label{rem:rephrasea} Item a) above can be rephrased saying that the singular subalgebroid $\cB$ is preserved by the Lie algebroid automorphism of 
$A=\ker(\bs_*|_M)$ induced by $\bc$.
Recall that the conjugation by the bisection $\bc$ is a Lie groupoid automorphism of $G$, which differentiates to the Lie algebroid automorphism $A\to A, a_x\mapsto (R_{{\bc(x)}^{-1}})_*(L_{\bc})_*(a_x)$, where  $\bc(x)$ is the unique point of $\bc$ with source $x$. The fact that this Lie algebroid automorphism preserves  $\cB$ is an immediate consequence of Prop. \ref{prop:preserveb} a), upon using the facts that 
vector fields of the form $(L_{\bc})_*(\rar{\balpha})$ are right-invariant and that   the right invariant vector fields in the global hull of $\rar{\cB}$  are exactly the right-translates of elements of $\cB$. 
\end{remark}
 
 \begin{proof}
a) 
We view the bisection $\bb$ as a map $\bb\colon M\to U$ which is a right-inverse to $\bs_U$. 
In Prop. \ref{lem:bisubrarb} we established that  $\widehat{U}:=U \times_{\bs_U, \bt}G$, with the target and source maps indicated there, is a bisubmersion (in the sense of \cite{AndrSk}) for the singular foliation $\rar{\cB}$ on $G$.
From $\bb$ we obtain a bisection of  $\widehat{U}$, namely $$\widehat{\bb}\colon G\to U \times_{\bs_U, \bt}G,\;\; g \mapsto (\bb(\bt(g)),g).$$
By the text just before Prop. \ref{prop:preserveb}, the diffeomorphism $\bt_{\widehat{U}}\circ \widehat{\bb}\colon G\to G$   carried by  this bisection 
is an automorphism of the singular foliation $\rar{\cB}$. This diffeomorphism
reads $g\mapsto \bc(\bt(g))\cdot g$, \ie it is exactly $L_{\bc}$.

b) We   check that the fact that $(U,\varphi,\cG)$ is a bisubmersion for $\cB$ implies that $(U,L_{\bc^{-1}}\circ \varphi,\cG)$ satisfies 
 the three conditions in Def. \ref{dfn:bisubm2}. 
Condition i) holds because 
$\bs\circ L_{\bc^{-1}} =\bs$, and $\bt\circ L_{\bc^{-1}}=\phi \circ\bt$ for  
some diffeomorphism $\phi$ of $M$.

For ii), let $\balpha \in \cB$. By item a) (or more precisely Remark \ref{rem:rephrasea}) we have $(L_{\bc})_*\rar{\balpha}=\rar{\balpha'}$ for some $\balpha'\in \cB$.
Hence any  $Z\in\vX(U)$
which is $\varphi$-related to $\rar{\balpha'}$ is also $(L_{\bc^{-1}}\circ \varphi)$-related to $\rar{\balpha}$. The second part of condition ii) holds trivially since $(L_{\bc})_*\lar{\balpha}=\lar{\balpha}$. 

Finally, the first part of condition iii) holds because 
$(L_{\bc^{-1}}\circ \varphi)^{-1}(\rar{\cB})=\varphi^{-1} (L_{\bc})_*(\rar{\cB})=\varphi^{-1} (\rar{\cB})$, where the last equation holds by item a). The second part of condition iii) holds similarly.
\end{proof}

{\begin{remark}\label{rem:ccinv} A variation of Prop. \ref{prop:preserveb} b) is the following: under
the same hypotheses of the proposition,  $(U,L_{\bc}\circ \varphi,\cG)$ is also a  bisubmersion for $\cB$. Indeed, since  $(L_{\bc})^{-1}=L_{\bc^{-1}}$,  Prop. \ref{prop:preserveb} a) implies that $(L_{\bc^{-1}})_*\rar{\cB}=\rar{\cB}$, and the proof of Prop. \ref{prop:preserveb} b) gives the claimed result.\end{remark}
}

\section{The Holonomy Groupoid}\label{section:HB}

In this whole section we
fix an integrable Lie algebroid $A\to M$ and a singular subalgebroid $\cB$. Further, we fix a Lie groupoid $\cG$ integrating $A$.

We give the construction of the holonomy groupoid associated with a singular subalgebroid $\cB$,  relying on the methods developed in \cite{AndrSk}.
{In particular, our \S \ref{subsec:compare} and \S \ref{sub:con} follow closely \cite{AndrSk}. A new feature is that the holonomy groupoid depends on the choice of $G$; in \S \ref{section:vary} we describe this dependence.}

\subsection{Comparison of bisubmersions}\label{subsec:compare}
 
We start with a technical result,  needed in the proof of Proposition \ref{prop:crucial}. 

\begin{lemma}\label{lem:linindip}
Let $(U,\varphi,\cG)$ be a bisubmersion for the singular subalgebroid $\cB$, $u\in U$, and  ${\balpha}_1,\dots,{\balpha}_n\in \cB$ which 
{induce a linearly independent set of vectors in $\cB/I_{\bt_U(u)}\cB$.}
Let $Y_1,\dots,Y_n \in \Gamma(U,\ker d\bs_U)$ such that 
$Y_i$ is $\varphi$-related to $\rar{\balpha}_i$ for every $i=1,\dots,n$.    
Then $Y_1(u),\dots,Y_n(u)$ are linearly independent.
\end{lemma}

 \begin{proof}
The existence of the $Y_i$ follows from Definition \ref{dfn:bisubm2} ii). We show that they are linearly independent at $u$.

First, recall from Rem. \ref{donto} that there is a well-defined map of $C^{\infty}(U)$-modules $d\varphi \colon \Gamma_c(U;\ker d\bs_U) \to \varphi^*(\rar{\cB})$. Denote by $\bt\colon G\to M$ the target map of the Lie groupoid. Upon the identification between the pullback vector bundle $\bt^*(\ker d\bs|_M)$ and $\ker d\bs$ given by right-translation, we have $\bt^* \cB=\rar{\cB}$. Hence, since $\bt_U=\bt\circ \varphi$, the above map can be written as $d\varphi \colon \Gamma_c(U;\ker d\bs_U) \to \bt_U^*{\cB}$.

Take constants $\lambda_i$ such that $\sum  \lambda_iY_i$ vanishes at $u$. We need to show that   $\lambda_i=0$ for all $i$.  By definition \ref{dfn:bisubm2} iii) and Lemma \ref{lem:bisubm2} c)
there are elements  $W_1,\dots,W_k\in \Gamma(U,\ker d\bs_U)$ which form a frame for $\ker d\bs_U$ over a neighborhood $U_0$ of $u$, and  which are $\varphi$-related to elements {$\rar{\boldsymbol{\beta}_j}$}, for some $\boldsymbol{\beta}_j\in \cB$.

On $U_0$ we have 
$\sum \lambda_iY_i=\sum f_jW_j$
for some  $f_j\in I_u\subset C^{\infty}(U_0)$.
Applying $d\varphi$ to this equation we obtain 
\begin{equation}\label{eqn:F2}
\sum \lambda_i \;\bt_U^*({\balpha_i})=\sum f_j\;\bt_U^*({\boldsymbol{\beta}_j})\end{equation}

 Choose a (locally defined) section $\tau$ of the submersion $\bt_U\colon U\to  M$. 
Notice that the l.h.s. of eq. \eqref{eqn:F2} is the pullback by $\bt_U$ of an element of $\cB$, namely $\sum \lambda_i  \balpha_i$, hence the above expression is determined by its value on the image $Im(\tau)$. Therefore the value of  the  r.h.s. of
eq. \eqref{eqn:F2} is unchanged 
if we replace the coefficients $f_j$ with $\bt_U^*F_j$, where $F_j=\tau^*f_j \in C^{\infty}(M)$.
Hence we have the following equality of elements of $\cB$:
\begin{equation*} 
\sum \lambda_i  {\balpha_i}=\sum F_j {\boldsymbol{\beta}_j}. 
\end{equation*}
Since $F_j\in I_{\bt_U(u)}$, the image of this element in $\cB/I_{\bt_U(u)}\cB$ vanishes. But the image is $\sum \lambda_i  [{\balpha_i}]$, and from the linear independence of the 
$[{\balpha_i}]$, we conclude that $\lambda_i=0$ for all $i$.
\end{proof}
 
{For the following fundamental result, recall that the minimality of a set of generators is defined in Rem. \ref{rem:basis}.}

\begin{prop}\label{prop:crucial}
Let $x \in M$ and $\balpha_1,\dots,\balpha_n
\in \cB$ which form
a minimal set of generators of $\cB$ around $x$.
Let  $(U_0,\varphi_0,\cG)$ be the {path holonomy} bisubmersion they define {(see Def. \ref{dfn:pathhol})}. Let $(U,\varphi,\cG)$ be a bisubmersion of $\cB$ and suppose that $u \in U$ with $\varphi(u)=1_x$ carries the identity bisection $1_M$ of $\cG$. 
 
Then there exists an open neighborhood $U'$ of $u$ in $U$ and a submersion $g\colon U' \to U_0$ which is a morphism of bisubmersions and $g(u)=(0,x)$.
\begin{equation*}
\xymatrix{
U'\ar[rd]_{ \varphi} \ar@{-->}[rr]^{g} & & U_0\ar[ld]^{ \varphi_0} \\
&\cG& }
\end{equation*}
\end{prop}
\begin{proof}
Replacing $U$ by an open subset containing $u$, we may assume $\bs_U(U) \subset \bs_{U_0}(U_0)$. By Lemma \ref{lem:linindip} there are $Y_1,\dots,Y_n \in \Gamma(U,\ker d\bs_U)$ such that 
$Y_i$ is $\varphi$-related to $\rar{\balpha}_i$ for every $i=1,\dots,n$, and
 the $Y_1(u),\dots,Y_n(u)$ are linearly independent. Let $Z'_{n+1},\dots,Z'_{k} \in \Gamma(U,\ker d\bs_U)$ such that $(Y_1,\dots,Y_n,Z'_{n+1},\dots,Z'_k)$ is a {frame of $\ker d\bs_U$} nearby $u$. Consider as in    Remark \ref{donto} the map $d\varphi\colon \Gamma_c(U;\ker d\bs_U) \to \varphi^*(\rar{\cB}).$  
For all
 $i=n+1,\ldots,k$ consider also $d\varphi(Z'_i)\in \varphi^*(\rar{\cB})$.
Since $\varphi^*(\rar{\cB})$ is generated by  $\{\varphi^*\rar{\balpha}_j\}_{j=1,\dots,n}$ nearby $u$,
 there exist functions $f_i^j$ nearby $u$ such that
$d\varphi(Z'_i)=d\varphi(\xi_i)$, where
$\xi_i := \sum_{j=1}^n f_i^j Y_j$. Put $Z_i = Z'_i - \xi_i$.
 Then $(Y_1,\dots,Y_n,Z_{n+1},\dots,Z_k)$ is also a {frame}  nearby $u$, since the $\xi$'s are linear combinations of the $Y_i$'s, and further
\begin{equation}\label{eq:nk}
\text{$\varphi_*(Y_i)=\rar{\balpha}_i$ for $i\leq n\;\;\;\;\;$ $\varphi_*(Z_i)=0$ for $i>n$}.
\end{equation}
To unify the notation, denote $Y_i:=Z_i$ for $i>n$. 

   For $\lambda = (\lambda_1,\dots,\lambda_k) \in \RR^k$ small enough, we denote by $\psi_{\lambda}$ the partially defined diffeomorphism $\exp(\sum_{i=1}^k\lambda_i Y_i)$ of $U$. {Denote by $\bb$ a bisection of $U$ through $u$ carrying the identity bisection of $G$}.
There is an open neighborhood $\bb' \subset \bb$ of $u$ and an open ball $B^k$ in $\RR^k$ such that $$h \colon \bb' \times B^k \to U',\quad (y,\lambda) \mapsto \psi_{\lambda}(y)$$ is a diffeomorphism of $\bb' \times B^k$ into an open neighborhood $U'$ of $u$ in $U$.  Notice that for all $y\in \bb'$ we have 
\begin{equation}\label{eq:varphipsi}
\varphi(\psi_{\lambda}(y))=\exp_{\varphi(y)}(\sum_{i=1}^{{n}}\lambda_i \rar{\balpha}_i),
\end{equation}
{where the sum runs only until $n$} as a consequence of eq. \eqref{eq:nk}.
Let $p \colon \RR^k \to \RR^n$ be the projection to the first $n$ coordinates. 
{Use the diffeomorphism $\varphi|_{\bb'}$ to identify $\bb'$ with an open subset of $M$ (thereby changing the domain of $h$).}
We define $$g := p \circ h^{-1} \colon U' \to   U_0.$$ 
The map $g$ is a morphism of bisubmersions by eq. \eqref{eq:varphipsi}, and it is a submersion.
\end{proof}

Corollary \ref{cor:crucial} below 
{allows to define the equivalence relation giving rise to the holonomy groupoid.}

\begin{cor}\label{cor:crucial}
Let $(U_i,\varphi_i,\cG)$, $i=1,2$ be bisubmersions of $\cB$ and $u_i \in U_i$ such that $\varphi_1(u_1) = \varphi_2(u_2)$. 
\begin{enumerate}
\item If the identity bisection $1_M$ of $G$ is carried by $U_i$ at $u_i$, for $i=1,2$, there exists an open neighborhood $U'_1$ of $U_1$ and a morphism of bisubmersions $f\colon U'_1 \to U_2$ such that $f(u_1)=u_2$.
\item If there is a bisection of $\cG$ carried by both $U_1$ at $u_1$ and by $U_2$ at $u_2$, there exists an open neighborhood $U'_1$ of $u_1$ in $U_1$ and a morphism of   bisubmersions $f \colon U'_1 \to U_2$ such that $f(u_1)=u_2$.
\item If there is a morphism of   bisubmersions $g \colon U_1 \to U_2$ such that $g(u_1)=u_2$, then there exists an open neighborhood  $U'_2$ of $u_2$ in $U_2$  and a morphism of   bisubmersions $f \colon U'_2 \to U_1$ such that $f(u_2)=u_1$.
\end{enumerate}
\end{cor}
\begin{proof}
Given proposition \ref{prop:crucial}, a) is proven exactly as in \cite[Cor. 2.11]{AndrSk} a).

b) By assumption there are bisections  $\bb_i$ of $U_i$ through $u_i$ ($i=1,2$) and a bisection $\bc$ of $\cG$ through $\varphi_1(u_1) = \varphi_2(u_2)$ 
such that  $\varphi_1(\bb_1)$ and 
$\varphi_2(\bb_2)$ are open subsets of $\bc$.
{By Prop. \ref{prop:preserveb},  $(U_i,L_{\bc^{-1}}\circ \varphi_i,\cG)$ are
  bisubmersions for $\cB$}, where
$L_{\bc^{-1}}$ denotes the automorphism of $\cG$ given by left multiplication with the bisection $\bc^{-1}:=\{g^{-1}:g\in \bc\}$.
The points $u_i$ of these bisubmersions carry $1_M$, for
the $\bb_i$ are bisections of these bisubmersions whose images in $\cG$ are contained in $1_M$. By  a) we therefore obtain a map $f$ making the following diagram commute:
\begin{equation*}
\xymatrix{
U_1\ar[rd]_{L_{\bc^{-1}}\circ \varphi_1}  \ar@{-->}[rr]^{f} & & U_2\ar[ld]^{L_{\bc^{-1}}\circ \varphi_2} \\
&\cG\ar[d]^{L_{\bc}}&\\
&\cG& }
\end{equation*}
Therefore $f$ is the desired morphism.

c) Let $\bb$ be a bisection of $U_1$ through $u_1$. Then $g(\bb_1)$ is a bisection of $U_2$ through $u_2$. Both $\bb$ and $g(\bb_1)$ carry the same bisection of $\cG$, that is, $\varphi_2(g(\bb))=\varphi_1(\bb)$. Hence we can apply item b) to obtain the existence of $f$.
\end{proof}

\subsection{Construction of the holonomy groupoid}\label{sub:con}

Recall that we fixed an integrable Lie algebroid $A\to M$, a singular subalgebroid $\cB$, and a Lie groupoid $\cG$ integrating $A$.
\begin{definition}\label{def:pathholatlas}
 Consider a family $(U_i,\varphi_i,\cG)_{i \in I}$ of  {source-connected} {minimal} path-holonomy bisubmersions defined as in definition \ref{dfn:pathhol} such that $M = \cup_{i \in I}\bs_{U_i}(U_i)$. Let $\mathcal{U}$ be the collection of all such bisubmersions, together with their inverses
  and finite compositions. {(We can omit the inverses, by Remark. \ref{rem:invphbisub}).}
     We call $\mathcal{U}$ a {\bf path holonomy atlas} of $\cB$.
\end{definition}

Corollary \ref{cor:crucial} c) shows that for $u_1 \in (U_1,\varphi_1,\cG), u_2 \in (U_2,\varphi_2,\cG)$ the relation 
\begin{align*}
u_1 \sim u_2 \Leftrightarrow &\;\;\text{there is an open neighborhood  $U'_1$ of $u_1$,}\\
&\;\;\text{there is a morphism of bisubmersions } f\colon U_1' \to U_2 \text{  such that } f(u_1)=u_2   
\end{align*}
is an equivalence relation. This allows us to give the following definition:

\begin{definition}\label{def:holgroupoid} Let $\cG$ be a  Lie groupoid and $\cB$   a singular subalgebroid of $Lie(\cG)$. The  {\bf holonomy groupoid of $\cB$ over $\cG$} is 
\begin{center}
\fbox{\begin{Beqnarray*}
\;\;\;\; H^{\cG}(\cB):= \coprod_{U \in \mathcal{U}} U/\sim
 \end{Beqnarray*}}
\end{center} 
We write $H(\cB)$ instead of $H^{\cG}(\cB)$ when the choice of $\cG$ is understood.
\end{definition} 

\begin{remark}\label{rem:HGBbisec} The equivalence relation $\sim$ can be made more explicit as follows, as a consequence of Cor. \ref{cor:crucial} b):
\begin{align*}
u_1\sim u_2 \Leftrightarrow &\;\text{
$\varphi_1(u_1)=\varphi_2(u_2)$}, \\
&\;\text{$\exists$   bisections $\bb_1$ through $u_1$,   $\bb_2$ through $u_2$, s.t. $\varphi_1(\bb_1)=\varphi_2(\bb_2)$. 
}
\end{align*}
\end{remark}

{Denote by $\natural \colon \coprod_{U \in \mathcal{U}} U \to H^{\cG}(\cB)$ the quotient map.}

\begin{remark}\label{rem:open}
{In the following, we endow $H^{\cG}(\cB)$ with the quotient topology induced by $\natural$. For any bisubmersion $U\in \cU$, the image $\natural U$ is open in $H^G(\cB)$, by the very same argument used at the beginning of \cite[\S 3.4]{AZ1}.}
\end{remark}
The next proposition justifies the use of the term ``groupoid'' for $H^G(\cB)$.

 \begin{thm}\label{thm:holgroidconstr}
Denote $q_U:=\natural|_{U}$ for all $U\in \cU$.
\begin{enumerate}
\item There are maps $\bs_H, \bt_H\colon H^{\cG}(\cB) \to M$ such that $\bs_H\circ q_U=\bs_U$ and $\bt_H\circ q_U=\bt_U$ for all $U\in \cU$.
\item There is a topological groupoid structure on $H^{\cG}(\cB)$, with objects $M$, source and target maps $\bs_H, \bt_H$ defined above, and with multiplication  $q_U(u)q_V(v):=q_{U\circ V}(u,v)$.
\item  The canonical map  
\begin{center}
\fbox{\begin{Beqnarray*}
\;\;\;\; \Phi \colon H^{\cG}(\cB)\to \cG,
 \end{Beqnarray*}}
\end{center}
determined by $\Phi\circ q_U=\varphi_U$ for all $U\in \cU$,
is a morphism of topological groupoids covering $Id_M$.
\end{enumerate}
\end{thm}

\begin{proof} 
a)  First notice that the map $\Phi$ introduced in c) is well-defined, by   the definition of morphism of bisubmersions (definition \ref{def:morph}). Hence $\bs_H:=\bs \circ \Phi$ and $\bt_H:=\bt \circ \Phi$ are well-defined maps $H^{\cG}(\cB) \to M$. They clearly  satify  $\bs_H\circ q_U=\bs_U$ and $\bt_H\circ q_U=\bt_U$ for all $U\in \cU$.

b) We prove that the multiplication is well-defined. Let $U,V,U',V'\in \cU$ and consider elements satisfying $q_U(u)=
q_{U'}(u')$ and $q_V(v)=
q_{V'}(v')$. Then there is exist local morphisms of bisubmersions $f\colon U\to U'$ with $u\mapsto u'$, and  $h\colon V\to V'$ with $v\mapsto v'$. Assume that $\bs_U(u)=\bt_V(v)$, which implies   $\bs_{U'}(u')=\bt_{V'}(v')$.
Since morphisms of bisubmersions preserve the source and target maps,
the product map restricts to a well-defined map $(f,h)\colon U\circ V\to  U'\circ V'$ with $(u,v)\mapsto(u',v')$ and which is a morphism of bisubmersions, showing that 
$(u,v)\sim(u',v')$ and therefore $q_{U\circ V}(u,v)=q_{U'\circ V'}(u',v')$.

{For any $x\in M$, the identity element $1_x\in H^{\cG}(\cB)$ is represented by any point $u$ in a bisubmersion $U\in \cU$ with $\bs_U(u)=\bt_U(u)=x$ and so that $u$ carries (locally) the identity bisection of $\cG$. (For instance, one can take $U\subset \RR^n\times M$ to be a minimal path holonomy bisubmersion and $u=(0,x)$.). The inverse of $q_V(v)\in H^{\cG}(\cB)$ is $q_{\bar{V}}(v)$, where $\bar{V}$ denotes the inverse bisubmersion to $V$ as in Def. \ref{def:inv}.} It is clear that, with these operations, $H^{\cG}(\cB)$ forms a topological groupoid.

c) The map $\Phi$ is a morphism of topological groupoids, by the definition of inverse and composition of bisubmersions (definitions \ref{def:inv} and \ref{def:comp}).
\end{proof}

\begin{remark}
In Appendix \ref{sec:atlases} we define the notion of atlas for singular subalgebroids, of which   $\cU$ appearing in definition \ref{def:pathholatlas} is an example, and from each atlas we construct a topological groupoid. \end{remark}

{
{Let us work out by hand an elementary example.}
More classes of examples will be given in \S \ref{subsec:exholgr}.}

\begin{ex}[Lie algebras]\label{ex:Liegr}
Let $\cB=A=\g$ a Lie algebra and $G$ any connected Lie group integrating $\g$. There is a neighborhood $U$ of the origin in $\g$ 
such that the exponential map $U \to G$ is a bisubmersion. Indeed, for any
 basis of $\g$, consider the induced path-holonomy bisubmersion $\RR^n \to G$; upon the identification $\g \cong \RR^n$ given by the basis, it is the exponential map, as the latter is obtained taking integral curves starting at the identity of right-invariant vector fields. The $n$-fold composition of this bisubmersion is $$U^{\times n}:=U\times\dots\times U \to G, \;\;\;(v_1,\dots,v_n)\mapsto \exp(v_1)\dots \exp(v_n).$$ 
The map $\cup_n U^{\times n} \to G$ is surjective and, by the definition of holonomy groupoid, it descends to an injective map $H^G(\g)\to G$ . We conclude that $H^G(\g)$ is isomorphic (as a topological groupoid)  to $G$.
\end{ex}

\subsection{Examples}\label{subsec:exholgr}

We give some examples of the holonomy groupoid defined in \S \ref{sub:con}.  We do so for the two basic examples of singular subalgebroid displayed in \S  \ref{subsec:motivex}:
singular foliations 
and wide Lie subalgebroids. 
{For wide Lie subalgebroids we show that  $H^{\cG}(\cB)$ agrees with $H_{min}$, the  minimal integral of $B$ over $\cG$ of Moerdijk and Mr{\v{c}}un recalled in the Introduction.} 
Wide Lie subalgebroids are treated as a special case of 
  images of Lie algebroid morphisms covering diffeomorphisms of the base (\S\ref{subsec:morph}).
 {Unfortunately at the moment we have no way to describe explicitly the holonomy groupoids for the other classes of singular subalgebroids displayed in \S \ref{subsec:ex}.}

  \subsubsection{For singular foliations}
\begin{ex}[Singular foliations]\label{ex:singfo}
When $A=TM$ (so $\cB$ is a singular foliation on $M$), and $\cG=M\times M$, then $H^{M\times M}(\cB)$ is the holonomy groupoid of the singular foliation as defined in \cite{AndrSk}. {This follows from Ex. \ref{ex:sfolpair}  and comparing the construction in \S \ref{sub:con} to the one of \cite{AndrSk}.}

In example \ref{ex:regfol} we will take $\cG=\Pi(M)$ (the fundamental groupoid of $M$, {which is   source simply connected}) and construct the topological groupoid $H^{\Pi(M)}(\cB)$. We will see that $H^{\Pi(M)}(\cB)$ is not   source simply connected in general, and has $H^{M\times M}(\cB)$ as a quotient.
\end{ex}

 \subsubsection{For singular subalgebroids arising from Lie groupoid morphisms}

The next proposition will allow us to construct holonomy groupoids for several classes of singular subalgebroids.  
It is based on the ideas explained in \cite[Ex. 3.4(4)]{AndrSk}.

\begin{prop}\label{prop:B} Let $\cG$ be a Lie groupoid over $M$ and $\cB$ a singular subalgebroid of $Lie(\cG)$. Let $\cK$ be a 
Lie groupoid  over $M$. Let
$\varphi\colon \cK \to \cG$ be a morphism of Lie groupoids  covering $Id_M$ which is also a bisubmersion for  $\cB$. Then:
\begin{itemize}
\item [i)]  $H^{\cG}(\cB)=\cK/\cI$ as topological groupoids, where  $$\cI:=\{k\in \cK: \text{$\exists$ a local bisection $\bb$ through $k$ such that $\varphi(\bb)\subset 1_M$}\}$$
\item [ii)] {the canonical map $\Phi\colon H^{\cG}(\cB)\to \cG$ coincides with the map $\cK/\cI\to \cG$ induced by $\varphi$.}
\end{itemize}

\end{prop}

\begin{remark}\label{rem:cI} We explain the notation and terminology in the definition of $\cI$ in the above proposition.
The symbol $1_M$ denotes the set of identity elements of the Lie groupoid $\cG$. The term ``bisection'' refers to bisection for the Lie groupoid $\cK$.
This coincides with the notion of 
bisection for the bisubmersion
$(\cK,\varphi,\cG)$ (in the sense of definition \ref{def:bisec}) 
since the morphism of Lie groupoids $\varphi$ covers $Id_M$. 

The set $\cI$ is a   normal  subgroupoid of $\cK$. This means that $\cI$ is a wide subalgebroid and is contained in the union of the isotropy subgroups of $\cK$ ({a consequence of  $\cI\subset \ker(\varphi)$}). Further, it means that $\cI$ is invariant under conjugation: for every $k\in \cK$  we have $k\cI_{y}k^{-1}\subset \cI_{x}$ where $y=\bs_{\cK}(k)$ and $x=\bt_{\cK}(k)$. The quotient $\cK/\cI$ has a unique topological groupoid structure such that the quotient map $\cK \to \cK/\cI$ is a groupoid morphism (see \cite[Prop. 2.2.3]{MK2}).
\end{remark}

\begin{remark}\label{rem:bisec}
Every normal subgroupoid $\cI$ of $\cK$ corresponds to the equivalence relation $$k_1\sim_{\cK}k_2 \Leftrightarrow k_1 (k_2)^{-1}\in \cI$$ For $\cI$ as in proposition \ref{prop:B} this relation can be written down\footnote{To see that $k_1 (k_2)^{-1}\in \cI$ implies that $k_1$ and $k_2$ satisfy the conditions on the r.h.s. of \eqref{eq:simKbis}, notice the following: if $\bb$ is a bisection through $k_1 (k_2)^{-1}$ such that $\varphi(\bb)\subset 1_M$ and $\bb_2$ is any bisection through $k_2$, then $\bb\cdot \bb_2$ is a bisection through $k_1$ such that $\varphi(\bb\cdot \bb_2)=\varphi(\bb_2)$.
} 
explicitly:
\begin{align}\label{eq:simKbis}
k_1\sim_{\cK}k_2 \Leftrightarrow &\;\text{
$\varphi(k_1)=\varphi(k_2)$} \text{ and} \\
\nonumber &\;\text{$\exists$ local bisections $\bb_1$ through $k_1$,   $\bb_2$ through $k_2$, s.t. $\varphi(\bb_1)=\varphi(\bb_2)$.
}
\end{align}
Analogously to Rem. \ref{rem:HGBbisec}, the equivalence \eqref{eq:simKbis} can be rephrased as follows:
 $$k_1\sim_{\cK}k_2 
\Leftrightarrow \text{$\exists$ neighborhood $\cK'$ of $k_1$  and  $f \colon \cK'\to \cK$ satisfying $\varphi\circ f=\varphi$ and $f(k_1)=k_2$}.
$$
\end{remark}
 
\begin{proof}[Proof of proposition \ref{prop:B}] {The proof relies in Appendix \ref{sec:atlases}.} i) We first show the following claim:

\underline{Claim:} \emph{$H^{\cG}(\cB)=\cK/\sim_{\cK}$ as topological spaces.}

The set $\cU:=\{(\cK,\varphi,\cG)\}$ is an atlas for $\cB$ (see definition \ref{def:atlas}) consisting of just one bisubmersion. This holds by the following two arguments which are consequences of the fact that $\varphi$ is a groupoid morphism: 
\begin{itemize}
\item The inverse bisubmersion $(\cK,i_{\cG}\circ \varphi,\cG)$ is adapted to $\cU$ because the inversion map $i_{\cK}$ is a morphism of bisubmersions from $(\cK,i_{\cG}\circ \varphi,\cG)$ to $(\cK,\varphi,\cG)$. 
\item The composition of bisubmersions $\cK\circ \cK$ is also adapted to $\cU$. Indeed it agrees with 
the space of composable arrows of $K$ and the groupoid multiplication of $\cK$ is a morphism of bisubmersions, \ie this diagram commutes:
\begin{equation*}
\xymatrix{
\cK\circ \cK \ar[rd]_{\varphi\cdot\varphi  }  \ar[rr]^{\text{multipl. of $\cK$}} & & \cK \ar[ld]^{\varphi} \\
&\cG &
 }
\end{equation*} 
\end{itemize}
The atlas $\cU$ is equivalent to a path holonomy atlas (definition \ref{def:pathholatlas}). To prove this, by proposition \ref{prop:pathholadapted}, we just need to show that the bisubmersion 
$(\cK,\varphi,\cG)$ is adapted to a path holonomy atlas. Proposition \ref{prop:crucial} implies that for any $x\in M$, there exists a morphism of bisubmersions from an open neighborhood of $1_x$ in $\cK$ to some {path holonomy} bisubmersion. Then simply use that the Lie groupoid $\cK$ is generated by such neighborhoods.

As $\cU=\{(\cK,\varphi,\cG)\}$ is equivalent to a path holonomy atlas, by proposition \ref{prop:rem33}  we have $H(\cB)^{\cU}=H^{\cG}(\cB)$.  The former, as a topological space, is defined as the quotient of $\cK$ by the equivalence relation 
$\sim_{\cK}$ (see eq. \eqref{eq:HUcb}). This proves the claim. \hfill$\bigtriangleup$

Now we can show that $H^{\cG}(\cB)=\cK/\sim_{\cK}$ as  groupoids. Let $k,k'\in \cK$ such that $\bs_{\cK}(k)=\bt_{\cK}(k')$. First, the product of $[k]$ and $[k']$ in $H(\cB)^{\cU}=H^{\cG}(\cB)$ is the class of  $k\circ k'$ in the composition of bisubmersions $\cK\circ \cK$. Second, as seen above, the groupoid multiplication is  a morphism of bisubmersions from $\cK\circ \cK$ to $\cK$, hence $k\circ k'\sim_{\cK}kk'$. Combining the last two statements we obtain $[k]\cdot[k']=[kk']$.

ii) {The canonical map $H(\cB)^{\cU}\to \cG$ is the map $\cK/\cI\to \cG$ induced by $\varphi$. Now apply Prop. \ref{prop:rem33} ii).}
\end{proof}

We immediately obtain a construction for the holonomy groupoid of singular subalgebroids arising from Lie groupoid morphisms (covering the identity), see Def. \ref{def:arises}: 
\begin{prop}\label{cor:HBimage}   
{Assume that the singular subalgebroid $\cB$  arises from a Lie groupoid morphism $\Psi\colon \cK\to\cG$ covering the identity.}
Then $H^{\cG}(\cB)=\cK/\cI$, where $\cI$ is an in proposition \ref{prop:B}, {and 
the canonical map $\Phi\colon H^{\cG}(\cB)\to \cG$ is the map $\cK/\cI\to \cG$ induced by $\Psi$.}
\end{prop}
\begin{proof}
By proposition \ref{prop:imagerelbi}, $\Psi\colon \cK \to \cG$  is a bisubmersion for $\cB$. Hence we can apply proposition \ref{prop:B}.
\end{proof}

 \begin{ex}[Singular foliations arising from Lie algebroids]\label{ex:explicitgr}
\begin{itemize}
\item [a)] {The following example is a rephrasing of \cite[Ex. 3.4(4)]{AndrSk}.}
 Let $A$ be a Lie algebroid over $M$, and let $\cF:=\rho(\Gamma_c(A))$ be the singular foliation on $M$ 
associated to $A$ (here $\rho$ is  the anchor map). Let $\cK\rightrightarrows M$ be a Lie groupoid integrating $A$. The Lie groupoid morphism given by the target-source map $$\Psi:= (\bt,\bs)\colon \cK \to M\times M$$  
integrates the anchor map, so it gives rise to $\cF$, i.e.
$\cF=\{\Psi_*(\Gamma_c(A))\}$. Therefore Prop.  \ref{cor:HBimage}  implies that the holonomy groupoid $H(\cF)$ of the foliation is $$H(\cF)=\cK/\cI,$$ where $\cI$ consists of the elements $k\in \cK$ through which passes a local bisection inducing the identity (local) diffeomorphism on $M$.

\item [b)] {We now spell out a special case of the above (for linear actions compare  with \cite[Ex. 3.7]{AndrSk})}. Consider an action of a Lie group $G$ on $M$. It gives rise to a singular foliation $\cF$ on $M$, which is generated by the image of the associated infinitesimal action $\psi\colon \g\to \vX(M), v\to v_M$. Let $A:=\g\times M$ be   the transformation algebroid of the infinitesimal action. Its anchor map $\rho\colon \g\times M\to TM, (v,p)\mapsto v_M(p)$ satisfies $\cF=\rho(\Gamma_c(A))$. By a), the holonomy groupoid of $\cF$ is obtained from the transformation groupoid $G\times M\rightrightarrows M$ as $$H(\cF)=(G\times M)/\cI,$$
where $\cI$ is very explicit: it consists of $(g,x)\in G\times M$ (necessarily with $g\cdot x=x$) for which there is a neighborhood $U$ of $x$ in $M$ and a smooth map $\tilde{g}\colon U\to G$ such that  $\tilde{g}(x)=g$ and $\tilde{g}(y)\cdot y=y$ for all $y\in U$. In other words, it consists of elements of $G\times M$ through which there is a local section of the second projection $G\times M\to M$ which lies in the isotropy groups of the action of $G$ on $M$.
\end{itemize}
\end{ex}

{For singular subalgebroids arising from Lie groupoid morphisms, the holonomy groupoid satisfies a minimality property. In particular, it is a quotient of any Lie groupoid giving rise to the given singular subalgebroid.}
\begin{prop}\label{prop:lift}
Any  Lie groupoid morphism {covering the identity} $\Psi\colon \cK\to\cG$ giving rise to a singular subalgebroid $\cB$
factors as
  \begin{equation*} 
 \xymatrix{
 &H^{\cG}(\cB)\ar_{\Phi}[d]\\
 {\cK}\ar@{-->}[ru]^\tau \ar[r]^{\Psi}&\cG   }
\end{equation*}
where $\tau\colon \cG\to H^{\cG}(\cB)$ is  a surjective morphism of topological groupoids.
\end{prop}
\begin{proof}
{Thanks to Prop. \ref{cor:HBimage}, we can take $\tau$ to be the quotient map $K\to \cK/\cI=H^{\cG}(\cB)$.}
\end{proof}

\subsubsection{For wide Lie subalgebroids}\label{subsubsec:wide}
 
{The proof of the following statement is similar to the one of \cite[Prop. 1.9]{AZ5} but  
has the advantage of relying only on elementary facts.}
{
\begin{prop}\label{prop:smoothgroid}
Let $B$ be a wide Lie subalgebroid of $A$, let $G$ be a Lie groupoid integrating $A$, and denote $\cB:=\Gamma_c(B)$. Then
\begin{itemize}
\item [i)]  $H^{\cG}(\cB)$ is a Lie groupoid integrating $B$
\item [ii)]  the canonical Lie groupoid morphism $\Phi\colon H^{\cG}(\cB)\to \cG$ integrates the inclusion $\iota\colon B\hookrightarrow A$.
\end{itemize}
\end{prop}
}
 
\begin{proof}
  The Lie algebroid $B$ is integrable since $A$ is \cite{MMRC}.
Let $K$ be the source simply connected Lie groupoid integrating $B$. Let
 $\Psi \colon \cK \to \cG$  the morphism of Lie groupoids which integrates the inclusion $\iota \colon B\hookrightarrow A$.
 
On one hand, clearly the Lie groupoid morphism $\Psi$ gives rise to $\cB$. Therefore we can apply Prop. \ref{cor:HBimage}, which states that  $H^{\cG}(\cB)=\cK/\cI$ and the canonical map $\Phi\colon H^{\cG}(\cB)\to \cG$ is the map $\cK/\cI\to \cG$ induced by $\Psi$.  
 
 On the other hand we have the following
 
\underline{Claim:} \emph{The subgroupoid $\cI$ of $\cK$, defined as in Prop. \ref{prop:B}, satisfies:
\begin{enumerate}
\item Set-theoretically, $\cI$ is a normal subgroupoid of $\cK$ lying in the union of the isotropy groups of $\cK$.
\item Topologically, $\cI$ is an embedded Lie subgroupoid of $\cK$
and it is $\bs$-discrete (\ie the intersection of $\cI$ with any $\bs$-fiber is discrete).
\end{enumerate}
}

The claim implies (see for example  \cite[Thm. 1.20]{GuLi}) that $\cK/\cI$ is also a Lie groupoid integrating $B$. Clearly the Lie groupoid morphism $\cK/\cI\to \cG$ induced by $\Psi$ integrates the inclusion $\iota \colon B\hookrightarrow A$. This concludes the proof {of the proposition, modulo the   claim which we prove right now}.
 
That a) in the claim holds was explained in Rem. \ref{rem:cI}. We argue that b) holds.
There is a neighborhood $V\subset K$ of the set of identities $1_M$ on which the morphism $\Psi$ is injective (this follows from $\Psi$ being a Lie groupoid morphism covering the identity and whose Lie algebroid map is injective). Since $\cI\subset \ker(\Psi)$, we have
  $\cI\cap V=1_M.$
Let $k\in \cI$, and take a bisection $\bb$ of $K$ through $k$ as in the definition of $\cI$. Denote $U:=\bs(\bb)$, an open subset of $M$.
Denote $r_{\bb}\colon \bs^{-1}(U) \to \bs^{-1}(U)$ the diffeomorphism given by right-multiplication by the bisection $\bb$. It maps $1_{\bs(g)}$ to $k$ and it preserves $\cI$, since
$\bb$ lies in the subgroupoid $\cI$. Applying $r_{\bb}$ to
$$\cI\cap (V\cap \bs^{-1}(U))=1_U$$
we obtain that the intersection of $\cI$ with $r_{\bb}(V)\cap \bs^{-1}(U)$ (an    open neighbourhood of $k$) is exactly $\bb$. This shows both that $\cI$ is an embedded submanifold (hence, a \emph{Lie subgroupoid}) of $K$ and that $\cI$ is \emph{$\bs$-discrete}. 
\end{proof}

\begin{remark}
 {Recall from \S \ref{subsec:morph} that $\cB$ is a projective singular subalgebroid  if $\cB\cong \Gamma_c(E)$ for some vector bundle $E\to M$, which automatically comes with a Lie algebroid structure and a morphism $\tau \colon E\to A$.  In \cite{AZ4}, {generalizing Prop. \ref{prop:smoothgroid} and its proof}, we show that $H^{\cG}(\cB)$ is a Lie groupoid if{f} $\cB$ is projective, that
the Lie groupoid $H^{\cG}(\cB)$ integrates $E$, and that the canonical morphism $\Phi \colon H^{\cG}(\cB)\to G$ integrates $\tau$.} 
\end{remark}
 
 {We now refine Prop. \ref{prop:smoothgroid}, showing that in the case of wide Lie subalgebroids $H^{\cG}(\cB)$ is exactly the the  minimal integral of $B$ over $\cG$   
defined in the work of Moerdjik and Mr{\v{c}}un {\cite[Thm. 2.3]{MMRC} recalled in the Introduction.}}

\begin{prop}\label{prop:HGBisHmin}
{Let $B$ be a wide Lie subalgebroid of an integrable Lie algebroid $A$, and fix a Lie groupoid $\cG$ integrating $A$. Let  $\cB=\Gamma_c(B)$. 
 Then 
 \begin{itemize}
\item [i)] $H^{\cG}(\cB)$ agrees with $H_{min}$, the  minimal integral of $B$ over $\cG$ {recalled in the introduction,}  
\item [ii)] the canonical map $\Phi\colon H^{\cG}(\cB)\to G$ agrees with the map 
 {recalled in the introduction}. 
\end{itemize}
}
\end{prop}
\begin{proof}
{
By Prop. \ref{prop:smoothgroid} the canonical map $\Phi\colon H^{\cG}(\cB)\to \cG$ 
satisfies properties 1) and 2) from  \cite[Thm. 2.3]{MMRC} recalled in the Introduction. Any Lie groupoid morphism $\tilde{H}\to G$ integrating the   inclusion $\iota \colon B\hookrightarrow A$ is a Lie groupoid morphism  giving rise to $\cB$. 
Hence, by  Prop. \ref{prop:lift}, property 3) from  that theorem holds too. The uniqueness statement in that theorem finishes the proof.}
\end{proof}

The following two examples generalize the elementary Example \ref{ex:Liegr}.

 \begin{ex}[Lie algebroids]\label{ex:GammaA}
For any integrable Lie algebroid $A$ take $\cB=\Gamma_c(A)$, and let $\cG$ be a  Lie groupoid  integrating 
$A$. Then $H^{\cG}(\cB)=\cG$ and $\Phi=Id_{\cG}$. {This follows taking $\Psi=Id_{\cG}$ in Cor. \ref{cor:HBimage}. {(It also follows directly  from Prop. \ref{prop:HGBisHmin})}.}
\end{ex}

\begin{ex}[Lie subalgebras]
\label{ex:Liegrrevisited} Let $\g$ a Lie algebra, $\mathfrak{k} $ a Lie subalgebra, and
fix a connnected Lie group $G$ integrating $\g$.
Let $\Psi \colon
K\to G$ be any morphism of Lie groups integrating the inclusion $\iota \colon \mathfrak{k} \hookrightarrow \g$, where $K$ is assumed to be connected. (For instance, take $K$ to be the simply connected integration of $\mathfrak{k}$.) Then $$H^G(\mathfrak{k})=K/ker(\Psi).$$ {This follows from Cor. \ref{cor:HBimage}, noticing} that  since the space of objects of $\cK$ is just a point, we have $\cI=\ker(\Psi)$.
Using Prop. \ref{prop:smoothgroid} we see that  $H^G(\mathfrak{k})$ is a Lie group integrating $\mathfrak{k}$, and the map $\Phi\colon H^G(\mathfrak{k})\to G$ induced by $\varphi$ is an injective immersion and group homomorphism. In other words,
 $(H^G(\mathfrak{k}),\Phi)$ is the Lie subgroup of $G$ integrating $\mathfrak{k}$.
   \end{ex}

\subsection{Dependence of \texorpdfstring{$H^{\cG}(\cB)$}{} on \texorpdfstring{$\cG$}{}} \label{section:vary}

Let $A\to M$ be a Lie algebroid and $\cB$ a singular subalgebroid. Fix a Lie groupoid $\cG$ integrating $A$. In \S \ref{sub:con} we constructed the holonomy  groupoid $H(\cB):=H^{\cG}(\cB)$ over $M$, as the quotient of a path-holonomy atlas of $ {\cG}$-bisubmersions\footnote{In this section we use the terminology ``$\cG$-bisubmersion''   instead ``bisubmersion'', to emphatize the dependence on the Lie groupoid $G$.}
 by the equivalence relation given by morphisms of $ {\cG}$-bisubmersions. Clearly the construction depends on $\cG$.  

Now take another Lie groupoid $\widetilde{\cG}$ integrating $A$, and so that $\cG$ is a quotient of it, \ie there is a surjective  Lie groupoid {morphism} $$\pi \colon \widetilde{\cG}\to\cG$$ with discrete fibers. Denote by $\widetilde{H}(\cB):=H^{\widetilde{\cG}}(\cB)$ the holonomy groupoid constructed using $\widetilde{\cG}$. In this subsection we describe $\widetilde{H}(\cB)$ in terms of $H(\cB)$.

\subsubsection{{A theorem describing $\widetilde{H}(\cB)$ in terms of $H(\cB)$}}

Consider the fiber product of the canonical groupoid morphism  $\Phi\colon H(\cB)\to \cG$ and $\pi \colon \widetilde{\cG}\to \cG$, \ie  $$H(\cB)\times_{\Phi,\pi}\widetilde{\cG}\rightrightarrows M$$ with the component-wise groupoid structure. Upon the identification between $M$ and the diagonal $\Delta M\subset M\times M$, it is the subgroupoid of the product groupoid $H(\cB)\times \widetilde{\cG}\rightrightarrows M\times M$ given by the preimage of $\Delta \cG$ under the groupoid morphism $(\Phi, \pi)\colon H(\cB)\times \widetilde{\cG}\to \cG\times \cG$. Notice that the latter morphism does not have connected fibers in general, so that $H(\cB)\times_{\Phi,\pi}\widetilde{\cG}$ will not be source-connected in general. Hence we consider the source-connected component of the identities:
\begin{align}\label{eq:pathschar}
(H(\cB)\times_{\Phi,\pi}\widetilde{\cG})_0=\{(h,\widetilde{g})\in
H(\cB)\times \widetilde{\cG}:& \text{ $\exists$ a continuous path $(h(t),\widetilde{g}(t))\subset (\bs_H,\widetilde{\bs})^{-1}(x,x)$}\\ 
&\text{from $(1_{x}^{H(\cB)},1_x^{\widetilde{\cG}})$ to $(h,\widetilde{g})$ with $\Phi(h(t))=\pi(\widetilde{g}(t))$}\},\nonumber
\end{align}
where $x:=\bs_H(h)=\widetilde{\bs}(\widetilde{g})\in M$.

\begin{ex}
Take the simple case $\cB=\{0\}$. {For any choice of $G$, we have that $H(\cB)$
is the trivial groupoid $M\rightrightarrows M$.}
We obtain $$H(\cB)\times_{\Phi,\pi}\widetilde{\cG}=M\times_{(
\iota_M,\pi)}\widetilde{\cG}\cong \pi^{-1}(1_M^{\cG})=\ker(\pi)$$
{where $\iota_M$ is the inclusion of the identity elements $1_M^{\cG}$ into $G$.}
Hence $(H(\cB)\times_{\Phi,\pi}\widetilde{\cG})_0$ consists of the identity elements of $\widetilde{\cG}\rightrightarrows M$, {and therefore is isomorphic to the trivial groupoid $M\rightrightarrows M$}.
\end{ex}

\begin{remark}\label{rem:unique}
As $\pi \colon \widetilde{\cG}\to \cG$ has discrete fibers, for any path $\gamma$ in the $\bs$-fiber of $\cG$ starting at $1_x^{\cG}$ there exists a unique lift starting at $1_x^{\widetilde{\cG}}$, \ie a unique path $\tau$ in $\widetilde{\cG}$ with $\gamma=\pi\circ \tau$ and $\tau(0)=1_x^{\widetilde{\cG}}$. 

This shows that, in the   characterization \eqref{eq:pathschar} of $(H(\cB)\times_{\Phi,\pi}\widetilde{\cG})_0$,  the path $\widetilde{g}(t)$ (hence in particular $\widetilde{g}$) is determined by the path $h(t)$: indeed, $\widetilde{g}(t)$ is the unique $\pi$-lift of $\Phi(h(t))$ starting at $1_x^{\widetilde{\cG}}$.
\end{remark}

Notice that, if we fix a local generating set $\balpha_1,\dots,\balpha_n$ for $\cB$, it gives rise to two path-holonomy bisubmersions: a $\cG$-bisubmersion and a $\widetilde{\cG}$-bisubmersion. The domains of both bisubmersions coincide (they are an open subset in $\RR^n\times M$). The domains of  the  compositions of  path-holonomy bisubmersions also coincide, hence if $\cU$ is a $\cG$-path-holonomy atlas, then $\coprod_{U \in \mathcal{U}} U$ will be the domain of a $\widetilde{\cG}$-path-holonomy atlas too. Given corresponding bisubmersions
$(U,\varphi,\cG)$ and $(U,\widetilde{\varphi},\widetilde{\cG})$, {from Def. \ref{dfn:pathhol} we see that} $\varphi=\pi \circ \widetilde{\varphi}\colon U\to \cG$: 
\begin{equation*} 
 \xymatrix{  U  \ar[dd]^{\varphi
 }\ar[dr]^{\widetilde{\varphi}}&\\
  &\widetilde{\cG}\ar[dl]^{\pi}\\
 \cG&}
\end{equation*}

We denote the quotient maps to the holonomy groupoids by $\natural|_U\colon U\to H(\cB)$ and 
$\widetilde{\natural}|_U\colon U\to \widetilde{H}(\cB)$ respectively,
as in \S\ref{sub:con}. 

{The central result of this subsection is the following.}

\begin{thm}\label{thm:tilde}
There is a canonical isomorphism over $Id_M$ of topological groupoids $$T\colon \widetilde{H}(\cB)\to (H(\cB)\times_{\Phi,\pi}\widetilde{\cG})_0 \quad \widetilde{h}\mapsto (\natural u, \widetilde{\Phi}(\widetilde{h}))$$ where $u$ is any point in a path-holonomy atlas with $\widetilde{\natural} u=\widetilde{h}$.
\end{thm}
{Clearly, under the above isomorphism, the canonical map $\widetilde{\Phi}\colon \widetilde{H}(\cB)\to \widetilde{\cG}$ corresponds to the  
{second projection $(h,\widetilde{g})\mapsto  \widetilde{g}$.}

\begin{remark}
The relevant diagram is
\begin{equation} \label{eq:bigdia}
 \xymatrix{
 &\coprod_{U \in \mathcal{U}} U \ar@{-->}[d]_{\widetilde{\natural}}\ar[dl]_{{\natural}}\ar[dr]^{\widetilde{\varphi}}&\\
  H(\cB)\ar[dr]_{\Phi}  & \widetilde{H}(\cB) \ar@{-->}[r]_{\widetilde{\Phi}} \ar@{-->}[l] &\widetilde{\cG}\ar[dl]^{\pi}\\
&\cG&}
\end{equation}
\end{remark}

\begin{proof}[Proof of theorem \ref{thm:tilde}]

\underline{Claim:} \emph{$T$ is well-defined.}
 
 We show that the map $$\widetilde{H}(\cB)\to H(\cB),\; \widetilde{\natural} u\mapsto {\natural} u$$ is well-defined.
 Let $U,V$ be $\widetilde{\cG}$-bisubmersions, and $u,v$ points with $\widetilde{\natural} u=\widetilde{\natural} v$. This means that there is a morphism of $\widetilde{\cG}$-bisubmersions $f\colon U\to V$ with
 $f(u)=v$ (possibly shrinking $U$). Composing $\widetilde{\varphi}_V\circ f=\widetilde{\varphi}_U\colon U\to \widetilde{\cG}$   with $\pi \colon \widetilde{\cG}\to \cG$ we find ${\varphi_V}\circ f={\varphi_U}\colon U\to {\cG}$ (using $\pi \circ \widetilde{\varphi}_U ={\varphi_U}$). This shows that $f$ is also a  morphism of $G$-bisubmersions, therefore ${\natural} u={\natural} v$.

Further, the image of $T$ is really contained in $(H(\cB)\times_{\Phi,\pi}\widetilde{\cG})_0$. {Indeed, it is contained in the fiber product $H(\cB)\times_{\Phi,\pi}\widetilde{\cG}$ because we have $\Phi(\natural u))=\varphi(u)=\pi(\widetilde{\varphi}(u))=\pi(\widetilde{\Phi}(\widetilde{h}))$ for all $u$. It is contained in the source-connected component of the identities because  $\widetilde{H}(\cB)$ is source-connected and, as we shall see immediately, $T$ is a continuous groupoid morphism.} \hfill$\bigtriangleup$

\underline{Claim:} \emph{$T$ is a groupoid morphism.}

One checks directly that $\widetilde{H}(\cB)\to H(\cB), \widetilde{\natural} u \mapsto \natural u$ is morphism of groupoids. Also, $\widetilde{\Phi}\colon \widetilde{H}(\cB)\to  \widetilde{\cG}$ is a morphism of groupoids by Thm. \ref{thm:holgroidconstr}.\hfill$\bigtriangleup$

\underline{Claim:} \emph{$T$  is continuous.}

{This holds since 
  $\widetilde{H}(\cB)\to H(\cB)$
is continuous (being the map induced by $\natural \colon \coprod_{U \in \mathcal{U}} U \to H(\cB)$ on the quotient $\widetilde{H}(\cB)$) and since 
 $\widetilde{\Phi}$ is continuous. 
}

\underline{Claim:} \emph{$T$ is injective.}

Since $T$ is a morphism of groupoids, it suffices to check that
if $\natural v=1_x^{H(\cB)}$ and  $\widetilde{\varphi}_V(v)=1_x^{\widetilde{\cG}}$ then 
$\widetilde{\natural} v =1_x^{\widetilde{H}(\cB)}$, for any element $v$ in a bisubmersion $V$. Let $(U,\widetilde{\varphi}_U,
\widetilde{\cG})$ be a path-holonomy $\widetilde{\cG}$-bisubmersion containing $(0,x)$, so $U \subset \RR^n\times M$. There
exists a morphism of $\cG$-bisubmersions $f\colon U\to V$ with $f(0,x)= v$, by the first assumption above. Notice that the diagram
\begin{equation*} 
 \xymatrix{
 U\ar[dd]^f\ar[r]^{\widetilde{\varphi}_U}&\widetilde{\cG}\ar[dr]^{\pi}&\\
 &&\cG\\
V\ar[r]^{\widetilde{\varphi}_V}&\widetilde{\cG}\ar[ur]_{\pi}&}
\end{equation*}
commutes, since $\varphi_U=\pi\circ\widetilde{\varphi}_U$.

Now consider  $S:=(\{0\}\times M)\cap U$. We have $\widetilde{\varphi}_U(S)=1_S^{\widetilde{\cG}}$, an open subset of the identity bisection of $\widetilde{\cG}$, and $(\widetilde{\varphi}_V\circ f)(S)$ is a bisection of $\widetilde{\cG}$ which, by the second hypothesis above, contains $1_x^{\widetilde{\cG}}$. Both bisections map under $\pi$ to $1_S^{\cG}$, an open subset of the identity bisection of $\cG$, by the commutativity of the above diagram. Since $\pi$ has \emph{discrete} fibers, it follows that the two bisections of $\widetilde{\cG}$ agree, \ie $\widetilde{\varphi}_U(S)=\widetilde{\varphi}_V(f(S))$. Hence corollary \ref{cor:crucial} b) implies that there is morphism of $\widetilde{\cG}$-bisubmersions $U\to V$ with $(0,x)\mapsto v$, hence $\widetilde{\natural}v=\widetilde{\natural}(0,x)=1_x^{\widetilde{H}(\cB)}$.\hfill$\bigtriangleup$

\underline{Claim:} \emph{$T$ is surjective.} 

Let $(h,\widetilde{g})\in
H(\cB)\times \widetilde{\cG}$ so that  there is a continuous path $(h(t),\widetilde{g}(t))$ from $(1_{x}^{H(\cB)},1_x^{\widetilde{\cG}})$ to $(h,\widetilde{g})$ with
$\bs_H(h(t))=\widetilde{\bs}(\widetilde{g}(t))=x$ and $\Phi(h(t))=\pi(\widetilde{g}(t))$.
We have to show that there is a bisubmersion $U$ in the path-holonomy atlas and $u\in U$ with $\natural u=h,\widetilde{\varphi}(u)=\widetilde{g}$. {Since $T$ is a groupoid morphism and every source-connected topological groupoid is generated by any symmetric neighbourhood of the identities, we can assume that $(h,\widetilde{g})$ is arbitrarily close to the set of identities. Hence, by Remark \ref{rem:open}, we can assume that there is a path holonomy bisubmersion $U_0$ such that $h\in \natural U_0$.}

Denote by $L\subset M$ the leaf of the foliation $\rho(\cB)$ through $x$. As we 
show in   \cite{AZ4},
$H(\cB)|_L:=\bs_H^{-1}(L)$ has a smooth structure such that 
for any $\cG$-bisubmersion $U$ in the path-holonomy atlas of $\cB$, the quotient map $\natural\colon U|_L:=\bs_U^{-1}(L)\to H(\cB)|_L$ is a submersion. Hence  there exists a continuous curve $u(t)$ {in $U_0$} with $\natural u(t)=h(t)$ and $u(0)=(0,x)$, where $(0,x)$ lies in a minimal path-holonomy $\cG$-bisubmersion. We claim that $u:=u(1)$ satisfies the above properties. 

By definition we have $\natural u(1)=h(1)=h$. Now consider the   part of   diagram \eqref{eq:bigdia} with solid arrows and the paths:
\begin{equation*} 
 \xymatrix{&u(t)  \ar[dl]_{{\natural}}\ar[dr]^{\widetilde{\varphi}}&\\
h(t) \ar[dr]_{\Phi}  &
 & \widetilde{\varphi}(u(t))\ar[dl]^{\pi}\\
 &{\Phi}(h(t))&}
\end{equation*}
The path $\widetilde{\varphi}(u(t))$ is a $\pi$-lift of $\Phi(h(t))$, and its starting point is 
$\widetilde{\varphi}((0,x))=1_x^{\widetilde{\cG}}$. The same holds for $\widetilde{g}(t)$. Hence by the uniqueness of the $\pi$-lift starting at $1_x^{\widetilde{\cG}}$ (see Remark \ref{rem:unique}) we obtain 
$\widetilde{\varphi}(u(t))=\widetilde{g}(t)$, and evaluating at $t=1$ we get $\widetilde{\varphi}(u)=\widetilde{g}$.\hfill$\bigtriangleup$

\underline{Claim:} \emph{$T$  is A homeomorphism.}
 It suffices to show that $T$ is an open map. Let $(U,\varphi,\cG)$ be a $G$-bisubmersion in the path-holonomy atlas. Then 
$\widetilde{\natural} U$ is open in $\widetilde{H}(\cB)$,
{by Remark \ref{rem:open}. We will show that its image $T(\widetilde{\natural} U)$ is open.
 The {Lie groupoid morphism $\pi \colon \widetilde{G}\to G$} is a local homeomorphism, hence the first projection $pr_1\colon H(\cB)\times_{\Phi,\pi}\widetilde{\cG}\to H(\cB)$ is also a local homeomorphism. 
Shrinking $U$ if necessary, we can assume that $T(\widetilde{\natural} U)=\{(\natural u, \widetilde{\varphi} u):u\in U\}$ is contained in an open subset $N$ of $(H(\cB)\times_{\Phi,\pi}\widetilde{\cG})_0$ such that $pr_1|_N\colon N\to pr_1(N)$ is a homeomorphism onto an open subset  of ${H}(\cB)$. The subset $pr_1(T(\widetilde{\natural} U))$ equals $\natural U$, which is open in $H(\cB)$  
by Remark \ref{rem:open}. Hence $T(\widetilde{\natural} U)$ is open in $N$ and therefore in $(H(\cB)\times_{\Phi,\pi}\widetilde{\cG})_0$.

Summarizing: for small enough bisubmersions $U$ in the path-holonomy atlas, $T$ maps the open subsets $\widetilde{\natural} U$ of $\widetilde{H}(\cB)$   to open subsets. Since any open subset of 
$\widetilde{H}(\cB)$ is a union of such $\widetilde{\natural} U$, we are done.} 
 \hfill$\bigtriangleup$
 
\end{proof}

\begin{remark}\label{rem:ker}  There is a groupoid morphism $$\widetilde{H}(\cB)\to H(\cB),\; \widetilde{\natural} u\mapsto {\natural} u$$(see the first two claims in the proof of theorem \ref{thm:tilde}), which is clearly surjective. Under the  
  canonical isomorphism $\widetilde{H}(\cB)\cong(H(\cB)\times_{\Phi,\pi}\widetilde{\cG})_0$ it is given by the projection onto the first component.  
{Its kernel is  
\begin{equation}\label{eq:kermor}
  (1_M^{H(\cB)}\times \widetilde{\cG})\cap (H(\cB)\times_{\Phi,\pi}\widetilde{\cG})_0,
\end{equation}
which is contained in
$\{(1_x^{H(\cB)},\widetilde{g}): x\in M, \pi(\widetilde{g})=  1_x^{\cG}\}\cong\ker \pi$. A more explicit description of the kernel can be obtained using equation \eqref{eq:pathschar}. (Here $\pi\colon\widetilde{\cG} \to \cG$ is the original covering map.)}
{An explicit example of the groupoid morphism $\widetilde{H}(\cB)\to H(\cB)$ is given in Example \ref{ex:notssc}.}

{Let $L$ be a leaf of $\cB$.
Then  $\widetilde{H}(\cB)|_L$ 
and $\widetilde{H}(\cB)|_L$ are
 transitive Lie groupoids over $L$  \cite[\S 2]{AZ4}. The kernel of the surjective morphism $\widetilde{H}(\cB)|_L\to \widetilde{H}(\cB)|_L$ is the restriction of   \eqref{eq:kermor} to $L$. This kernel is contained in the product of 
$1_L^{H(\cB)|_L}$ with $(\bs_{\widetilde{G}}^{-1}(L)\cap \bt_{\widetilde{G}}^{-1}(L))_{0}$, the Lie groupoid given by the source connected component of the restriction of $\widetilde{G}$ to $L$. This follows from the fact that $L$ is contained in a leaf of the Lie groupoid $\widetilde{G}$.
 }
\end{remark}

\subsubsection{Examples for {Theorem \ref{thm:tilde}.}}

{We present a few examples for Theorem \ref{thm:tilde}.   Ex. \ref{ex:notssc}   in particular shows that even when we use a source simply connected  Lie groupoid $G$   to construct the holonomy groupoid $H^G(\cB)$, the latter  might not be  source simply connected}.

\begin{ex}
When $\cB=\Gamma_c(A)$, Theorem \ref{thm:tilde} and Example \ref{ex:GammaA} recover the obvious isomorphism $\widetilde{\cG}\cong \cG\times_{Id,\pi}\widetilde{\cG}$.
\end{ex}

\begin{ex}\label{ex:regfol}
Consider the case of a singular foliation $\cF$.  First we integrate the Lie algebroid $TM$ to the pair groupoid $\cG:=M \times M$, giving rise to $H(\cF):=H^{\cG}(\cF)$, the holonomy groupoid of the singular foliation as in  \cite{AndrSk} (see example \ref{ex:singfo}). We have $\Phi=(\bt_H,\bs_H)\colon H(\cF)\to M\times M$.
We can also integrate $TM$  to the fundamental groupoid $\widetilde{\cG}:=\Pi(M)$.
 The construction of \S \ref{sub:con} gives rise to another topological groupoid  $ \widetilde{H}(\cF):=H^{\widetilde{\cG}}(\cF)$, which has a canonical groupoid morphism $\widetilde{\Phi}$ to $ \Pi(M)$.
By Theorem \ref{thm:tilde} we have $$\widetilde{H}(\cF)\cong (H(\cF)\times_{(\bt_H,\bs_H),\pi}\Pi(M))_0,$$
where $\pi \colon \Pi(M)\to M\times M$ is the target-source map of $\Pi(M)$ (sending the homotopy class of a path in $M$ to its endpoints).
\end{ex}
Example \ref{ex:regfol} can be made more explicit when $\cF$ is a regular foliation.

\begin{prop}\label{prop:regfol}
When $\cF$ is a regular foliation one obtains 
\begin{equation}\label{eq:paths}
(H(\cF)\times_{(\bt_H,\bs_H),\pi}\Pi(M))_0= \big\{([\gamma],\langle \gamma \rangle_M): \gamma \text{ is a path in a leaf of the foliation}\big\} 
\end{equation}
where $[\gamma]\in H(\cF)$ denotes the holonomy class of $\gamma$, and 
$\langle \gamma \rangle_M\in \Pi(M)$ its \emph{homotopy} class (fixing endpoints) as a path in $M$.
\end{prop}
\begin{proof}
Use that, by equation \eqref{eq:pathschar}, $(H(\cF)\times_{(\bt_H,\bs_H),\pi}\Pi(M))_0$ equals  
 \begin{align*}
\big\{([\delta],\langle \sigma \rangle_M)\in
H(\cF)\times \Pi(M):& \exists \text{  homotopy $\{\delta^t\}$ in $L$ with $\delta^0=$(loop   with trivial holonomy), $\delta^1=\delta$}\\ &\exists \text{ 
 homotopy $\{\sigma^t\}$ in $M$ with $\sigma^0=$(contractible loop), $\sigma^1=\sigma$}\\ 
 & \text{ such that $\delta^t(0)=\sigma^t(0)= x$ and $\delta^t(1)=\sigma^t(1) \text{ for all $t$}$}\big\},
\end{align*}
where 
$L$ denotes the leaf through ${x:=\delta(0)=\sigma(0)}$.
Fix  a path $\delta$ in $L$ and a path $\sigma$ in $M$ as above (in particular, both start at $x$ and have the same endpoint). We first focus on $\delta$.
Consider the map
 $$S\colon [0,1]\times [0,1]\to L, (s,t)\mapsto \delta^t(s).$$
 \begin{figure}[htp] \centering{
\includegraphics[scale=0.75]{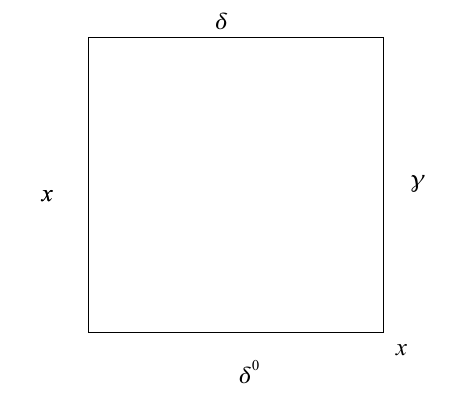}}
\end{figure} 

  As the square is contractible, the holonomy of the restriction of $S$ to the  boundary is zero. $S$ maps the left edge of the square to the constant path at $x$, and the lower edge to the loop  $\delta^0$ which has trivial holonomy. Considering the remaining two edges we conclude that
 $t \mapsto \gamma(t):=\delta^t(1)$ is a path in the leaf $L$ with the same  starting/ending point and same holonomy as $\delta$, \ie, $[\delta]=[\gamma]$.
 
Further $\bar{\gamma}\circ \sigma$, which is defined composing $\sigma$ with   $\widebar{\gamma}(t):=\gamma(1-t)$, is a contractible loop in $M$ based at $x$. Indeed, {recalling that $\gamma(t)=\sigma^t(1)$, we see that} the family of loops $\widebar{\gamma|_{[0,t]}}\circ \sigma^t$  parametrized by $t\in[0,1]$  provides a contraction, since at time $t=0$  it equals the contractible loop $\sigma^0$. In other words,  $\langle \sigma \rangle_M=\langle \gamma \rangle_M$. Altogether we get the inclusion ``$\subset$'' in equation \eqref{eq:paths}.
For the opposite inclusion, given a path $\gamma$ in a leaf, use the homotopy  $\{\gamma^t\}$ defined by $\gamma^t(s)=\gamma(st)$ to deform it to {the constant path at $\gamma(0)$}.
 \end{proof}

\begin{ex}\label{ex:notssc}
As above, let $\cF$ be a regular foliation.
{Denote by $D$ the  associated involutive distribution, which in particular is a Lie algebroid. We display three Lie groupoids integrating the Lie algebroid $D$. The first one is  $H(\cF)$, the   holonomy groupoid of $\cF$. The second is $Mon(\cF)$, the monodromy groupoid of $\cF$. It is a} source simply connected  Lie groupoid, consisting of all homotopy classes (fixing endpoints) $\langle \gamma \rangle_{leaf}$ of paths $\gamma$ in the leaves of the foliation. 
{The third one is  $\widetilde{H}(\cF)$ as in Ex. \ref{ex:regfol}. (It integrates $D$ since it is the minimal integral of $D$ over $\Pi(M)$, by Prop. \ref{prop:HGBisHmin}.)}

{As for any (source connected) Lie groupoid integrating $D$, the Lie groupoid $\widetilde{H}(\cF)$ is a quotient of $Mon(\cF)$ and maps surjectively onto $H(\cF)$.} 
 Due to Proposition \ref{prop:regfol} the quotient maps read
$$Mon(\cF)\to \widetilde{H}(\cF)\to H(\cF),\;\; \langle \gamma \rangle_{leaf} \mapsto
([\gamma], \langle \gamma \rangle_M)\mapsto [\gamma].$$
{In particular {we see that}, even though $\widetilde{H}(\cF)$ was constructed using a source simply connected Lie groupoid $\widetilde{\cG}$ in Ex. \ref{ex:regfol},   $\widetilde{H}(\cF)$ itself is   not source simply connected in general.}
\end{ex}

\section{Morphisms  of holonomy groupoids covering the identity}\label{section:morph}

In \S \ref{section:HB}, starting from a Lie groupoid $\cG$ and singular subalgebroid $\cB$ of the Lie algebroid $Lie(\cG)$, we constructed the holonomy  groupoid $H^{\cG}(\cB)$ endowed with a map of topological groupoids to $\cG$.
In this section we extend this construction to morphisms covering the identity: given a suitable ``morphism between singular subalgebroids'', we construct a morphism between the associated holonomy groupoids. 
More precisely, given a morphism of Lie groupoids $F\colon \cG_1\to \cG_2$ covering the identity on the base and singular subalgebroids  $\cB_i$ of $Lie(\cG_i)$ (${i=1,2}$)
with 
$F_*(\cB_1)\subset \cB_2$,   there is a canonical  morphism of topological groupoids 
$$H^{\cG_1}(\cB_1)\to H^{\cG_2}(\cB_2)$$ commuting with the canonical maps (see Theorem \ref{thm:morph}).

We consider the simple case of {``surjective morphisms between singular subalgebroids''} in \S\ref{subsec:fol}. 
The integration of arbitrary morphisms (covering the identity) then follows easily in \S\ref{sec:genmor}. All examples are collected in \S\ref{sec:morex}, where  
 we recover in a unified fashion several of the constructions we already gave.
 We describe the resulting functor in \S\ref{sec:functor}. 
 
{It is only for the sake of presentation that} in this section we restrict ourselves to morphisms covering the identity on the base manifold. Analogous results for morphisms covering surjective submersions hold, and  are collected in Appendix \ref{sec:appmorsub} (see Thm. \ref{thm:morphsub}).

\subsection{{Surjective morphisms covering the identity}}\label{subsec:fol}

\begin{prop}\label{prop:hbhfb}  
Let $F\colon \cG_1\to \cG_2$ be a morphism of Lie groupoids over $M$, covering $Id_M$.
Let $\cB_1$ be a singular subalgebroid of $A_1:=Lie(\cG_1)$, and  $$\cB_2:=F_*(\cB_1):=\{F_*(\balpha):\balpha\in \cB_1\},$$ which clearly is a singular subalgebroid  of $ Lie(\cG_2)$.

Then there is a canonical, surjective morphism of topological groupoids $$\Xi \colon H^{\cG_1}(\cB_1)\to H^{\cG_2}(\cB_2)$$ covering $Id_M$ and 
making the following diagram commute:
\begin{equation*} 
 \xymatrix{
H^{\cG_1}( {\cB}_1)  \ar[d]^{\Phi_1}   \ar@{-->}[r]^{\Xi} &H^{\cG_2}(\cB_2)  \ar[d]^{\Phi_2}     \\
\cG_1 \ar[r]^{F} &   \cG_2 }
\end{equation*}
\end{prop}
\begin{proof}
Let $x \in M$ and $\balpha_1,\dots,\balpha_n \in \cB_1$ such that $[\balpha_1],\dots,[\balpha_n]$  is a basis of $\cB_1/I_x\cB_1$. Denote by $(U,\varphi, \cG_1)$ the associated path holonomy bisubmersion, where $U\subset \RR^n\times M$.

{\underline{Claim} \emph{$(U,F\circ \varphi,\cG_2)$ is 
the path holonomy bisubmersion for $\cB_2$ 
 associated to the  (not necessarily minimal) set of local generators $\{F_*(\balpha_1),\dots,F_*(\balpha_n)\}$  of $\cB_2$.  }
}

\begin{equation*} 
 \xymatrix{
U\ar[d]^{\varphi} &      \\
\cG_1 \ar[r]^{F} &   \cG_2 }
\end{equation*}
By   definition \ref{dfn:pathhol} we have
$\varphi\colon U  \to \cG_1, (\lambda,x)\mapsto \exp_x \sum \lambda_i \overset{\rightarrow}{\balpha_i}$.
Composing with $F$ we obtain
\begin{equation*}
(F\circ \varphi)((\lambda,x))=F(\exp_x \sum \lambda_i \overset{\rightarrow}{\balpha_i})=\exp_{f(x)} \sum \lambda_i \overrightarrow{F_*(\balpha_i)}.
\end{equation*}
Here the second equality holds because  $F$ being a Lie groupoid morphism  implies that $F_*(\overset{\rightarrow}{\balpha_i})=\overrightarrow{F_*(\balpha_i)}$. 
     \hfill$\bigtriangleup$
 
 Consider a family $(U_i,\varphi_i,\cG_1)_{i \in I}$ of {path holonomy} bisubmersions for $\cB_1$  such that $M = \cup_{i \in I}\bs_{i}(U_i)$. Let $\mathcal{U}$ be   the  path holonomy atlas it generates (definition \ref{def:pathholatlas}), \ie the collection of the $U_i$'s together with their inverses and finite compositions. Since $F$ is a Lie groupoid morphism over $Id_M$,    the family  $\{(U,F\circ \varphi, \cG_2):U
\in \cU\}$
 defines an atlas of bisubmersions for $\cB_2$. 
 {Denote by $\sim_{i}$   the equivalence relation defined on $\coprod_{U \in \mathcal{U}}U$ viewing the $U$'s as $G_i$-bisubmersions  for ${\cB_i}$, for $i=1,2$.
 The equivalence classes of  $\sim_{1}$ are contained in those of $\sim_{2}$,} inducing a surjective morphism of topological groupoids
$$ H^{\cG_1}(\cB_1)=\coprod_{U \in \mathcal{U}} U/\sim_{1}\to \coprod_{U \in \mathcal{U}} U/\sim_{2}.$$
The latter groupoid is the holonomy groupoid  $H^{\cG_2}(\cB_2)$,   by Proposition \ref{prop:equivpathhol} and Rem. \ref{rem:HHu}. {The morphism is independent of the chosen path-holonomy bisubmersions and hence canonical.}
\end{proof}

{We consider the special case in which the  Lie groupoid morphism $F\colon \cG_1\to \cG_2$ is injective (then $F_*\colon A_1\to A_2$ is injective, and therefore induces an isomorphism of $C^{\infty}(M)$-modules $\cB_1\cong \cB_2$). We show that in that case it makes no difference whether we regard  $\cB_1$ a singular subalgebroid of $Lie(\cG_1)$ or as a singular subalgebroid of $Lie(\cG_2)$, for the 
 corresponding holonomy groupoids are isomorphic.}

\begin{cor}\label{prop:morphG}
If the Lie groupoid morphism $F\colon \cG_1\to \cG_2$ is injective, then the canonical morphism of topological groupoids $\Xi \colon H^{\cG_1}(\cB_1)\to H^{\cG_2}(\cB_2)$ of Prop. \ref{prop:hbhfb} is an isomorphism.
\end{cor}
\begin{proof}
{We just have to show that $\Xi$ is injective.
To this aim, we show that the equivalence classes of  the equivalence relation $\sim_2$ appearing in the proof of Prop. \ref{prop:hbhfb} are contained in those of $\sim_1$. Let $u\in U$ and $v\in V$ be points of bisubmersions in $\mathcal{U}$. Assume that $u\sim_2 v$.
 By definition, this means that there is a locally defined map $\tau\colon U\to V$  taking $u$ to $v$ and which is a morphism of $G_2$-bisubmersions, that is, $$(F \circ \varphi_U)=(F \circ\varphi_V)\circ \tau.$$
Since $F$ is injective, it follows that  $\varphi_U= \varphi_V\circ \tau$, that is, $\tau$ is a morphism of $G_1$-bisubmersions. In particular, $u\sim_1 v$.}
\end{proof}

\begin{remark}
If we only assume that  Lie algebroid map $F_*$ is injective, but   $F\colon \cG_1\to \cG_2$   itself not, then $\Xi$ is not injective in general.
This can be seen already in the case that $\cG_1$ and $\cG_2$ are non-isomorphic Lie groupoids integrating the same Lie algebroid $A$, and $F\colon \cG_1\to \cG_2$ is a  surjective but not injective   map
differentiating the identity at the level of Lie algebroids {(use Ex. \ref{ex:GammaA})}. 
 \end{remark}

\subsection{{Arbitrary morphisms covering the identity}}
\label{sec:genmor} 
 
We now extend  Prop. \ref{prop:hbhfb}  by 
removing the assumption that the map $F_*|_{\cB_1} \colon \cB_1\to \cB_2$ be surjective.

We first look at two singular subalgebroids of the same Lie algebroid, one containing the other.

\begin{lemma}\label{prop:morph0} 
Let $\cG$ be a   Lie groupoid over $M$, and denote its Lie algebroid by $A$. Let $\cB,\widetilde{\cB}$ be a singular subalgebroids of $A$, with $\cB \subset \widetilde{\cB}$.

Then there is a canonical  morphism of topological groupoids $H^G(\cB)\to H^G(\tilde{\cB})$ making the following diagram commute:
\begin{equation*}
\xymatrix{
H^G( {\cB})  \ar[rd]_{\Phi}    \ar@{-->}[rr] & & H^G(\widetilde{\cB})  \ar[ld]^{\widetilde{\Phi}}  \\
&\cG & }
\end{equation*}
\end{lemma}

\begin{proof} 
Let $x\in M$. Let  $\{\balpha_1,\dots,\balpha_n\}$ be a minimal set of local generators of $\cB$ near $x$, that is, 
$[\balpha_1],\dots,[\balpha_n]$ form a basis of the vector space
$\cB/I_x\cB$.  Let  
$(U_0,\varphi,\cG)$ be the corresponding (minimal)   path holonomy bisubmersion for $\cB$, hence $U_0\subset \RR^n\times M$. 

The inclusion $\cB\subset \widetilde{\cB}$ induces a linear map $J\colon \cB/I_x\cB \to \widetilde{\cB}/I_x\widetilde{\cB}$, which is generally not injective. Completing  the image of the above basis to a spanning set of 
$\widetilde{\cB}/I_x\widetilde{\cB}$ we obtain a generating set 
$\{\balpha_1,\dots,\balpha_n, \bgamma_1,\dots,\bgamma_k\}$ of 
$\widetilde{\cB}$ near $x$, {see Remark \ref{rem:basis}}.  Let $(\widetilde{U}_0,\widetilde{\varphi},\cG)$ 
 be the corresponding     path holonomy bisubmersion for $\widetilde{\cB}$, so $\widetilde{U}_0\subset \RR^n\times \RR^k \times M$. The inclusion 
$$\iota \colon U_0\to \widetilde{U}_0, (\lambda,y)\mapsto (\lambda,0,y)$$
commutes with the maps to $\varphi\colon U_0\to \cG$ and $\widetilde{\varphi} \colon \widetilde{U}_0\to \cG$ associated to the bisubmersions 
as in Definition \ref{dfn:pathhol}. 

Let $\cU$ denote the atlas for $\cB$ generated by bisubmersions $U_0$ as above, and $\widetilde{\cU}$   the atlas for $\widetilde{\cB}$ generated by the corresponding $\widetilde{U}_0$ as above. {The inclusion} 
 $\iota$ extends in a straightforward way to compositions of bisubmersions, and hence to a map
\begin{equation}\label{eq:utildeu}
\iota\colon \coprod_{U\in \cU}U   \to \coprod_{\widetilde{U}\in \widetilde{\cU}}\widetilde{U}
\end{equation}
 commuting with the canonical maps to $\cG$.

The latter property assures that 
if  {$u\in U$ for some element $U$ of the atlas $\cU$}, and $\bb$ is a bisection of $U$ through $u$, then its image $\iota(\bb)$ is a bisection of $\widetilde{U}$ through $\iota(u)$ {and both carry the same bisection of $G$}. {By Remark \ref{rem:HGBbisec}} this implies that the map \eqref{eq:utildeu} descends to a morphism of topological groupoids
\begin{equation}\label{eq:utildeugpd}
 H^G(\cB)=\coprod_{U\in \cU}U/\sim  \;\; \to\;\; \coprod_{\widetilde{U}\in \widetilde{\cU}}\widetilde{U}/\sim = H^G(\widetilde{\cB})
\end{equation}
which commutes the canonical maps to $\cG$, {and also that this morphism is 
independent of the chosen atlas $\cU$ and hence
canonical}.
{Notice that the topological groupoid on the r.h.s. is really $H^G(\widetilde{\cB})$,   by Proposition \ref{prop:equivpathhol} and Rem. \ref{rem:HHu}.}
\end{proof}

\begin{remark}
{The morphism $H^G(\cB)\to H^G(\tilde{\cB})$ in Lemma \ref{prop:morph0} is not injective in general. For instance, taking $G=M\times M$, $\cB$  a   foliation and $\widetilde{\cB}=\Gamma_c(TM)$, this map is the target-source map of the holonomy groupoid of the foliation.}
\end{remark}

The next theorem generalizes Prop. \ref{prop:hbhfb} and establishes the functoriality of the holonomy groupoid construction: 

\begin{thm}\label{thm:morph}
Let $F\colon \cG_1\to \cG_2$ be a morphism of Lie groupoids covering $Id_M$. Let $\cB_i$ be a singular subalgebroid of $Lie(\cG_i)$ for $i=1,2$,
 such that
$F_*(\cB_1)\subset \cB_2$.
Then there is a canonical morphism of topological groupoids $$\Xi \colon H^{\cG_1}(\cB_1)\to H^{\cG_2}(\cB_2)$$ covering $Id_M$ and 
making the following diagram commute:
\begin{equation*} 
 \xymatrix{
H^{\cG_1}( {\cB}_1)  \ar[d]^{\Phi_1}   \ar@{-->}[r]^{\Xi} &H^{\cG_2}(\cB_2)  \ar[d]^{\Phi_2}     \\
\cG_1 \ar[r]^{F} &   \cG_2 }
\end{equation*}.
\end{thm}
\begin{proof}
Compose the canonical morphism   $\Xi \colon H^{{G_1}}(\cB_1)\to H^{{G_2}}(F_*(\cB_1))$ given by Prop. \ref{prop:hbhfb} with the canonical morphism  $H^{{G_2}}(F_*(\cB_1))\to H^{{G_2}}(\cB_2)$   given by Lemma \ref{prop:morph0}.
\end{proof}

 \subsection{Examples}\label{sec:morex}

We display a few examples for
Prop. \ref{prop:hbhfb}, recovering in a unified manner  several results obtained so far.  {Then we display   examples} for Thm. \ref{thm:morph}.

\begin{ex}[Images of Lie algebroid morphisms]
Let $F\colon \cK\to\cG$ be a Lie groupoid morphism over $Id_M$.
Applying Prop. \ref{prop:hbhfb} with $\cB_1:=\Gamma_c(Lie(\cK))$ 
{we obtain a canonical surjective morphism $$H^{\cK}(\cB_1)\to 
H^{\cG}(F_*(\cB_1)).$$
In particular the latter groupoid} is a quotient of {$H^{\cK}(\cB_1)=K$ (here we used Ex. \ref{ex:GammaA})}, recovering part of Cor. \ref{cor:HBimage}.
\end{ex}

\begin{ex}[The underlying singular foliation $\cF_{\cB}$]\label{ex:folia}
Let $\cG\rightrightarrows M$ be a Lie groupoid, and $\cB$ a singular subalgebroid of $A:=Lie(\cG)$. Consider the Lie groupoid morphism
$$F=(\bt,\bs)\colon \cG \to M\times M$$ over the identity $Id_M$. The corresponding Lie algebroid morphism is   the anchor map $\rho \colon A\to TM$, hence $F_*(\cB)=\cF_{\cB}$, the singular foliation on $M$ induced by $\cB$.
Prop. \ref{prop:hbhfb} implies the existence of a canonical surjective morphism of topological groupoids 
\begin{equation}\label{eq:hghfact}
H^{\cG}(\cB)\to H(\cF_{\cB})
\end{equation}
to the holonomy groupoid \cite{AndrSk} of the singular foliation $\cF_{\cB}$.

{We describe explicitly the kernel of the morphism $H^{\cG}(\cB)\to H(\cF_{\cB})$. In the special case that 
$H^{\cG}(\cB)$ is a Lie groupoid,
the kernel is 
$$\{h\in H^{\cG}(\cB): \text{$\exists$ a local bisection through $h$ carrying the identity diffeomorphism on $M$}\}.$$
Indeed, when $H^{\cG}(\cB)$ is a Lie groupoid\footnote{As an aside, this is equivalent to $\cB$ being  a projective singular subalgebroid \cite[\S 3.2]{AZ4}.}
the canonical morphism $\Phi\colon H^{\cG}(\cB)\to G$ is a Lie groupoid morphism giving rise to $\cB$ 
\cite[\S 3.2]{AZ4}. Hence it is a bisubmersion for $\cB$, by Prop. \ref{prop:imagerelbi}. Actually it is an atlas of bisubmersions, equivalent to the path-holonomy atlas (the argument is exactly as  in \cite[Ex 3.4(4)]{AndrSk}). The construction of the morphism in Prop. \ref{prop:hbhfb}  then implies the claim.}

{In the general case, the kernel is still described as indicated above, but 
the smoothness of the bisection is to be understood w.r.t. the canonical diffeology on $H^{\cG}(\cB)$ \cite[\S 4.2]{AZ4}.
}
\end{ex}

\begin{remark} 
{We give an alternative description of the morphism \eqref{eq:hghfact}  in the special case that 
$\cB$  arises from a Lie groupoid morphism (see Def. \ref{def:arises}), i.e
 there is a Lie groupoid morphism  $\Psi\colon \cK\to\cG$ over $Id_M$ such that   $\cB=\Psi_*(\Gamma_c(Lie(\cK)))$.
The composition $$\cK\overset{\Psi}{\to} G\overset{(\bt_G,\bs_G)}{\to} M\times M$$
 is the target-source map $(\bt_K,\bs_K)$ of $\cK$. It gives rise to the singular foliation $\cF_\cB$, hence
$H(\cF_\cB)=\cK/\sim_2$ where $\sim_2$ {identifies two points of $\cK$ if{f} there are bisections through them that carry the same diffeomorphism of $M$ (see Example \ref{ex:explicitgr})}. Since by assumption $\cB$ arises from the Lie groupoid morphism $\Psi$,  
by Prop. \ref{cor:HBimage}   we have
$H^{\cG}(\cB)=\cK/\sim_1,$  where $\sim_1$ {identifies two points of $\cK$ if{f} there are bisections through them that map under $\Psi$ to the same bisection of $G$ 
 (see Rem. \ref{rem:bisec})}. The equivalence classes of $\sim_1$ are contained in those of $\sim_2$, giving rise to a morphism $H^{\cG}(\cB)\to H(\cF_{\cB})$. The latter agrees with
the morphism \eqref{eq:hghfact}.
}
\end{remark}

\begin{ex}[{Covers of the Lie groupoid $\cG$}]
Let $$F\colon  \widetilde{\cG}\to\cG$$ a morphism of Lie groupoids integrating the identity on the Lie algebroid $A:=Lie(\widetilde{\cG})=Lie(\cG)$. Let $\cB$ be a singular subalgebroid of $A$. 
Clearly, $F_*(\cB)=\cB$. Prop. \ref{prop:hbhfb}  implies the existence of a canonical surjective  morphism of topological groupoids 
$$ {H}^{\widetilde{\cG}}(\cB)\to H^{\cG}({\cB})$$ from the holonomy groupoid of $\cB$ constructed using $\widetilde{\cG}$ to the 
 holonomy groupoid of $\cB$ constructed using ${\cG}$. {Under the identification of 
Theorem \ref{thm:tilde}, this map is just the first projection, since both maps are induced by the identity map at the level of bisubmersions.}
{Hence this map is exactly the one discussed in Remark \ref{rem:ker}.}
\end{ex}

 {We now present examples of Thm. \ref{thm:morph}. }
 We start with a simple statement about singular foliations:
 \begin{ex}[Singular foliations]\label{ex:morsingfol}
Let $\cF_1$, $\cF_2$ be singular foliations on a manifold $M$, with $\cF_1 \subset \cF_2$.
Then there is a canonical  morphism of topological groupoids $H(\cF_1)\to H(\cF_2)$ covering the identity. {This follows applying Thm. \ref{thm:morph} with $F$  the identity map on the pair groupoid $M\times M$.}
\end{ex}

{The following example} shows that the canonical map $\Phi\colon H^{\cG}(\cB)\to \cG$ arises 
from the inclusion $\cB\hookrightarrow \Gamma_c(A)$.

\begin{ex}[{Recovering $\Phi$}]\label{ex:Phi}
{Let $G$ be a Lie groupoid, and $\cB$ be a {singular} subalgebroid of $A:=Lie(G)$.} By Lemma \ref{prop:morph0}, the inclusion $\cB\subset \Gamma_c(A)$ induces a  
 canonical  morphism of topological groupoids $H^{\cG}(\cB)\to H^{\cG}(\Gamma_c(A))$ making the following diagram commute:
\begin{equation*}
\xymatrix{
H( {\cB})  \ar[rd]_{\Phi}    \ar@{-->}[rr] & & H^{\cG}(\Gamma_c(A))  \ar[ld]  \\
&\cG & }
\end{equation*}
But $H^{\cG}(\Gamma_c(A))=\cG$ and the right map above is $Id_{\cG}$, by Ex. \ref{ex:GammaA}. Hence the canonical morphism dotted above {is exactly $\Phi$}.
\end{ex}

{Specializing Thm. \ref{thm:morph} to wide Lie subalgebroids and using Prop. \ref
{prop:HGBisHmin} to relate holonomy groupoids with minimal integrals we obtain the following example.
\begin{ex}
  Let $F\colon \cG_1\to \cG_2$ be a morphism of Lie groupoids covering the identity on $M$. Let $B_i$ be a wide Lie subalgebroid of $Lie(\cG_i)$ for $i=1,2$,
 such that
$F_*(B_1)\subset B_2$. Denote by $H_{min}^i$ the minimal integral of $B_i$ over $G_i$.
Then there is a canonical morphism of Lie groupoids $$\Xi \colon H_{min}^1\to H_{min}^2$$ covering $Id_M$ and which, together with $F$, intertwines the  immersions $H_{min}^i\to G_i$ integrating the inclusions.
\end{ex}
 }

\subsection{The integration functor}
\label{sec:functor}

The purpose of this subsection is put in place our ``integration'' process, describing it as a functor. The term ``integration''  is in quotes, since we have not specified in which sense the holonomy groupoid $H^{\cG}(\cB)$ is an integration of the singular subalgebroid $\cB$ of $Lie(\cG)$.  
Indeed this is the topic of
a separate publication \cite{AZ4}.

We fix a manifold $M$ and  consider two categories. The first one, denoted by $\textsf{SingSub}^{Gpd}_M$, is:
\begin{itemize}
\item objects: 
\vspace{-1mm}\begin{align*}
\{(\cG,\cB)|\;& \cG \text{ a  Lie groupoid over $M$},\\ 
&\cB \text{ a singular subalgebroid of } Lie(\cG)\}
\end{align*}
\item arrows from $(\cG_1,\cB_1)$ to $(\cG_2,\cB_2)$: 
\vspace{-1mm}
\begin{align*}
\{F\colon \cG_1\to \cG_2 &\text{ a morphism of Lie groupoids covering $Id_M$},\\ &\text{ such that } F_*(\cB_1)\subset \cB_2\}
\end{align*}
\end{itemize}

The second
category,  denoted by $\textsf{TopGrd}_M$, is:
\begin{itemize}
\item objects:
\vspace{-1mm}
 \begin{align*}
\{\Phi\colon H\to \cG|&\; H \text{ a topological groupoid over $M$},\\ 
& \;\cG \text{ a Lie groupoid over $M$,}\\
&\; \Phi  \text{ a morphism of topological groupoids covering $Id_M$}\}
\end{align*}
\item arrows from $(\Phi_1\colon H_1\to \cG_1)$ to  $(\Phi_2\colon H_2\to \cG_2)$: 
\vspace{-1mm}
\begin{align*}
\{(\Xi,F)|& \;\Xi\colon H_1\to H_2 \text{ a morphism of topological groupoids over $Id_M$},\\ 
&\; F \colon \cG_1\to \cG_2 \text{ a morphism of Lie groupoids over $Id_M$},\\ 
& \;\text{s.t.  
the diagram below commutes}\}
\end{align*}
\begin{equation*} 
\xymatrix{
H_1  \ar[d]_{\Phi_1}    \ar[r]^{\Xi} &H_2  \ar[d]^{\Phi_2}     \\
\cG_1 \ar[r]^{F} &   \cG_2 }
\end{equation*}
\end{itemize}

Our construction  provides a functor 
\begin{align*}
 \text{$\textsf{SingSub}^{Gpd}_M$}\;\;\;\;\;&\to\;\;\;\;\;\text{$\textsf{TopGrd}_M$}\\
(\cG,\cB)\;\;\;\;\;&\mapsto\;\;\;\;\;  H^{\cG}(\cB)\\
F\;\;\;\;\;&\mapsto\;\;\;\;\;  (\Xi, F)
\end{align*}

where $H^{\cG}(\cB)$ is constructed as  in Definition \ref{def:holgroupoid} and 
$\Xi$ is constructed as in Theorem \ref{thm:morph}. This is really a functor, due to the canonicity of our constructions.

 \appendix
\section{The convolution algebra of a singular subalgebroid\\ (by Iakovos Androulidakis)}\label{section:convsing}

Here we explain how the construction of the $C^*$-algebra(s) of a singular foliation given in \cite[\S 4]{AndrSk} can be adapted to the context of singular subalgebroids. The only thing that needs to be explained here is the construction of the $\ast$-algebra associated with a singular subalgebroid $\cB$; its completion(s) is exactly as given in \cite[\S 4.4, \S 4.5]{AndrSk}. In fact, this appendix aims to exhibit that it is possible to do Noncommutative Geometry with singular subalgebroids; {recall from \cite{AndrSk}, \cite{IakAnal}, \cite{PseudodiffCalcSingFol} that this convolution algebra is the starting point in order to develop longitudinal pseudodifferential operators and the associated index theory.}

 {The $\ast$-algebra of a singular subalgebroid $\cB$ (Def. \ref{def:*algcB}) is constructed using the holonomy groupoid $H^G(\cB)$, or more precisely a path holonomy atlas
for $\cB$.}

\begin{remark}
It turns out that a singular subalgebroid $\cB$ corresponds to a singular foliation $\rar{\cB}$ on $\cG$. Hence, following  \cite{MMRC}, we can
consider the holonomy groupoid $H(\rar{\cB})$ of the singular foliation. The latter has an induced action of $\cG$, and the quotient is the holonomy groupoid $H^G(\cB)$, as 
is shown in \cite{AZ4}. This construction is satisfactory from a geometric point of view, but is less suited for the purposes of $*$-algebras, for it is not clear that the $*$-algebra associated to $H(\rar{\cB})$ and the $\cG$ action on $H(\rar{\cB})$ give rise to the $*$-algebras of
$H^G(\cB)$.
\end{remark}

\subsection*{Remarks on the topology of the holonomy groupoid}

Just like the case of singular foliations, the topology of $H^G(\cB)$ is quite bad. 
Here we extend some statements from  \cite[\S 3.3]{AndrSk}.
We will need this material for the construction of the convolution algebra.
{In fact, our constructions can be carried out for the groupoid associated to any atlas of $\cB$ (\cf Appendix \ref{sec:atlases}), not only the path-holonomy one, so we give them in full generality.}

Fix a singular subalgebroid $\cB$.
Fix an atlas $\cU = (U_i,\varphi_i,\cG)_{i \in I}$ (see Appendix \ref{sec:atlases})   and let $H(\cB)^{\cU}$ be the associated groupoid. For every bisubmersion $(U,\varphi,\cG)$ {adapted to $\cU$,} consider the set $$\Gamma_{U} = \{u \in U : \text{{dim$(U)$} = dim$(M)$ + dim$(\cB_{\bs(u)})$}\} \subset U,$$ 
where $\cB_{\bs(u)}=\cB/I_{\bs(u)}\cB$ for $I_{\bs(u)}$ the smooth functions on $M$ vanishing at {$\bs(u)$}. {Note that $\Gamma_{U}$  is either
  empty or it} consists of a union of fibers of $\bs\colon U\to M$. It is an open subset of $U$ when $U$ is endowed with the smooth structure along the leaves of the regular foliation $\varphi^{-1}(\rar{\cB}) = \Gamma_c(U;\ker d\bs_U)$, {\ie the \emph{longitudinal smooth structure} in \cite[Prop. 1.14]{AndrSk}.}} 
  
\begin{lemma}\label{lem:sm1}
For every $h \in H(\cB)^{\cU}$ there exists a bisubmersion $(U,\psi,\cG)$ {adapted to $\cU$} and $u \in \Gamma_U$ such that $q_{U}(u)=h$.
\end{lemma} 
\begin{proof}
Let $(W,\varphi,G)$ be a bisubmersion {in $\cU$, let} $w \in W$ such that $q_W(w)=h$ and $\bb\subset W$ a bisection through $w$. 
Consider the minimal path holonomy bisubmersion $(U, \psi_0,\cG)$ constructed by a basis of $\cB_{\bs(h)}$ (see Definition \ref{dfn:pathhol} and   Proposition \ref{prop:pathhol}). Then $u:=(\bs(h),0)\in U$ carries the identity bisection of $\cG$.
Consider the map $\psi:=L_{\varphi(\bb)}\circ\psi_0$, where $L_{\varphi(\bb)}$ is the diffeomorphism of $G$ defined by left multiplication by the bisection $\varphi(\bb)$.
{As shown in {Remark \ref{rem:ccinv}}} we obtain a new bisubmersion $(U,\psi,\cG)$.
{The points $u\in (U,\psi,\cG)$ and $w\in W$ both carry the bisection $\varphi(\bb)$ of $G$, hence under the quotient map to {$H(\cB)^{\cU}$}, the point $u$ also maps to $h$. Further $u\in \Gamma_{U}$, because
the path holonomy bisubmersion $(U, \psi_0,\cG)$ is minimal at $\bs(h)$, and
since the source maps of $(U, \psi_0,\cG)$
and $(U,\psi,\cG)$ agree.} 
\end{proof}

\begin{lemma}\label{lem:sm2}
Consider bisubmersions $(U,\varphi,\cG)$ and $(U',\varphi',\cG)$ and $f\colon U \to U'$ a morphism of bisubmersions. Let $u \in U$.
\begin{enumerate}
\item If $u \in \Gamma_U$, then $(df)_u$ is injective;
\item If $f(u) \in \Gamma_{U'}$ then $(df)_u$ is surjective
\end{enumerate}
\end{lemma}
\begin{proof}{We denote by $\bs$ and $\bs'$  the source maps of $U$ and $U'$.} 
Since $\bs$ and $\bs'$ are submersions and $\bs'\circ f = \bs$, $(df)_u \colon T_u U \to T_{f(u)}U'$ is injective or surjective if and only if the restriction $(df)_u|_{\ker (d\bs)_u} \colon \ker (d\bs)_u \to \ker (d\bs')_{f(u)}$ is injective or surjective. Consider the composition $$\ker (d\bs)_u \stackrel{(df)_u|_{\ker (d\bs)_u}}{\longrightarrow} \ker (d\bs')_{f(u)} \stackrel{\varphi'_*}{\longrightarrow} {\vec{\cB}_{\varphi(u)}}.$$
 By the definition of bisubmersions, the maps $\varphi'_*$ and $\varphi_* = \varphi'_* \circ (df)_u$ are onto.
 
 If $u \in \Gamma_U$, then {by dimension reasons} $\varphi_* \colon \ker(d\bs)_u \to \rar{\cB}$ is an isomorphism, {implying a)}.
 
 If $f(u) \in \Gamma_{U'}$ then $\varphi'_* \colon \ker(d\bs')_{f(u)} \to \rar{\cB}{{_{\varphi(u)}}}$ is an isomorphism, {implying b)}.
\end{proof}
 
\subsection*{Preliminaries on densities}

Let $(U,\varphi_U,\cG)$ be a bisubmersion. Just like in \cite[\S 4]{AndrSk}, we are going to work with bundles of ({complex}) $a$-densities ($a \in \R$) associated with a vector bundle $E \to U$. When $E=(\ker d\bs_U \times \ker d\bt_U) \to U$, we write $\Omega^a(U)$. Let us also denote $\Gamma_c(\Omega^a(U))$ the space of compactly supported sections of $\Omega^a(U)$. Now put $a = \frac{1}{2}$ and recall from \cite[\S 4.1, \S 4.2]{AndrSk}:
\begin{enumerate}
\item If $(V,\varphi_V,\cG)$ is another bisubmersion, then $\hd{U\circ V}_{(u,v)} = (\hd{U})_u \otimes (\hd{V})_v$  
 {at all $(u,v)\in U\circ V$.} So if $f \in \shd{U}$ and $g \in \shd{V}$, we  obtain an element $f \otimes g \in \shd{U\circ V}$ defined by
 $$f \otimes g \colon (u,v) \mapsto f(u)\otimes g(v).$$
\item If $\kappa \colon U \to U^{-1}$ is the identity isomorphism {(where $U^{-1}$ denotes the inverse bisubmersion)}
we put 
\begin{align*}
  \shd{U}&\to \shd{U^{-1}}\\
 f &\mapsto{f^{\ast}:=\overline{f}\circ \kappa^{-1} }
\end{align*}
\item If $p \colon U \to V$ is a submersion and a morphism of bisubmersions we have $\hd{U}=\Omega^1(\ker dp) \otimes p^{\ast}(\hd{V})$.
 So integration along the fibers of $p$ gives a map $$p_{!} \colon\shd{U} \to \shd{V}$$ If $p$ is a surjective submersion then $p_{!}$ is onto.
\end{enumerate}

\subsection*{The \texorpdfstring{$\ast$}{Lg}-algebra of an atlas of bisubmersions}\label{sec:Lgalgebra}

Let us fix an atlas of bisubmersions $\cU = (U_i,\varphi_i,\cG)_{i \in I}$ for the singular subalgebroid $\cB$. Consider the disjoint union $U:=\coprod_{i \in I}U_i$ and $\varphi \colon U \to \cG$ the map defined by $\varphi|_{U_i}=\varphi_i$. It is easy to see that $(U,\varphi,\cG)$ is a bisubmersion and $\shd{U}=\bigoplus_{i \in I}\shd{U_i}$.

{For the construction of the convolution algebra, will have to identify smooth densities between two different bisubmersions, and this can be done by means of integration along the fibers of a submersion. Since there exist bisubmersions which are adapted to $\cU$, but not necessarily through a submersive morphism, the next lemma is in order. Its proof is exactly the same as \cite[Lem. 4.3]{AndrSk}, using our Lemma \ref{lem:sm1}, so we omit it.}

\begin{lemma}\label{lem:algebra}
Let $(V,\varphi_V,\cG)$ be a bisubmersion adapted to $\cU$. The following two statements hold:
\begin{enumerate}
\item Let $v \in V$. Then there exists a bisubmersion $(W,\varphi_W,\cG)$ and submersions $p \colon W \to U$, $q \colon W \to V$ which are morphisms of bisubmersions such that $v \in q(W)$.
{
\begin{equation*}  
 \xymatrix{
 & W \ar[dl]_{q} \ar[dr]^{p}&\\
 V&&U 
 }
\end{equation*}
}

\item Let $f \in \shd{V}$. Then there exists a bisubmersion $(W,\varphi_W,\cG)$ and submersions $p \colon W \to U$, $q\colon W \to V$ which are morphisms of bisubmersions and $g \in \shd{W}$ such that $q_{!}(g)=f$.
\end{enumerate}
\end{lemma}

Analogously to \cite[\S 4.3]{AndrSk}, Lemma \ref{lem:algebra} allows us to consider the quotient {vector space}
$$\cA_{\cU} = \bigotimes_{i \in I}\shd{U_i}/\cI$$ Here $\cI$ is the 
{linear}
 subspace spanned by $p_{!}f$, where $p\colon W \to U$ is a submersion and morphism of bisubmersions and $f \in \shd{W}$ is such that there exists a submersion and morphism of bisubmersions $q\colon W \to V$ with $q_{!}(f)=0$.
 
\begin{remark}
{A bisubmersion $(U,\varphi,\cG)$ for the singular subalgebroid $\cB$ can be viewed a bisubmersion $(U,\bt_U,\bs_U)$ for the underlying foliation $\cF_{\cB}$, see Remark \ref{rem:FB}. In spite of this, 
the space $\cA_{\cU}$ constructed here differs from the one constructed in \cite[\S 4.3]{AndrSk}
out of the singular foliation $\cF_{\cB}$. That is because, while every morphism of bisubmersions $p\colon (U,\varphi,\cG) \to (U',\varphi',\cG)$ is a morphism of bisubmersions $p\colon(U,\bt_U,\bs_U) \to (U',\bt_{U'},\bs_{U'})$,   the converse is not true.}
\end{remark}
The proofs of \cite[Prop. 4.4, 4.5]{AndrSk} go through verbatim in the context of bisubmersions to endow $\cA_{\cU}$ with the following $\ast$-algebra structure:
\begin{itemize}
\item Given a bisubmersion $(V,\varphi_V,\cG)$ adapted to $\cU$ there is a linear map $Q_V\colon \shd{V} \to \cA_{\cU}$ such that:
\begin{enumerate}
\item If $(V,\varphi_V,\cG)=(U_i,\varphi_i,\cG)$ then $Q_V$ is the quotient map 
\item If $p\colon W \to V$ is a morphism of bisubmersions which is a submersion, then $Q_W = Q_V \circ p_{!}$.
\end{enumerate}
\item If $(V,\varphi_V,\cG)$, $(W,\varphi_W,\cG)$ are bisubmersions adapted to $\cU$ then for sections $f \in \shd{{V}}$, $g \in \shd{W}$, {the $*$-involution and product in $\cA_{\cU}$ are defined as follows}: $$(Q_V(f))^{\ast} = Q_{V^{-1}}(f^{\ast}), \qquad Q_V(f)Q_W(g)=Q_{V\circ W}(f\otimes g)$$
\end{itemize}
\begin{definition}\label{def:*algcB}
\begin{enumerate}
\item The \emph{$\ast$-algebra of the atlas $\cU$} is $\cA_{\cU}$ with the above operations. 
\item The \emph{$\ast$-algebra $\cA(\cB)$ of the singular subalgebroid $\cB$} is $\cA_{\cU}$ for $\cU$ the path-holonomy atlas of bisubmersions.
\end{enumerate}
\end{definition}
The completion of the $\ast$-algebra $\cA_{\cU}$ is verbatim as explained in \cite[\S 4.4, \S 4.5]{AndrSk} (but using the foliation $\cF_{\cB}$ for the definition of the $L^1$-norm in \cite[\S 4.4]{AndrSk}). This way we obtain the full and reduced $C^{\ast}$-algebras $C^{\ast}(\cU)$ and $C^{\ast}_r(\cU)$ respectively. When $\cU$ is the path-holonomy atlas of bisubmersions associated with $\cB$, we write $C^{\ast}(\cB)$ and $C^{\ast}_r(\cB)$.

\begin{remark}
{
The $C^*$-algebra $C^{\ast}(\cB)$  enjoys functorial properties. Consider for instance the anchor  $\rho \colon A\to TM$,  a Lie algebroid morphism
which maps  a given singular subalgebroid  $\cB$ to the underlying singular foliation $\cF_{\cB}$.  There is an  induced canonical morphism $H^{\cG}(\cB)\to H(\cF_{\cB})$, as we saw in Ex. 
\ref{ex:folia}. One can show that the anchor also induces a 
$\ast$-homomorphism $C^{\ast}(\rho) : C^{\ast}(\cB) \to C^{\ast}(\cF_{\cB})$,
given by integration along fibers of  submersive morphisms between bisubmersions of $\cB$ and of $\cF_{\cB}$.}
\end{remark}

\section{{A technical lemma about bisubmersions}}
\label{app:bisum}

The proof of Prop. \ref{prop:preserveb} relies on the following proposition, which relates bisubmersions for a \emph{singular subalgebroid} on $M$ with bisubmersions for the induced \emph{singular foliation} on $G$.

\begin{prop}\label{lem:bisubrarb} 
Let   $(U,\varphi,\cG)$ be a bisubmersion for a singular subalgebroid $\cB$.
Then  $$\widehat{U}:=U \times_{\bs_U, \bt}G,$$
with target and source maps   
\begin{align*}
\bt_{\widehat{U}}(u,g):=&\;\varphi(u)\cdot g\\
 \bs_{\widehat{U}}(u,g):=&\;g
\end{align*}  
to $G$, is a bisubmersion  for the singular foliation $\rar{\cB}$ on $G$ (in the sense of \cite{AndrSk}, see  \S \ref{section:usualbisub}). 
\end{prop}

The main step in the proof of Prop. \ref{lem:bisubrarb} is to check that eq. \eqref{eq:bisubfol} holds, \ie that

\begin{equation}\label{eq:bisubfolwidehat}
\bt_{\widehat{U}}^{-1}(\rar{\cB})=\Gamma_c(\ker d\bt_{\widehat{U}}) + \Gamma_c( \ker d\bs_{\widehat{U}})=\bs_{\widehat{U}}^{-1}(\rar{\cB}).
\end{equation}

We start making these submodules more explicit.
By Lemma \ref{lem:fix} we have
\begin{align}\nonumber
\bt_{\widehat{U}}^{-1}(\rar{\cB})&=Span_{C^{\infty}_{c}(\widehat{U})} \bigcup_{\balpha \in \cB}\big\{(Y,Z): Y\in \vX({U}), Z\in \Gamma (\ker d\bs_{U}) 
\text{ s.t. $(d_u\bs_U)Y=(d_g\bt)Z$ for all $(u,g)\in \widehat{U}$}\\
&\label{eq:btwub}
\qquad\qquad\qquad\qquad\qquad\qquad\qquad\qquad\qquad\qquad\;\;\;\text{ and $(Y,Z)$ is  $\bt_{\widehat{U}}$-related to $\rar{\balpha}$}\big\}\\
\label{eq:bswub}
\bs_{\widehat{U}}^{-1}(\rar{\cB})&= Span_{C^{\infty}_{c}(\widehat{U})} \bigcup_{\balpha \in \cB}\big\{(Y,\rar{\balpha}): Y\in \vX(U) \text{ s.t. $(d_u\bs_U)Y=(d_g\bt)\rar{\balpha}$ for all $(u,g)\in \widehat{U}$}\big\}.
\end{align}

\begin{lemma}\label{lem:kerbtw}
$$\Gamma_c(\ker d\bt_{\widehat{U}})=Span_{C^{\infty}_{c}(\widehat{U})}\bigcup_{\balpha \in \cB} \big\{
(Y,-\rar{\balpha}): Y\in \Gamma(\ker d\bt_{U}) \text{ s.t. $Y$ is $\varphi$-related to $\lar{\balpha}$}\big\}.$$
\end{lemma}

\begin{proof}
At every point $(u,g)\in \widehat{U}$, we have 
$$\ker d_{(u,g)}\bt_{\widehat{U}}=\{(Y,Z): Y\in \ker(d_u\bt_U), Z\in \ker(d_g\bs) \text{ s.t. $d_u\bs_{U}(Y)=d_g\bt(Z) \text{ and }(d_u\varphi (Y))\cdot Z_g=0$}\}.$$

The inclusion ``$\supset$'' in the statement of the lemma holds because 
$\lar{\balpha}\cdot (-\rar{\balpha})=\lar{\balpha}\cdot (i_*\lar{\balpha})=0$
for all  $\balpha \in \cB$, where $i$ is the inversion on $G$. For the opposite inclusion, let $(Y,Z)\in \Gamma_c(\ker d\bt_{\widehat{U}})$. In particular 
$Y\in  {\Gamma_c(U,\ker d\bt_U)}$, so
by Lemma \ref{lem:bisubm2} b) (applied to $\bt_U$) we have $Y=\sum f_i Y_i$ where $f_i\in C^{\infty}_c(U)$ and the $Y_i\in  {\Gamma(U,\ker d\bt_U)}$ are $\varphi$-related to  $\lar{\balpha_i}$ for suitable  ${\balpha_i\in \cB}$. As seen earlier, each $(Y_i,-\rar{\balpha_i})\in \Gamma(\ker d\bt_{\widehat{U}})$, so the difference
$(Y,Z)-\sum (pr_U)^*f_i \cdot (Y_i,-\rar{\balpha_i})$ lies in $\Gamma(\ker d\bt_{\widehat{U}})$ too. Now this difference is of the form $(0,*)$, and the condition  $ (d\varphi (0))\cdot *=0$ implies that $*=0$. Hence this difference is zero. Multiplying by an element of  $C^{\infty}_c(\widehat{U})$ which is $1$ on $supp(Y,Z)$, we are done.
\end{proof}

\begin{proof}[Proof of Prop. \ref{lem:bisubrarb}]
The map $\bs_{\widehat{U}}$ is a submersion because $\bs_U$ is. We now show that 
$\bt_{\widehat{U}}\colon \widehat{U}\to G$ is a submersion, by showing that its derivative at any point $(u,g)\in \widehat{U}$ is surjective. Let $\bb\colon M\to U$ be a bisection of $U$ through $u$. Then $\bc:=\varphi\circ \bb\colon M\to G$ is a bisection of $G$ through $\varphi(u)$. Since the left multiplication $L_{\bc}$ is a diffeomorphism, any vector in $T_{( \varphi(u)\cdot g)}G$ can be written 
as $(L_{\bc})_*v$ for some $v\in T_gG$, \ie as $(\bc_*\bt_*v)\cdot v$, which is the image under $(\bt_{\widehat{U}})_*$ of $(\bb_*\bt_*v, v)\in T_{(u,g)}\widehat{U}$.

We now prove the first equality in eq. \eqref{eq:bisubfolwidehat}. For ``$\supset$'', it suffices to show that $\bt_{\widehat{U}}^{-1}(\rar{\cB}) \supset\Gamma_c(\ker d\bs_{\widehat{U}})$. By Lemma \ref{lem:bisubm2} b), it suffices to consider $(Y,0)$ where $Y\in {\Gamma(U,\ker d\bs_U)}$ is $\varphi$-related to $\rar{\balpha}$ for some $\balpha\in \cB$. Such an element   
is as on the r.h.s. of eq. \eqref{eq:btwub},
since $\rar{\balpha}\cdot 0=\rar{\balpha}$. For ``$\subset$'', let $(Y,Z)$ be as in the r.h.s. of eq. \eqref{eq:btwub}. Since $(U,\varphi,G)$ is a bisubmersion for $\cB$, there is $Y'\in \Gamma(\ker d\bs_{U})$ which is $\varphi$-related to $\rar{\balpha}$. Then $(Y',0)$ is $\bt_{\widehat{U}}$-related to $\rar{\balpha}$. Hence the difference satisfies $(Y,Z)-(Y',0)\in \Gamma(\ker d\bt_{\widehat{U}}),$ and therefore
$(Y,Z)\in \Gamma(\ker d\bt_{\widehat{U}}) + \Gamma( \ker d\bs_{\widehat{U}})$. 
Taking  $C^{\infty}_c(\widehat{U})$-linear combinations we are done.

We are left with proving the second equality in eq. \eqref{eq:bisubfolwidehat}.
For ``$\subset$'', it suffices to show that $\Gamma_c(\ker d\bt_{\widehat{U}}) \subset\bs_{\widehat{U}}^{-1}(\rar{\cB})$, which is easily seen to hold using Lemma \ref{lem:kerbtw} and $i_*\lar{\balpha}=-\rar{\balpha}$. For the inclusion ``$\supset$'', it suffices to consider elements $(Y,\rar{\balpha})$ as on the r.h.s. of eq. \eqref{eq:bswub}.
Since $(U,\varphi,G)$ is a bisubmersion for $\cB$, there is $Y'\in \Gamma(\ker d\bt_{U})$ which is $\varphi$-related to $-\lar{\balpha}$. By Lemma \ref{lem:kerbtw} we have $(Y', \rar{\balpha})\in \Gamma(\ker d\bt_{\widehat{U}})$. Since $Y-Y'\in \Gamma (\ker d\bs_{U})$, the difference satisfies
$(Y, \rar{\balpha})-(Y', \rar{\balpha})=(Y-Y',0)\in \Gamma(\ker d\bs_{\widehat{U}})$.
Therefore $(Y, \rar{\balpha})\in \Gamma(\ker d\bt_{\widehat{U}}) + \Gamma( \ker d\bs_{\widehat{U}}).$
\end{proof}

\section{On atlases and holonomy groupoids}\label{sec:atlases}

We recall some material from \cite[\S 3.1]{AndrSk}, spelling out part of it and rephrasing it in the context of singular subalgebroids (rather than singular foliations). {This material is used in Prop. \ref{prop:B}, and in the construction of morphisms between holonomy groupoids (covering the identity in \S \ref{section:morph} and
 covering  submersions in Appendix \ref{sec:appmorsub}).}
 
Fix a singular subalgebroid $\cB$ of a Lie algebroid $A$, and a Lie groupoid $\cG$ integrating $A$. 
 
\begin{definition}\label{def:atlas} 
Let $\cU:=(U_i,\varphi_i,\cG)_{i\in I}$ be a family of   bisubmersions for $\cB$.
\begin{enumerate}
\item Let $(U,\varphi,\cG)$ be a   bisubmersion of $\cB$. We say that $(U,\varphi,\cG)$ is {\bf adapted} to $\mathcal{U}$   if for every point $u\in U$ there is an open subset $U'\subset U$ containing $u$, an index $i \in I$ and a morphism of   bisubmersions $U' \to U_i$.
\item  $\cU$ is an {\bf atlas} if
\begin{enumerate}
\item $\cup_{i \in I}\bs_{i}(U_i)=M$,
\item the inverse of every element of $\cU$ is adapted to $\cU$,
\item the composition of any two elements of $\cU$ is adapted to $\cU$.
\end{enumerate}
\item Let $\cU$ and $\cV$ be two atlases. $\cU$ is adapted to $\cV$ is every element of $\cU$ is adapted to $\cV$. The atlases $\cU$ and $\cV$ are equivalent if they are adapted to each other. 
\end{enumerate}
\end{definition}

Let $\cU$ be an atlas for $\cB$.   
By Corollary \ref{cor:crucial} c),  the relation  
 \begin{align*}
u_1 \sim u_2 \Leftrightarrow &\;\;\text{there is an open neighborhood  $U'_1$ of $u_1$,}\\
&\;\;\text{there is a morphism of bisubmersions } f\colon U_1' \to U_2 \text{  such that } f(u_1)=u_2,   
\end{align*}

 is an equivalence relation on {the disjoint union} $\coprod_{U\in \cU}U$. (It can be expressed in terms of bisections too, as in Rem. \ref{rem:HGBbisec}.) The quotient
 \begin{equation}\label{eq:HUcb}
H(\cB)^{\cU}:=\coprod_{U\in \cU}U/\sim
\end{equation}
 is a topological groupoid.  It comes with a canonical morphism of topological groupoids $\Phi\colon H(\cB)^{\cU}\to \cG$, {induced by the bisubmersions $\varphi\colon U\to G$ in $\cU$.} {This is proven in a way similar to \cite[Prop. 3.2]{AndrSk} and Thm. \ref{thm:holgroidconstr}.}
 
 {We now display some properties of this construction.} 
The following is \cite[Rem. 3.3]{AndrSk}, with details added.
\begin{prop}\label{prop:rem33}
Let the atlas $\cU$ be adapted to the atlas $\cV$. Then
\begin{itemize}
\item [i)] there is a canonical, injective morphism of topological groupoids ${\Theta}\colon H(\cB)^{\cU}\to H(\cB)^{\cV}$
\item [ii)] {${\Theta}$ commutes with the canonical maps from $H(\cB)^{\cU}$ and $H(\cB)^{\cV}$ to $\cG$.}
\end{itemize}
 
\end{prop}
\begin{proof}
i)  Take an element $U$ of $\cU$ and $u\in U$. By definition, there is an open subset $U'\subset U$ containing $u$, an element $V_1$ of $\cV$, and a morphism of   bisubmersions $g_1\colon U' \to V_1$. Assume that there is another element $V_2$ of $\cV$, and a morphism of   bisubmersions $g_2\colon U' \to V_2$. Then there is a 
morphism of   bisubmersions $ V'_1\to V_2$ -- obtained  applying Cor. \ref{cor:crucial} c) to $g_1$ -- {defined on a neighborhood of $g_1(u)$ in $V_1$,} making this diagram commute:
\begin{equation*} 
 \xymatrix{
 &V_1 \ar@{-->}[dd]   \\
U' \ar[ur]^{g_1}\ar[dr]_{g_2} & \\
 &V_2\\
 }
\end{equation*}
This gives a well-defined map
$\coprod_{U \in \cU}U\to H(\cB)^{\cV}$, mapping a point $u_1$ to $[g_1(u_1)]$, where $g_1$ is any morphism of   bisubmersions into an element of $\cV$.

Now consider two morphisms of   bisubmersions $g_1 \colon U_1\to V_1$ and $g_2 \colon U_2\to V_2$, where $U_1,U_2$ are open subsets of elements of $\cU$, and $V_1,V_2\in \cV$. Suppose $f$ is a   morphism of   bisubmersions mapping a point $u_1\in U_1$ to $u_2\in U_2$. 
Then there is a  morphism of   bisubmersions mapping $g_1(u_1)$ to $g_2(u_2)$. This can be seen using
Cor. \ref{cor:crucial} c) to ``invert'' the top horizontal map in 
 \begin{equation*} 
 \xymatrix{
 {U_1}\ar^{f}[d]\ar[r]^{g_1}& {V_1}\ar@{-->}[d]\\ 
 {U_2}\ar[r]^{g_2}& {V_2}\\
 }
  \end{equation*}
Hence, by  ${\Theta}([u_1]):=[g_1(u_1)]$ we obtain a well-defined  map $ {\Theta}\colon H(\cB)^{\cU} \to  H(\cB)^{\cV}$.

The injectivity of ${\Theta}$ can be seen applying the above reasoning to the bottom horizontal map in the diagram
\begin{equation*} 
 \xymatrix{
 {U_1}\ar@{-->}[d]\ar[r] & {V_1}\ar[d]\\ 
 {U_2}\ar[r] & {V_2}\\
 }
  \end{equation*}

Last, we show that ${\Theta}$ is a groupoid morphism. We indicate only how to check that the groupoid multiplications are preserved. Again, consider two morphisms of   bisubmersions $g_1 \colon U_1\to V_1$ and $g_2 \colon U_2\to V_2$, where $U_1,U_2$ are open subsets of elements of $\cU$, and $V_1,V_2\in \cV$. 
Let $u_1\in U_1$ and $u_2\in U_2$ so that their product $[u_1]\cdot[u_2]$ in $H(\cB)^{\cU}$ is defined. Then $g_1\times g_2$,
which clearly maps $u_1\circ u_2$ to $g_1(u_1)\circ g_2(u_2)$,
 is a morphism of   bisubmersions,     as can be seen from the commutativity of the diagram below. 
\begin{equation*}
\xymatrix{
U_1\circ U_2 \ar[rd]_{  }  \ar[rr]^{g_1\times g_2} & & V_1\circ V_2 \ar[ld]^{} \\
&\cG\times_{\bs,\bt}\cG\ar[d]^{\text{multiplication}}&\\
&\cG& }
\end{equation*} Hence $${\Theta}([u_1]\cdot[u_2])=[(g_1\times g_2)(u_1\circ u_2)]=[g_1(u_1)\circ g_2(u_2)]={\Theta}([u_1])\cdot {\Theta}([u_2]).$$

ii) {The morphism ${\Theta}\colon H(\cB)^{\cU}\to H(\cB)^{\cV}$ commutes with the maps to $G$ because ${\Theta}$ is constructed assembling morphisms of bisubmersions, which by definition commute with the respective maps to $G$ (see the commutative diagram in Def. \ref{def:morph}).}
\end{proof}

\begin{prop}\label{prop:pathholadapted}
Let $\cU$ be a path holonomy atlas as in definition \ref{def:pathholatlas}.
Then $\cU$ is adapted to any atlas $\cV$. \end{prop}
 \begin{remark}
1) In particular, any two path holonomy atlases are adapted to each other. Proposition \ref{prop:rem33} implies that the corresponding topological groupoids agree. (We denote them $H(\cB)$ along the whole paper).

2) For any atlas $\cV$, there is a canonical
injective morphism of topological groupoids $H(\cB) \to H(\cB)^{\cV}$,
as a consequence of proposition \ref{prop:rem33} and proposition \ref{prop:pathholadapted}.

In the converse, we can only say that there is an injective  morphism of local topological groupoids from \emph{a neighborhood of the identity section} of $H(\cB)^{\cV}$ to $H(\cB)$. 
This can be seen using proposition \ref{prop:crucial} and proceeding as in the proof of proposition \ref{prop:rem33}.
\end{remark}

\begin{proof}
Let $U$ be a {minimal} path holonomy   bisubmersion lying in $\cU$ and $x\in M$ with $(0,x)\in U$. Take a preimage $v$ of the identity element $1_x$ under the quotient map $\natural \colon \coprod_{V \in \cV}V\to H(\cB)^{\cV}$, and suppose that  $v\in V$. Then by proposition \ref 
{prop:crucial} and Cor. \ref{cor:crucial} c)
there exists an open neighbourhood $U'$ of $(0,x)$ in $U$ and a morphism of   bisubmersions $U'\to V$ mapping $(0,x)$ to $v$.  
Repeating for all $x$, we see that there is a neighbourhood $N$ of $U\cap M$ in $U$, such that every point of $N$ lies in the domain of some morphism of   bisubmersions into some element of $\cV$. In other words, $N$ is adapted to $\cV$.

This implies that the same holds for any arbitrary point of $U$: if $(\lambda,x)\in U$, then there is a positive integer $k$ such that
$(\lambda/k,x)\in N$, and 
there is a neighborhood $U $ of $(\lambda,x)$ and a morphism of   bisubmersions $U \to N\circ\dots\circ N$ {mapping} $(\lambda,x)$ to
$(\lambda/k,\dots)\circ \dots(\lambda/k,\bt_U({(\lambda/k,x)}))\circ (\lambda/k,x)$. (It exists by Cor. \ref{cor:crucial} b) since the constant bisection of $U$ with value $\lambda$ and the constant bisection of $N\circ\dots\circ N$ with value $(\lambda/k,\dots,\lambda/k)$ map to the same bisection of $\cG$.)  
Using that $N$ is is adapted to $\cV$, we obtain a morphism of   bisubmersions $U'\to V_k\circ\dots\circ V_1$ for some elements $V_i$ of $\cV$. This shows that $U$ is adapted to $\cV$.

For elements of $\cU$ which are compositions  $U_n\circ\dots\circ U_1$ of path holonomy   bisubmersions, apply the above to each $U_i$. {Notice that we do not need to consider inverses of path holonomy bisubmersions, due to Remark. \ref 
{rem:invphbisub}.}\end{proof}

\begin{prop}\label{prop:equivpathhol}
Let $\cU$ be an atlas generated by path-holonomy   bisubmersions  as in definition \ref{dfn:pathhol} {(not necessarily minimal ones)}. 
Then $\cU$ is adapted to a path-holonomy atlas.
 \end{prop}

\begin{remark}\label{rem:HHu}
Such an atlas $\cU$ is   equivalent to a path-holonomy atlas, by proposition
\ref{prop:pathholadapted} and proposition \ref{prop:equivpathhol}, hence $H(\cB)^{\cU}=H(\cB)$.
\end{remark}

\begin{proof}
 Let  $\{\balpha_1,\dots,\balpha_n\}$ be a set of local generators  of $\cB$, giving rise to the   bisubmersion $(U,\varphi,\cG)$. {Let $x\in \bs_U(U)\subset M$.} We may assume that, for some $k\le n$,  
$\{\balpha_1,\dots,\balpha_k\}$ is a minimal set of local generators {at $x$ (see Remark \ref{rem:basis})}. So {on a neighborhood $M'$ of $x$}, for $a> k$, we can write $\balpha_a=\sum_{i=1}^k f^i_a\balpha_i$ for   functions $f^i_a$ on ${M'}$. It is straightforward to check that 
 $$\RR^n \times M\to \RR^k \times M, (\lambda_1,\cdots,\lambda_n; x)\mapsto (\lambda_1+\sum_{a=k+1}^n \lambda_a f^1_a, \cdots,\lambda_k+\sum_{a=k+1}^n \lambda_a f^k_a; x)$$
restricts to a morphism of bisubmersions from ${U|_{M'}}$ to the minimal path holonomy   bisubmersion constructed using $\{\balpha_1,\dots,\balpha_k\}$. 
 
As such   bisubmersions $(U,\varphi,\cG)$ generate the atlas $\cU$, we are done.
\end{proof}

\section{Morphisms  of holonomy groupoids  covering  submersions}\label{sec:appmorsub}

 {In \S \ref{section:HB}, starting from a Lie groupoid $\cG$ and singular subalgebroid $\cB$ of the Lie algebroid $Lie(\cG)$, we constructed a holonomy  groupoid $H^{\cG}(\cB)$ endowed with a map of topological groupoids to $\cG$. In \S \ref{section:morph} we extended this construction to morphisms \emph{covering the identity on the base}, obtaining canonical  morphisms of topological groupoids 
$$H^{\cG_1}(\cB_1)\to H^{\cG_2}(\cB_2)$$ commuting with the canonical maps.
In this appendix we do the same for morphisms \emph{covering surjective submersions}, see Theorem \ref{thm:morphsub}.} 
 
Let $F_* \colon A_1\to A_2$ be a morphism of Lie algebroids, \emph{covering a surjective submersion} $f \colon M_1\to M_2$. Let $\cB_1$ be a singular subalgebroid of $A_1$ satisfying the following condition:
\begin{align}\label{star}
& {\widehat{\cB_1}}^{proj}:=\{\balpha\in {\widehat{\cB_1}}: \balpha \text{ is $F_*$-projectable to a section of $A_2$}\}\\
&\;\;\;\;\;\;\;\;\;\;\;\;\;\text{generates $\cB_1$ as a ${C^{\infty}_c(M_1)}$-module.}\nonumber
\end{align}
{Recall that the global hull $\widehat{\cB_1}$ was defined in \S \ref{section:singdef}.}
Here by ``$\balpha$ is $F_*$-projectable to a section of $A_2$'' we mean that  there exists $\mathbf{b}\in \Gamma(A_2)$ such that $F_*(\balpha|_x)=\mathbf{b}|_{f(x)}$
for all $x\in M_1$. As  $f$ is surjective, such a section $\mathbf{b}$ is unique, and will be denoted by $F_*\balpha$.

 \begin{lemma}\label{lem:FB2}
Condition $\eqref{star}$ implies that $$F_*(\cB_1):=Span_{{C_c^{\infty}(M_2)}}\{F_*\balpha : \balpha\in {\widehat{\cB_1}}^{proj}\}$$ is a singular subalgebroid of $A_2$. 
\end{lemma}
 \begin{proof}
First notice that $F_*(\cB_1)$ is a well-defined $C^{\infty}(M_2)$-submodule of ${\Gamma_c(A_2)}$, since $f$ is surjective.

 $F_*(\cB_1)$  is involutive: the fact that $F_*$ is a morphism of Lie algebroids implies that if $\balpha,\mathbf{b}\in \Gamma(A_1)$ are $F_*$-projectable, then their bracket also is, and  $F_*[\balpha,\mathbf{b}]=[F_*\balpha,F_*\mathbf{b}]$ \cite[\S3.4, equation  (24)]{MK2}.

 We now show that $F_*(\cB_1)$ is locally finitely generated. {For all $x\in M_1$}, the restriction to 
 ${\widehat{\cB_1}}^{proj}$  of the linear map $\cB_1\to {\cB_1}/{I_x\cB_1}$   is surjective, as a consequence of condition $\eqref{star}$. 
Choosing $\{\balpha_1,\dots,\balpha_n\}\subset {\widehat{\cB_1}}^{proj}$ so that its image forms a basis of ${\cB_1}/{I_x\cB_1}$, we obtain a set of generators of $\cB_1$ near $x$ (see Remark \ref{rem:basis}, {which remains true for global hulls}). {We claim that} $\{F_*\balpha_1,\dots,F_*\balpha_n\}$ is a  set of generators of $F_*(\cB_1)$ near $f(x)$.
{To this aim, it suffices to show that $F_*\alpha$ is a $C^{\infty}(M_2)$-linear combination of the $F_*\balpha_i$'s, for all $ \balpha\in {\widehat{\cB_1}}^{proj}$. We can write 
\begin{equation}\label{ref:balphalc}
  \balpha=\sum g_i \balpha_i
\end{equation} nearby $x$, for some $g_i\in C^{\infty}(M_1)$.
Take a small enough submanifold $S$ of $M_1$ through $x$ which is transverse to the $f$-fibers (hence $f|_S$ is a diffeomorphism onto an open subset of $M_2$). Restricting the equation \eqref{ref:balphalc} to $S$ and applying $F_*$ gives the desired conclusion.}
\end{proof}

We display two classes of singular subalgebroids that satisfy condition $\eqref{star}$.
\begin{lemma}\label{lem:diffeo}
i) If $f$ is a diffeomorphism,  any singular subalgebroid $\cB_1$ satisfies condition $\eqref{star}$.

ii) Assume that  $F_* \colon A_1\to A_2$ has constant rank  and $\cB_2$ is a singular subalgebroid  of $A_2$,  such that   any element of $\cB_2$ can be $F_*$-lifted to a section of $A_1$.
Then  $$\cB_1:={Span_{C_c^{\infty}(M_1)}\{\balpha\in \Gamma(A_1): \text{$\balpha$ is $F_*$-projectable to an element of $\cB_2$} \},}$$ 
is a singular subalgebroid of $A_1$  satisfying condition $\eqref{star}$. Further  $F_*(\cB_1)=\cB_2$.
\end{lemma}

\begin{remark} 
If $F_*$ is fiberwise surjective, then any singular subalgebroid $\cB_2$ satisfies the assumptions of Lemma \ref{lem:diffeo} ii). The singular subalgebroid $\cB_1$ appearing there deserves to be called the \emph{pullback of $\cB_2$}. 
{Indeed, when $A_1=TM_1$ and $A_2=TM_2$ are tangent bundles, $\cB_1$ is exactly the pullback of the singular foliation $\cB_2$ by $f$, as defined in \cite[Proposition 1.10]{AndrSk}.} {See also Remark \ref{rem:Moritamor} below.}

\end{remark}
\begin{proof}
i) is clear since in that case ${\widehat{\cB_1}}^{proj}={\widehat{\cB_1}}$.
For ii) we proceed as follows.

We first check that $\cB_1$ is a singular subalgebroid of $A_1$. The involutivity of $\cB_1$ follows from that of $\cB_2$. To see that $\cB_1$ is locally finitely generated nearby a given point $x\in M_1$, choose a finite set of local generators $\{\mathbf{b}^i\}$ of $\cB_2$ in a neighborhood $V$ of $f(x)$.
{Shrinking $V$ if necessary, we can assume that the $\mathbf{b}^i$'s are compactly supported, and hence elements of $\cB_2$.}
By assumption, 
there are sections $\{\balpha^i\}\subset\Gamma(A_1)$ which lift the $\{\mathbf{b}^i\}$, \ie,  each $\balpha^i$ is $F_*$-projectable, and $F_*\balpha^i=\mathbf{b}^i$. 
We have $\balpha^i\in {\widehat{\cB_1}}$ by the definition of the latter. Now let $\{\mathbf{c}^j\}\subset\Gamma(\ker(F_*))$ be a finite set which forms a frame for the vector bundle $\ker(F_*)$  near $x$. It exists since $F_*$ has constant rank. We claim that $$\{\balpha^i\}\cup \{\mathbf{c}^j\}$$ is a generating set for $\cB_1$ near $x$. Indeed, given $\balpha\in \Gamma(A_1)$ which is  $F_*$-projectable to an element of $\cB_2$, there exist $\{g_i\}\subset C^{\infty}(M_2)$ such that $F_*\balpha=\sum g_i \mathbf{b}^i$ on $V$. On $f^{-1}(V)$, the section $\sum f^*(g_i) \balpha^i$  projects to $F_*\balpha$. Hence, in neighborhood of $x$, we have $\balpha=\sum f^*(g_i) \balpha^i+\sum h_j\mathbf{c}^j$ for certain $\{h_j\}\subset C^{\infty}(M_1)$. This shows that  $\cB_1$ is locally finitely generated nearby $x$.

By construction $\cB_1$ satisfies condition $\eqref{star}$. By definition $F_*\cB_1\subset \cB_2$, and the reverse inclusion holds since by assumption elements of $\cB_2$ can be lifted to sections of $A_1$.
 \end{proof}

\subsection*{ {Surjective morphisms}}
\label{sec:surmor}

The following proposition address surjective morphisms, and a special case   (corresponding to the case $f=Id$)  was already addressed in Prop. \ref{prop:hbhfb}.

\begin{prop}\label{prop:morph}
Let $F\colon \cG_1\to \cG_2$ be a morphism of Lie groupoids, covering a   surjective submersion $f \colon M_1\to M_2$.
Let $\cB_1$ be a singular subalgebroid of $Lie(\cG_1)$, and assume that it satisfies  condition $\eqref{star}$ above. 
Let  $\cB_2:=F_*(\cB_1)$, {which is a singular subalgebroid  of $ Lie(\cG_2)$ by Lemma \ref{lem:FB2}}.

Then there is a canonical, surjective morphism of topological groupoids $$\Xi \colon H^{\cG_1}(\cB_1)\to H^{\cG_2}(\cB_2)$$ making the following diagram commute:
\begin{equation}\label{diag:commtoG}
  \xymatrix{
H^{\cG_1}( {\cB}_1)  \ar[d]^{\Phi_1}   \ar@{-->}[r]^{\Xi} &H^{\cG_2}(\cB_2)  \ar[d]^{\Phi_2}     \\
\cG_1 \ar[r]^{F} &   \cG_2 }
\end{equation}
\end{prop}
 
 {To construct the morphism $\Xi$ we will relate {path-holonomy} bisubmersions for $\cB_1$ with 
 bisubmersions for $\cB_2$. We first need a lemma\footnote{{Lemma \ref{lem:U12} is concerned with path-holonomy bisubmersions. One can check that the statement of the lemma  holds for all $\cG_1$-bisubmersion for $\cB_1$, but we will not need this.}}

\begin{lemma}\label{lem:U12}
Assume the set-up of Prop. \ref{prop:morph}. Let  $\{\balpha_1,\dots,\balpha_n\}$ be a minimal set of local generators of $\cB_1$ {lying in $\widehat{\cB_1}^{proj}$}.
 Let $(U,\varphi_1,\cG_1)$ the corresponding  {path-holonomy} $\cG_1$-bisubmersion for $\cB_1$. Then $(U,\varphi_2:=F\circ \varphi_1,\cG_2)$ is a 
$\cG_2$-bisubmersion for $\cB_2$.
\end{lemma}
{The situation is visualized in the following diagram, which is helpful to follow  the proof of Prop. \ref{prop:morph} too.}
  \begin{equation*} 
 \xymatrix{
 U\ar_{
 \varphi_1}[d]\ar@{-->}[rd]^{\varphi_2}&\\
 {\cG_1} \ar[r]^{F}&\cG_2   }
 \end{equation*} 
 
\begin{proof}
Recall that $U$ is an open subset of $\RR^n \times M_1$, and by   definition \ref{dfn:pathhol} we have
$\varphi_1\colon U  \to \cG_1, (\lambda,x)\mapsto \exp_x \sum \lambda_i \overset{\rightarrow}{\balpha_i}$.
Hence 
\begin{equation}\label{eq:lambdax}
\varphi_2((\lambda,x))=F(\exp_x \sum \lambda_i \overset{\rightarrow}{\balpha_i})=\exp_{f(x)} \sum \lambda_i \overrightarrow{F_*(\balpha_i)}.
\end{equation}
Here the second equality holds because, since $F$ is a Lie groupoid morphism and $\balpha_i$ is $F_*$-projectable (to  $F_*(\balpha_i)\in \cB_2$), we have $F_*(\overset{\rightarrow}{\balpha_i})=\overrightarrow{F_*(\balpha_i)}$.

Let $(W,\varphi_W,\cG_2)$ be the {path-holonomy} $\cG_2$-bisubmersion for $\cB_2$ 
(not necessarily minimal) associated to the  set of local generators $\{F_*(\balpha_1),\dots,F_*(\balpha_n)\}$  of $\cB_2$. In particular, $W$ is  an open subset of $\RR^n \times M_2$. The submersion $(Id,f)\colon \RR^n \times M_1 \to \RR^n \times M_2$ restricts to map $p\colon U\to W$ which by eq. \eqref{eq:lambdax} makes this diagram commute:
\begin{equation*}
\xymatrix{
U  \ar[r]^p \ar[rd]_{\varphi_2}      &   W \ar[d]^{\varphi_{W}}  \\
 & \cG_2 }
\end{equation*}
By Lemma \ref{lem:UV}, we conclude that 
$(U,\varphi_2,\cG_2)$ is a 
$\cG_2$-bisubmersion for $\cB_2$. 
\end{proof}

 {Notice that if if $(U,\varphi_1,\cG_1)$ and $(U',\varphi_1',\cG_1)$ are $\cG_1$-bisubmersion for $\cB_1$ as in Lemma \ref{lem:U12}, and $k$ is a (locally defined) morphism  between them, then $k$ is is also a   morphism of bisubmersions between the corresponding $\cG_2$-bisubmersion for $\cB_2$. Applying this fact in the proof of Prop. \ref{prop:morph} implies that the morphism $\Xi$ is canonical in a neighborhood of the identity section of $H^{\cG_1}(\cB_1)$, and by  source-connectedness on the whole of  $H^{\cG_1}(\cB_1)$.}

\begin{proof}[Proof of Prop.  \ref{prop:morph}]

Take a family $\cS$ of {path-holonomy} $\cG_1$-bisubmersions
 for $\cB_1$, constructed out of   
 minimal sets of local, $F$-projectable  generators of $\cB_1$, and so that $\cup_{U\in 
\cS} \bs_{U}(U)=M_1$. It exists since $\cB_1$ satisfies property $\eqref{star}$. Denote by $\cU_1$ the atlas
(see {definitions  \ref{def:pathholatlas} and} \ref{def:atlas}) generated by the elements of  $\cS$, viewed as $\cG_1$-bisubmersions 
for $\cB_1$. In other words, $\cU_1$ is constructed from elements of {$\cS$}, taking   their  finite compositions as  $\cG_1$-bisubmersions 
for $\cB_1$ {(we do not need to take inverse bisubmersions by Remark \ref{rem:invphbisub}.)} 
Similarly, denote by $\cU_2$ the atlas generated by the elements of  $\cS$, viewed as $\cG_2$-bisubmersions 
for $\cB_2$ as in Lemma \ref{lem:U12} {(again\footnote{Indeed, since $F$ is a groupoid morphism, an isomorphism of bisubmersions between $(U,\varphi,\cG_1)\in \cS$ and its inverse $(U,i_{\cG_1}\circ\varphi,\cG_1)$ provides an 
isomorphism of $\cG_2$-bisubmersions between $(U,F\circ \varphi,\cG_2)$ and its inverse $(U,i_{\cG_2}\circ F \circ\varphi,\cG_2)$.}
 we do not need to take inverses)}. 
 
\underline{Claim:} \emph{{There is a canonical injective map} $\iota$ that makes the following diagram commute:
 \begin{equation} \label{diag:comp}
 \xymatrix{
\coprod_{U_1\in \cU_1}U_1\ar[d] \ar[r]^{\iota} & \coprod_{U_2\in \cU_2}U_2\ar[d]\\
 {\cG_1} \ar[r]^{F}&\cG_2   }
\end{equation}
}

 By Lemma \ref{lem:U12}, for every ${(U,\varphi_1,\cG_1)}\in \cS$,
 we have that
  $(U,\varphi_2:=F\circ \varphi_1,\cG_2)$ is a  $\cG_2$-bisubmersion  for $\cB_2$. {We take $\iota|_U$ to be simply the identity.}

Let $U,V\in \cS$. Then $U\circ_1 V:=U\times_{(\bs_U)_1,(\bt_V)_1}V$,
their composition as ${\cG_1}$-bisubmersions for $\cB_1$, is contained in 
$U\circ_2 V:=U\times_{(\bs_U)_2,(\bt_V)_2}V$,
their composition as $\cG_2$-bisubmersions for $\cB_2$. {(Notice that the former is 
 a fiber product over $M_1$, the latter a fiber product over $M_2$.)}
 We define 
 $\iota|_{U\circ_1 V}$ to be this inclusion.
The following diagram commutes, since $F$ is a morphism of groupoids:
 \begin{equation*} 
 \xymatrix{
U\circ_1 V\ar_{
 (\varphi_U)_1\cdot (\varphi_V)_1}[d]\ar[d]\ar[r] & U\circ_2 V\ar^{
 (\varphi_U)_2\cdot (\varphi_V)_2}[d]\\
 {\cG_1} \ar[r]^{F}&\cG_2   }
\end{equation*}
{The same holds for the composition of any finite number of bisubmersions in $\cS$.}
\hfill$\bigtriangleup$
 
\underline{Claim} \emph{The   map $\iota$ descends to a map 
\begin{equation}\label{eq:maponquot}
H^{\cG_1}(\cB_1)=\coprod_{U_1\in \cU_1}U_1/\sim_1 \;\to\; \coprod_{U_2\in \cU_2}U_2/\sim_2
\end{equation} 
where $\sim_1$ (respectively $\sim_2$) denotes the equivalence relation 
{of $G_1$-bisubmersions (respectively $G_2$-bisubmersions) as in Appendix \ref{sec:atlases}.}
It is clearly a morphism of topological groupoids.}

It is clear that if $U,V\in \cS$ and $f \colon U \to V$ is a morphism for $\cG_1$-bisubmersions, then it is also a morphism for $\cG_2$-bisubmersions. Hence, for all $u\in U$ and $v\in V$,  $u\sim_1 v$ implies that $u\sim_2 v$.
\begin{equation*}
\xymatrix{
U  \ar[rd]_{(\varphi_U)_1}     \ar[rr]^f & & V  \ar[ld]^{(\varphi_V)_1}   \\
&{\cG_1} \ar[d]^{F}&\\
&\cG_2 & }
\end{equation*}

The  equivalences between points in arbitrary elements of $\cU_1$ are more delicate. We first describe how to  construct certain bisections. Let $U_1\in \cU_1$, so that $U_1=U^1\circ_1\dots\circ_1U^k$ for $U^1,\dots,U^k\in \cS$. We denote by $U_2$ the ``corresponding'' element of $\cU_2$, \ie, $U_2:=U^1\circ_2\dots\circ_2U^k$. As remarked {in the previous claim}, we have a (usually strict) inclusion $\iota\colon U_1\hookrightarrow U_2$.
Fix $u\in U_1$. We construct a bisection for $U_2$ passing through $\iota(u)$, as follows.  Let $\mathbf{b}$ be a $dim({M_2})$-dimensional  submanifold of $U_1$ so that  $T_u\mathbf{b}$  is transverse to both $(\bs_U)^{-1}_1(\ker(f_*))$ and  $(\bt_U)^{-1}_1(\ker(f_*))$.
Notice that $(\bs_U)_1(\mathbf{b})$ and $(\bt_U)_1(\mathbf{b})$ are both submanifolds of $M_1$ transverse to the fibers of $f \colon M_1\to M_2$, and that the map $$(\bs_U)_1(\mathbf{b})\to (\bt_U)_1(\mathbf{b}),$$ obtained $(\bs_U)_1$-lifting to points of $\mathbf{b}$ and applying 
$(\bt_U)_1$, is a diffeomorphism. Together with the commutativity of \eqref{diag:comp}, this implies that $\iota(\mathbf{b})$ is a bisection for the $\cG_2$-bisubmersion $U_2$.

Now let $U_1,V_1$ be arbitrary elements of $\cU_1$, and $U_2, V_2$ the ``corresponding'' elements of $\cU_2$. Let $u\in U_1$ and $v\in V_1$ with $u\sim_1 v$. By definition, this means that there is a locally defined morphism $f\colon U_1\to V_1$ of bisubmersions over $\cG_1$, mapping $u$ to $v$ (see \S\ref{sub:con}). Choosing a submanifold $\mathbf{b}\subset U_1$ through $u$ as above, we obtain  bisections $\iota(\mathbf{b})$  for the bisubmersion $U_2$, and $\iota(f(\mathbf{b}))$
for the bisubmersion $V_2$. Notice that these bisections 
pass  through $\iota(u)$ and $\iota(v)$ respectively.
Both bisections carry the same bisection of $\cG_2$, by the commutativity of \eqref{diag:comp}. By Cor. \ref{cor:crucial} b) we obtain 
$\iota(u)\sim_2 \iota(v)$. \hfill$\bigtriangleup$

\underline{Claim} \emph{ {We have $\coprod_{U_2\in \cU_2}U_2/\sim_2
=H^{\cG_2}(\cB_2)$, hence the morphism  in \eqref{eq:maponquot} reads $H^{\cG_1}(\cB_1)\to H^{\cG_2}(\cB_2)$.}}
Let $(U,\varphi_1,\cG_1) \in \cS$, constructed out of 
 a minimal set of local generators    $\{\balpha_1,\dots,\balpha_n\}$ of $\cB_1$. Consider  
  $(U,\varphi_2{:=F\circ \varphi_1},\cG_2)$, a 
$\cG_2$-bisubmersion for $\cB_2$. Further consider  $(W,\varphi_W,\cG_2)$, the $\cG_2$-bisubmersion for $\cB_2$ 
(not necessarily minimal) associated to the  set of local generators $\{F_*(\balpha_1),\dots,F_*(\balpha_n)\}$  of $\cB_2$. {In the proof of Lemma \ref{lem:U12} we saw that there is a  morphism of bisubmersions  $p\colon U\to W$, so} 
the bisubmersion $(U,\varphi_2,\cG_2)$ is adapted (see definition \ref{def:atlas}) to the bisubmersion $(W,\varphi_W,\cG_2)$. Hence the atlas $\cU_2$, {which is generated by all $(U,\varphi_2,\cG_2)$'s    as above,} is adapted to the atlas  generated by the corresponding $W$'s, which we denote by $\cW$. {In turn, by Rem. \ref{rem:HHu},
$\cW$ is equivalent to a path holonomy atlas.
 The latter is adapted to $\cU_2$ by Prop. \ref{prop:pathholadapted}. In conclusion, $\cU_2$ is equivalent to a path holonomy atlas, so by Proposition \ref{prop:rem33}  there is a canonical  isomorphism of topological groupoids $\coprod_{U_2\in \cU_2}U_2/\sim_2
\cong H^{\cG_2}(\cB_2)$.}
    \hfill$\bigtriangleup$

\underline{Claim} \emph{The map $H^{\cG_1}(\cB_1)\to H^{\cG_2}(\cB_2)$ constructed above is surjective.} 
 
The inclusion $\iota$ in diagram \eqref{diag:comp} is not surjective. We have to show that any point in the codomain is equivalent, under $\sim_2$, to a point in the image of $\iota$. {We do so only for points in the codomain which lie in the product of two bisubmersions, as the general case is similar.}   
Consider
  path holonomy $G_1$-bisubmersions $U$ and $V$ {lying in $\cS$, i.e.} constructed out of   
 minimal sets of local $F$-projectable  generators $\{a_i\}_{i\le I}$ and $\{b_k\}_{k\le K}$ of $\cB_1$. Consider points
 $(\lambda,y)\in U\subset \RR^I \times M_1$ and $(\eta,x)\in V\subset \RR^K \times M_1$   such that their images under\footnote{Recall that $(\bs_U)_2$ denotes the composition of $\varphi_2\colon U\to \cG_2$ and the source map of $\cG_2$.}
 $(\bs_U)_2$ and $(\bt_V)_2$ respectively coincide.
{We want} to show that  $(\lambda,y)\circ_2 (\eta,x)\in  U\circ_2 V$ is equivalent under $\sim_2$ to an element in the image of $\iota$. The difficulty is that $(\lambda,y)\circ_1 (\eta,x)$ might not be well-defined.

{
  \begin{equation*} 
 \xymatrix{
 U\ar_{
 \varphi_1}[d]\ar[rd]^{\varphi_2}&\\
 {\cG_1} \ar[r]^{F}\ar[d]^{\bs_{\cG_1}}&\cG_2  \ar[d]^{\bs_{\cG_2}} \\
 M_1 \ar[r]^{f}&M_2
 }
 \end{equation*}
 }

Define $y':=(\bt_V)_1(\eta,x)\in M_1$. Notice that $f(y')=(\bt_V)_2(\eta,x)=f(y)$.
Choose a 
 path holonomy   bisubmersion $U'$, constructed out of   a
 minimal set of local, $F$-projectable  generators  $\{a'_j\}_{j\le J}$ of $\cB_1$ defined nearby $y'$. There is $\lambda'\in \RR^J$ such that 
 \begin{equation}\label{eq:ll'}
(\lambda,y)\sim_2 (\lambda',y').
\end{equation}
To see this, first use the fact that the $\{F_*a'_j\}_{j\le J}$ generate $\cB_2$ nearby $f(y')$ to write $F_*a_i=\sum_j c_i^j(F_*a'_j)$ where $c_i^j\in C^{\infty}(M_2)$. Then check using {Eq. \eqref{eq:lambdax}} that $$U\to U', (\gamma_1,\cdots,\gamma_I; z)\mapsto  
 (\sum c_i^1 \gamma_i ,\cdots, \sum c_i^J\gamma_i ; \psi(z))
 $$
 is a morphism of $G_2$-bisubmersions, which maps  
$(\lambda,y)$  to a point of the form $(\lambda',y')$ {for some $\lambda'\in \RR^J$}. Here
 $\psi$ is any  local diffeomorphism  of $M_1$ with $\psi(y)=y'$ preserving the $f$-fibers, {which exists since $f$ is a submersion.}

Notice that $(\lambda',y')\circ_1 (\eta,x)\in U'\circ_1 V$ is well-defined. Thanks to eq. \eqref{eq:ll'} we can hence write
$$(\lambda,y)\circ_2 (\eta,x)\sim_2 (\lambda',y')\circ_2 (\eta,x)=\iota((\lambda',y')\circ_1 (\eta,x)),$$
proving the claim.  \hfill$\bigtriangleup$

{To finish the proof, notice that diagram \eqref{diag:commtoG}
 commutes since diagram \eqref{diag:comp} does.}
\end{proof}

\subsection*{{Arbitrary morphisms}}

The next theorem, {which specializes to Thm. \ref{thm:morph} for morphisms covering the identity}, extends Prop. \ref{prop:morph}.

\begin{thm}\label{thm:morphsub}
Let $F\colon \cG_1\to \cG_2$ be a morphism of Lie groupoids, covering a   surjective submersion $f \colon M_1\to M_2$ between their spaces of identity elements. Let $\cB_1$ be a singular subalgebroid of $A_1:=Lie(\cG_1)$, and assume that it satisfies  condition $\eqref{star}$ above. 
Let  $\cB_2$ be a singular subalgebroid of $A_2:=Lie(\cG_2)$, such that
$F_*(\cB_1)\subset \cB_2$.
Then there is a canonical  morphism of topological groupoids 
$$\Xi \colon H^{\cG_1}(\cB_1)\to H^{\cG_2}(\cB_2)$$ {over $f$}  making the following diagram commute:
\begin{equation}  \xymatrix{
H^{\cG_1}( {\cB}_1)  \ar[d]^{\Phi_1}   \ar@{-->}[r]^{\Xi} &H^{\cG_2}(\cB_2)  \ar[d]^{\Phi_2}     \\
\cG_1 \ar[r]^{F} &   \cG_2 }
\end{equation}
\end{thm}
\begin{proof}
Compose the canonical morphism   $\Xi \colon H^{\cG_1}(\cB_1)\to H^{\cG_2}(F_*(\cB_1))$ given by Prop. \ref{prop:morph}  with the canonical morphism  $H^{\cG_2}(F_*(\cB_1))\to H^{\cG_2}(\cB_2)$   given by Lemma \ref{prop:morph0}.
\end{proof}

 \begin{remark}\label{rem:Moritamor}
{One can wonder whether an analogue of Thm. \ref{thm:morphsub} holds if one replaces the Lie groupoid morphism $F\colon G_1\to G_2$ by a generalized morphism.
A useful characterization of generalized morphisms is the one \cite[Remark 4.5.3]{MatiasME} as a diagram of Lie groupoid morphisms
$$G_1\overset{\phi}{\leftarrow} K \overset{\psi}{\rightarrow} G_2,$$
where $\phi$ is a strong equivalence (hence the base map $f\colon P\to M_1$ is a surjective submersion, $K\cong f^{-1}G_1:=P\times_M G_1\times_M P$ is the pullback groupoid, and $\phi$ is the second projection).
We have $Lie(K)\cong f^{-1}Lie(G_1)$ (the pullback Lie algebroid). Pulling back the singular subalgebroid  $\mathcal{B}_1$ 
 one obtains a singular subalgebroid $f^{-1}\mathcal{B}_1$ of $Lie(K)$.  Indeed one can define the pullback of singular subalgebroids by fiber-wise surjective Lie algebroid morphisms just as for singular foliations  \cite[Def. 1.9]{AndrSk}  
 (notice that \cite[Prop. 1.10 b)]{AndrSk}   extends straightforwardly to singular subalgebroids). 
}

{
 Under the  assumption that $f\colon P\to M_1$ has connected fibers, we expect that $f^{-1}(H^{G_1}(\mathcal{B}_1))\cong H^K(f^{-1}\mathcal{B}_1)$:
this is true for   singular foliations \cite[Thm. 3.21]{MEsingfol}, and we 
expect the proof given there to extend to singular subalgebroids.
If so,  one obtains a strong equivalence of topological groupoids  $\Xi_1\colon  H^K(f^{-1}\mathcal{B}_1)\to H^{G_1}(\mathcal{B}_1)$.
}
{
If the assumptions of Theorem \ref{thm:morphsub} are satisfied for the Lie groupoid morphism $\psi\colon K\to G_2$ and the singular subalgebroids $f^{-1}\mathcal{B}_1$ and $\mathcal{B}_2$, we can then apply that theorem 
 to obtain a morphism of topological groupoids  $\Xi_2\colon H^K(\phi^{-1}\mathcal{B}_1)\to H^{G_2}(\mathcal{B}_2)$. Altogether, this would yield a  generalized morphism from
  $H^{\cG_1}( {\cB}_1)$ to $H^{\cG_2}( {\cB}_2)$.
  }
\end{remark}
\subsection*{Examples}

The following are applications of Prop. \ref{prop:morph}.

\begin{ex}[The singular foliation $\rar{\cB}$ on $\cG$] Let $\cG\rightrightarrows M$ be a Lie groupoid, and $\cB$ a singular subalgebroid of $A:=Lie(\cG)$.
Consider the Lie groupoid $\cG \times_{\bs,\bs} \cG\rightrightarrows \cG$ 
associated to the submersion $\bs \colon \cG\to M$. (It is a subgroupoid of the pair groupoid associated to $\cG$, where we regard $G$ as a mere manifold). Consider the morphism of Lie groupoids
$$F \colon \cG \times_{\bs,\bs} \cG \to \cG,\;\; (g,h)\mapsto gh^{-1}$$ over $\bt \colon \cG\to M$. The corresponding Lie algebroid morphism  ${F}_*\colon \ker(\bs_*)\to  \ker(\bs_*)|_M=A$ {sends $v\in \ker(\bs_*)_h$ to $(R_{h^{-1}})_*v$. From this we see that the  singular foliation $\rar{\cB}$ satisfies condition \eqref{star} and }
$F_*(\rar{\cB})={\cB}$. 

Notice that, since $\cG \times_{\bs,\bs} \cG$ is a subgroupoid of the pair groupoid  $\cG \times \cG$, by proposition \ref{prop:morphG}
we have $H^{\cG \times_{\bs,\bs} \cG}(\rar{\cB})\cong H^{\cG \times \cG}(\rar{\cB})$, that is:
the holonomy groupoid of $\rar{\cB}$ constructed using $\cG \times_{\bs,\bs} \cG$ agrees with the one constructed using the pair groupoid $\cG \times \cG$. Prop. \ref{prop:morph} implies the existence of a canonical surjective  morphism of topological groupoids $$H^{\cG \times \cG}(\rar{\cB})\to H^{\cG}({\cB}).$$
Notice that the domain of the above morphism is simply
 the holonomy groupoid of the singular foliation $\rar{\cB}$, in the sense of \cite{AndrSk}.
In \cite{AZ4} we 
use this morphism to study the geometric properties of the holonomy groupoid $H^{\cG}({\cB})$.
\end{ex}

\begin{ex}[Singular foliations]
Let $f\colon M_1\to M_2$ be a surjective submersion. Let $\cF_1$ be a singular foliation on $M_1$, such that 
${\widehat{\cF_1}}^{proj}:=\{X\in {\widehat{\cF_1}}:X \text{ is $f$-projectable to a vector field on $M_2$}\}$ generates $\cF_1$ as a $C_c^{\infty}(M_1)$-module. 
  Then $f_*\cF_1$ is a singular foliation on $M_2$, by Lemma \ref{lem:FB2}.
Let $$(f,f)\colon M_1\times M_1\to M_2\times M_2$$ be 
the induced map on  pair groupoids. Prop. \ref{prop:morph} implies the existence of a canonical surjective  morphism of topological groupoids $$H(\cF_1)\to H(f_*\cF_1)$$ between the holonomy groupoids of the two singular foliations.  \end{ex}

{We remark that, under the additional assumptions that $f$ has connected fibers and that $\cF_1$ contains $\Gamma_c(\ker(f_*))$, the foliation $\cF_1$ agrees with the pullback foliation $f^{-1}(f_*\cF_1)$, by \cite[Lemma 3.2]{AZ1}. It is shown in \cite{MEsingfol} that $H(\cF_1)$ then agrees with the pullback of the groupoid $H(f_*\cF_1)$ via $f$, and the above morphism is the canonical projection of the pullback groupoid.}

 \bibliographystyle{habbrv} 

\end{document}